\documentclass{article}
\usepackage{graphicx} 
\usepackage{dsfont}
\usepackage{hyperref}
\usepackage{enumitem}
\usepackage{bbm}
\usepackage{caption}
\usepackage{authblk}
\hypersetup{
colorlinks = true, 
urlcolor = blue, 
linkcolor = red, 
citecolor = blue 
}
\usepackage{amsmath} 
\usepackage{amssymb}
\usepackage{amsthm}
\usepackage{imakeidx}
\makeindex[title=Notation Index, intoc,columns=1]

\indexsetup{othercode=\noindent\par}

\usepackage{thmtools}
\usepackage{ifthen}

\declaretheorem[numberwithin=section]{theorem}

\usepackage{stmaryrd}
\usepackage{mathtools}
\usepackage{tikz}
\usetikzlibrary{decorations.pathreplacing}
\usetikzlibrary{patterns}

\newtheorem{lemma}[theorem]{Lemma}
\newtheorem{question}[theorem]{Question}
\newtheorem{definition}[theorem]{Definition}
\newtheorem{example}[theorem]{Example}

\newtheorem{notation}[theorem]{Notation}
\newtheorem{proposition}[theorem]{Proposition}
\newtheorem{corollary}[theorem]{Corollary}
\newtheorem{remark}[theorem]{Remark}

\newcommand{\Z}{\mathbb Z}
\newcommand{\N}{\mathbb N}

\newcommand{\notationidx}[2]{#1\index{#1@#1 --- #2|hyperpage}}

\title{Undecidability of the block gluing classes\\ of homshifts}

\author[1]{\normalsize Nishant Chandgotia} 

\affil[1]{\small
	Tata Institute of Fundamental Research - Centre for Applicable Mathematics
    \url{nishant.chandgotia@gmail.com}
	}

\author[2]{\normalsize Silvère Gangloff}

\affil[2]{
National Supercomputing Centre IT4Innovations, University of Ostrava,
	IRAFM,
	30. dubna 22, 70103 Ostrava,
	Czech Republic
    \url{silvere.gangloff@osu.cz} \url{piotr.oprocha@osu.cz}
	}

\author[3]{\normalsize Benjamin Hellouin de Menibus}

\affil[3]{
Université Paris-Saclay, CNRS, Laboratoire Interdisciplinaire des Sciences du Numérique, Orsay, France
\url{hellouin@lisn.fr}}

\author[2]{Piotr Oprocha}

\date{\today}

\begin{document}

\maketitle
\begin{abstract}
A homshift is a $d$-dimensional shift of finite type which arises as the space of graph homomorphisms from the grid graph $\Z^d$ to a finite connected undirected graph $G$. 
While shifts of finite type are known to be mired by the swamp of undecidability, homshifts seem to be better behaved and there was hope that all the properties of homshifts are decidable.
In this paper we build on the work by Gangloff, Hellouin de Menibus and Oprocha \cite{GHO23} to show that finer mixing properties are undecidable for reasons completely different than the ones used to prove undecidability for general multidimensional shifts of finite type. Inspired by the work of Gao, Jackson, Krohne and Seward \cite{gao2018continuous} and elementary algebraic topology, we interpret the square cover introduced by Gangloff, Hellouin de Menibus and Oprocha topologically. Using this interpretation, we prove that it is undecidable whether a homshift is $\Theta(n)$-block gluing or not, by relating this problem to the one of finiteness for finitely presented groups.
\end{abstract}

 \bigskip 

 
\tableofcontents

\section{Introduction}
A homshift is a shift of finite type which is defined by means of a finite undirected connected graph $G$: vertices of $G$ are the symbols and edges of $G$ indicate which symbols are allowed to be adjacent.
Many important statistical physics models are homshifts: the hard square model, proper $k$-colorings, the iceberg model, the beach model, Lipschitz functions and the clock model. For instance the set of three-colorings of $\mathbb{Z}^d$ is the homshift associated with the triangle graph while the hard square model is the homshift associated with the graph on two vertices with an {edge} between them and a self-loop on one of them. 

Undecidability phenomena form an overarching theme over the study of multidimensional shifts of finite type, starting with the Domino problem. A large number of such problems, of combinatorial or dynamical nature, have been studied and proved to be undecidable: for example, there are shifts such that it is not decidable whether a pattern can be extended to a full configuration, or whose entropy is not a computable number. There are structural results that imply that, under mild technical conditions, every nontrivial problem regarding shifts of finite type is undecidable \cite{Delvenne, Carrasco}.

These undecidability results rely one way or another on the ability of shifts of finite type to simulate universal computation, usually a universal Turing machine. Such constructions cannot be done straightforwardly in a homshift because the simulated machine uses two distinct directions for time and space (memory), while homshifts are invariant under rotations and symmetries. As a matter of fact, the same problems become decidable for homshifts:

\begin{itemize}
\item a homshift is non-empty as long as its graph has at least one edge, which is trivially decidable;
\item a pattern can be extended
to a full configuration if and only if it can be extended to a rectangle
\cite[Lemma 3.2]{GHO23}, which means that the domino and extension
problems are decidable.
\item the topological entropy is computable from a description of the homshift (a particular case of \cite[Theorem 3.1]{Friedland}).
\end{itemize}

Furthermore, transitivity-type properties are tractable for homshifts. For instance, a homshift on a graph $G$ is transitive if and only if $G$ is connected. It is mixing if and only if $G$ is connected and not bipartite \cite{MR3743365}. Homshifts are also almost Borel universal \cite{zbMATH07433654}.

Nevertheless, understanding finer aspects of these shifts, such as for instance measures of maximal entropy and mixing properties like strong irreducibility, has been a challenge. Even for simple cases, such as for instance when $G$ is a complete graph, classifying the corresponding homshifts according to their mixing properties is a non-trivial task \cite{zbMATH07359187} and while several papers have been written for general homshifts, we are not anywhere near answering the question completely \cite{briceno2017strong, MR3743365, zbMATH07359187, GHO23}. 

This may be due to undecidability phenomena in the class of homshifts. As a matter of fact, in a recent article, Gao, Jackson, Krohne and Seward \cite{gao2018continuous} showed that it is undecidable whether or not there is a continuous equivariant factor from the free part of the full shift to $X^2_G$. That work took inspiration from algebraic topology and defined the so-called reduced homotopy group of the graph $G$ (which we call the \emph{square group} in the present paper) and provided necessary conditions and sufficient conditions for such a factor to exist. The groupoid version of the square group appeared earlier in the context of topological graph theory \cite{zbMATH06656178}. The square group is however distinct though similar to the A-homotopy group which considers graphs where all the vertices have self-loops \cite{zbMATH02207449}.

 We are mainly interested in \emph{strong irreducibility} - meaning there exists a constant $K$ such that any two admissible patterns separated by a distance at least $K$ together form an admissible pattern -  as this property has many important consequences. For instance, any shift which is strongly irreducible is entropy minimal \cite{quas2000shifts} and any of its stationary Gibbs measures for a summable interaction is an equilibrium state \cite{dobrushin1968gibbsian}. The following question is the main motivation for our research:

\begin{restatable}{question}{mainquestion}
\label{question: main question}
Is it decidable whether the homshift $X^d_G$ is strongly irreducible?
\end{restatable}

We do not answer this question, but prove that a similar property is undecidable for homshifts, hinting that this problem is also undecidable. Quantified block gluing is a mixing property which was studied by Gangloff, Hellouin de Menibus and Oprocha \cite{GHO23} for homshifts, where the authors proved that a two-dimensional homshift is either $\Theta(n)$-phased block gluing or $O(\log(n))$-phased block gluing. This dichotomy is based on the finiteness or infiniteness of the square cover, a particular quotient of the universal cover by the squares in the graph. That work used an earlier result \cite{MR3743365} that the block gluing rate corresponds to the growth of the diameter of the graph $G_n$ whose vertices are the walks of length $n$ on $G$ and in which two walks are adjacent when they are pointwise adjacent. The result in \cite{GHO23} relied on proving that the growth rate of the diameter of $G_n$ exhibits the following dichotomy: it is $O(\log(n))$ if and only if the square cover is finite and it is $\Theta(n)$ if and only if the square cover is infinite. 

The strips of width $n$ in $X^2_G$ correspond to a one-dimensional homshift associated to the graph $G_n$, which can be seen as an approximation of $X^2_G$. The asymptotic behaviour of the diameter of $G_n$, that is, the block-gluing rate, is related to the spectrum of the corresponding transfer matrix and how fast one can approximate the entropy of $X^2_G$. For general two-dimensional shifts of finite type, an upper bound on the speed of computability of $O(1)$-block gluing shifts is known \cite{PavlovSchraudner}.

In this article, we prove that the dichotomy of the growth rate of the diameter of $G_n$ is undecidable, answering a question left open in \cite{GHO23}. This result is based on formally relating the square cover with the square group via a `Galois correspondence' of graph covers. In particular the square group is finite if and only if the square cover is finite. 

Furthermore, every finitely presented group is the square group of a finite connected undirected graph. We obtain the following as a consequence of the fact that it is not possible to decide, from its presentation, whether a finitely presented group is finite or infinite:

\begin{restatable}{theorem}{main}
\label{thm:main}
The classes of $\Theta(n)$-block gluing and $O(\log(n))$-block gluing two-dimensional homshifts are computably inseparable. That is, it is not possible to algorithmically decide if a two-dimensional homshift is $\Theta(n)$-block gluing or $O(\log(n))$-block gluing.
\end{restatable}

For general SFTs, \cite{Paviet} proved that every finitely presented group is the projective fundamental group, introduced by Geller and Propp \cite{geller-propp}, of a two-dimensional SFT. We do not know if our results are related - in particular if the square group is related to the projective fundamental group. However, their construction breaks symmetries and thus might not be easily adapted to homshifts. Furthermore, the fact that our result is not the product of a systematic undecidability phenomenon, as in SFT, suggests that the source of undecidability is qualitatively of a different, more algebraic, nature. 

Most of the constructions presented here are paralleled in algebraic topology. However there are a few important subtleties to note. Most of the constructions (whether in standard algebraic topology books or \cite{zbMATH06656178, gao2018continuous,zbMATH07969288}) don't take self-loops into account in a manner suited for dealing with graph homomorphisms, but many important examples of homshifts are defined with graphs having self-loops. This requires reworking of some standard ideas (Remark \ref{remark: Self-loop problem}). Finally we construct graphs whose square group is any given finitely presented group. While the corresponding construction is well-known in algebraic topology (see \cite{rotman2013introduction} for instance) and has been adapted in \cite{zbMATH06656178, gao2018continuous} to the graph setting, we wanted to provide a self-contained construction tailor-made for our purpose, namely, towards proving that strong irreducibility is undecidable for homshifts. Our construction is fairly general and should be adaptable to exhibit all sorts of properties of homshifts via minor modifications of the parameters - possibly generalizations beyond $\Z^d$ as well. One of the key technical issues in our construction is taking special care to avoid creating extraneous squares which can possibly modify the square group in an irregular way. Finally, we have written the paper for an audience who might not be well-versed in algebraic topology and have made attempts to sketch proofs even for somewhat standard ideas whenever possible.

The remainder of the text is structured as follows: Section \ref{section.square.group.def} contains some elementary background and notations for graphs, shifts and finitely presented groups; in Section \ref{section.background.algebraic.topology}, we provide a detailed and self-contained account of algebraic topology on graphs, introducing in particular the fundamental group and universal covers; in Section \ref{section.square.quotients}, we define the notion of square cover and other similar notions which can be found in \cite{GHO23} and additional ones (in particular the notion of square group); we turn in Section \ref{section.undecidability} to a proof of Theorem \ref{thm:main}, which relies on a rigorous and fully detailed proof that every finitely presented group is the square group of a graph. We leave some questions open in Section \ref{section.open.questions}.

\section{\label{section.square.group.def} Elementary background and notations}

In this section, we define background notions related to graph theory (Section~\ref{section.background.graphs}), 
shifts (Section~\ref{section.background.shift.spaces}), finitely presented groups (Section~\ref{section.background.group.presentations}). 

\subsection{Graphs\label{section.background.graphs}}

For every graph $G$, we denote by \notationidx{$V_G$}{vertex set of the graph $G$} the set of its vertices and \notationidx{$E_G$}{edge set of the graph $G$} the set of its edges, where an edge is a tuple $(u,v)$ with $u,v \in V_G$. We say that the graph is \textbf{undirected} when for all $u,v$ vertices in $V_G$, if $(u,v) \in E_G$ then $(v,u) \in E_G$. We say that it is \textit{finite} when both $V_G$ and $E_G$ are finite. 

In the remainder of this article, graphs denoted by $G$ are assumed to be finite, undirected and connected. Thus whenever we remove or add an edge $(u,v)$ in the graph $G$ we also remove or add its reverse $(v,u)$ even if we don't explicitly mention it.\bigskip 

A \textbf{graph homomorphism} from an undirected graph $G'$ to another undirected graph $G'$ is a map  \notationidx{$\phi : V_G \rightarrow V_{G'}$}{graph homomorphism from $G$ to $G'$} such that for all $u,v \in V_G$, if $(u,v) \in E_G$ then $(\phi(u),\phi(v)) \in E_{G'}$. In order to simplify the notation, we will often write $v\in G$ instead of $v \in V_G$ and  \notationidx{$\phi : G \to G'$}{graph homomorphism from $G$ to $G'$ (simplified notation)} instead of $\phi : V_G \to V_{G'}$ for a homomorphism. For all $a\in V_G$, we call \textbf{neighborhood} of $a$ the set \notationidx{$N_G(a)$}{neighborhood of $a$ in $G$} $\coloneqq \{b\in V_G~:~(a,b) \in E_G\}$.
A \textbf{walk} on $G$ is a finite word $p = p_0 \cdots p_n$ of vertices of $G$ such that for all $i < n$, $(p_i,p_{i+1})$ is an edge of $G$. The integer $n$ is called the \textbf{length} of $p$ and is denoted by \notationidx{$l(p)$}{length of a walk $p$}. Such a walk $p$ is called a \textbf{cycle} 
when $p_{l(p)} = p_0$. Such a cycle is said to be \textbf{simple} when for all $i<j$, if $p_i = p_j$ then $i=0$ and $j=l(p)$. 

A cycle of length two is called a \textbf{backtrack}. A walk which does not contain any backtrack is called non-backtracking. A \textbf{square} is a non-backtracking cycle of length four. The empty walk {(that we also see as empty cycle)} is denoted by  \notationidx{$\varepsilon$}{the empty cycle}. We say that a graph $G$ is connected when for all $u,v \in G$, there exists a walk $p$ on $G$ such that $p_0 = u$ and $p_{l(p)} = v$.

\begin{notation}\label{notation:odot}
    For any two walks $p,q$ such that $p_{l(p)} = q_0$, we denote by \notationidx{$p \odot q$}{Notation for the `concatenation` operation on walks} the walk $p_0 \ldots p_{l(p)} q_1 \ldots q_{l(q)}$ and by $p^{-1}$ the reverse walk $p_{l(p)} \ldots p_0$. 
\end{notation}

\begin{notation}
    We denote by $\varphi$ the function such that to any walk $p$ associate the walk $\varphi(p)$ obtained by replacing successively all backtracks $aba$ by $a$. It is not difficult to check that the order of removal does not change the resulting non-backtracking walk, so $\varphi$ is well-defined. Furthermore, for all walks $p,q$ such that $p_{l(p)} = q_0$, we set \notationidx{$p \star q$}{Notation for the word obtained by removing backtracks from the walk concatenation of $p$ with $q$} $\coloneqq \varphi(p \odot q)$. The operation $\star$ is associative. 
\end{notation}
 
\begin{remark}
Removing backtracks from two cycles that are equal up to circular shift may yield different cycles. Consider for instance consider the graph in Figure~\ref{fig.cycle.jumble}: we have that $\varphi(abcdba)=abcdba$, but $\varphi(bcdbab) = bcdb$.
\end{remark}

\begin{figure}[!ht]
    \captionsetup{labelsep=none}
    \centering
    \begin{tikzpicture}[scale=0.4]
        \draw (0,0) -- (-1.5,-2) -- (1.5,-2) -- (0,0) -- (0,1.5);
        \draw[fill=gray!40] (0,0) circle (3pt);
        \draw[fill=gray!40] (0,1.5) circle (3pt);
        \draw[fill=gray!40] (-1.5,-2) circle (3pt);
        \draw[fill=gray!40] (1.5,-2) circle (3pt);
        \node[scale=0.9] at (-2,-2.5) {c};
        \node[scale=0.9] at (2,-2.5) {d};
        \node[scale=0.9] at (0.5,0.15) {b};
        \node[scale=0.9] at (0.5,1.5) {a};
    \end{tikzpicture}
    \caption{\label{fig.cycle.jumble}}
\end{figure}

\subsection{Shift spaces \label{section.background.shift.spaces}}

For every finite set \notationidx{$\mathcal{A}$}{standard notation for the alphabet of a shift} and integer $d \ge 1$, we call the \textit{full shift} on alphabet $\mathcal{A}$ the set $\mathcal{A}^{\mathbb{Z}^d}$ endowed with the infinite product of discrete topologies. We call its elements \textbf{configurations}. 
The \textbf{shift action} on $\mathcal{A}^{\mathbb{Z}^d}$ is the group action \notationidx{$\sigma$}{the shift action} : $\mathbb{Z}^d \times \mathcal{A}^{\mathbb{Z}^d} \rightarrow \mathcal{A}^{\mathbb{Z}^d}$, defined by 
$\sigma(\boldsymbol{u},x) = (x_{\boldsymbol{u}+\boldsymbol{v}})_{\boldsymbol{v} \in \mathbb{Z}^d}$ for all $\boldsymbol{u} \in \mathbb{Z}^d$ and $x \in \mathcal{A}^{\mathbb{Z}^d}$. We will also denote $\sigma(\boldsymbol{u},x)$ by $\sigma_{\boldsymbol{u}}(x)$. A $d$-dimensional \textbf{shift} on alphabet $\mathcal{A}$ is a compact subspace \notationidx{$X$}{standard notation for a shift} of $\mathcal{A}^{\mathbb{Z}^d}$ which is stable under the shift action, meaning that for all $\boldsymbol{u} \in \mathbb{Z}^d$, 
$\sigma_{\boldsymbol{u}} (X) \subset X$.

\begin{notation}
For every graph $G$, we denote by \notationidx{$X^d_G$}{the $d$-dimensional homshift associated with graph $G$} the shift with alphabet $V_G$ defined as the set of configurations $x \in {V_G}^{\mathbb{Z}^d}$ such that for all $\boldsymbol{u},\boldsymbol{v} \in \mathbb{Z}^d$ which are neighbors (here we see $\mathbb{Z}^d$ as the grid graph), we have $(x_{\boldsymbol{u}},x_{\boldsymbol{v}}) \in E_G$. In other words, $X^d_G$ is the set of graph homomorphisms from $\mathbb{Z}^d$ to $G$. We call $X_G^d$ the $d$-dimensional \textbf{homshift} associated with $G$.
\end{notation}

For two subsets $\mathbb{U}$ and $\mathbb{U}'$ of 
$\mathbb{Z}^d$, we set $\delta(\mathbb{U},\mathbb{U}') \coloneqq \min_{\substack{\textbf{u} \in \mathbb{U}\\ \textbf{u}' \in \mathbb{U}'}} \|\textbf{u}-\textbf{u}'\|_{\infty}$.
We denote by $\N$ the set of natural numbers that is $\{n \in \Z~:~n \geq 0\}$ and by $\N^*$ the set of positive natural numbers.

\begin{definition}
Let us consider a function $f : \mathbb{N}^{*} \rightarrow \mathbb{N}^{}$, and an integer $k \in \mathbb{N}^{*}$. A shift space is said to be $(f,k)$-\textbf{phased block gluing} when, for every globally admissible block patterns $p$ and $p'$ of the same size $n$, 
and $\textbf{u},\textbf{u}' \in \mathbb{Z}^d$ such that 
\[\delta\left(\textbf{u} + \llbracket 0,n-1\rrbracket^d, \textbf{u}' + \llbracket 0,n-1\rrbracket^d \right) \ge f(n),\]
there exists some $x\in X$ and some $\textbf{v} \in \mathbb{Z}^d$
such that $\|\textbf{v}\|_{\infty} < k$, $x_{\textbf{u} + \llbracket 0,n-1\rrbracket^d}=p$ and $x_{\textbf{u}'+\textbf{v} + \llbracket 0,n-1\rrbracket^d}=p'$. A shift which is $(f,1)$-\textbf{phased block gluing} for some $f$ is simply said to be $f$-\textbf{block gluing}. A shift which is $(f,k)$-phased block gluing for some $f$ and $k \ge 1$ is said to be \textbf{phased block gluing}.
\end{definition}

We are using the Landau notations $o, O$ and $\Theta$ throughout the article.

\begin{definition}
A shift $X$ is said to be $(\Theta(g),k)$-phased block gluing (resp. $(O(g),k)$-block gluing) when it is $(f,k)$-phased block gluing with $f \in \Theta(g)$ (resp. $O(g)$).
\end{definition}

The following result follows from the proofs in \cite[Section 3]{MR3743365} although it is not stated explicitely.

\begin{lemma}
    For all $d > 1$ and every finite undirected graph $G$, the homshift $X_G^d$ is $O(n)$-phased block gluing.
\end{lemma}

The following can be deduced from \cite[Lemma 3.5 and Proposition 3.6]{GHO23}:

\begin{lemma}\label{lemma.phase.to.non.phase}
    For all $d >1$, and any function $f : \mathbb{N} \rightarrow \mathbb{N}$, the homshift $X^d_G$ is $f(n)$-block gluing if and only if it is $f(n)$-phased block gluing and $G$ is not bipartite.
\end{lemma}

\subsection{Finitely presented groups\label{section.background.group.presentations}}

Provided a finite set $\mathcal{A}$, we denote the \textbf{free group} generated by $\mathcal{A}$ by $\mathbb{F}_{\mathcal{A}}$. All the free groups generated by a set of cardinality $k$ are isomorphic and we denote by \notationidx{$\mathbb{F}_k$}{the free group with $k$ generators} their isomorphic class, which we assimilate to any of its elements. Provided a group $\mathbb{G}$, a subgroup $\mathbb{H}$ is said to be \textbf{normal} when for all $h \in \mathbb{H}$ and $g \in \mathbb{G}$, $g h g^{-1} \in \mathbb{H}$. Given a finite subset $R$ of $\mathbb{G}$, we denote by $\mathbb{G}/R$ the quotient of $\mathbb{G}$ by the smallest normal subgroup containing $R$. A group $\mathbb{G}$ is said to be \textbf{finitely presented} when it is the quotient $\mathbb{F}_E/R$ of some free group $\mathbb{F}_E$ for a fixed finite set $R \subset \mathbb{F}_E$. We say that $(E,R)$ is a \textit{presentation} of $\mathbb{G}$ and we denote this by $\mathbb{G} =$ \notationidx{$\langle E:R\rangle$}{group generated by generator set $E$ and relation set $R$}. \textit{Note that traditionally $\langle E:R\rangle$ is also called a presentation. In this text, we distinguish a notation for presentations and for groups, as it is crucial to make our results clear.}

Provided two finitely presented groups $\mathbb{G},\mathbb{H}$, the \textbf{free product} of $\mathbb{G} = \langle E : R\rangle$ and $\mathbb{H} = \langle E' : R'\rangle$, where $E, E'$ are disjoint, is the group \notationidx{$\mathbb{G} * \mathbb{H}$}{free product of two groups $\mathbb{G}$,$\mathbb{H}$}$=\langle E \cup E' : R\cup R' \rangle$ where the set $R,R'$ are seen as subset of the free group generated by $E \cup E'$.

We use the following technical lemma in many proofs: 

\begin{lemma}\label{lemma.mod.presentation}
    Let $(E, R)$ be a finite presentation of a group, and:
    \begin{enumerate}
        \item For some $r \in R$ written as $r=w e w'$ where $e$ does not appear in $w$ or $w'$, then, $\langle E : R \rangle = \langle \tilde E : \tilde R \rangle$, where $\tilde E = E \backslash \{e\}$ and $\tilde R$ is obtained from $R$ by removing $r$ and replacing every occurrence of $e$ in other elements of $R$ with $w^{-1} (w')^{-1}$.
        \item { If $r \in R$ belongs to} the smallest normal subgroup of $\mathbb F_E$ containing $R\setminus \{r\}$, we have $\langle E : R \rangle = \langle E : R \backslash \{r\} \rangle$.
    \end{enumerate} 
\end{lemma}

\begin{proof}
Point 1 follows from the fact that the kernel of the homomorphism $\phi: {\mathbb F}_E\to \langle \tilde E: \tilde R\rangle$ given by 
\[\phi(f)=\begin{cases}
    f\text{ if } f\in E\setminus \{e\}\\
    w^{-1} (w')^{-1}\text{ if }f=e
\end{cases}\]
is the smallest normal subgroup containing $R$. 
Point 2 is straightforward.
\end{proof}

\section{Algebraic topology background\label{section.background.algebraic.topology}}

In this section we provide necessary background on algebraic topology on graphs, which will be useful in Section \ref{section.square.quotients}. We introduce graph covers in Section \ref{sec:covers}, then fundamental groups and universal covers in Section \ref{sec:universal.covers}. 
Section \ref{section:regular.covers} is devoted to regular graph covers.

\subsection{Fundamental group of a graph}\label{sec:fundamental_group}

Roughly speaking, the fundamental group $\pi_1(\mathcal{T})$ of an (arcwise connected) topological space $\mathcal{T}$ based on $a\in \mathcal{T}$ is the set of loops (continuous paths) from $a$ to $a$ endowed with the concatenation operation, where two loops are considered equal if they can be continuously deformed into one another (notion of homotopy). This forms a group which is independent from $a$, up to isomorphism. For formal definitions and standard properties of fundamental group the reader is referred to any textbook on algebraic topology (e.g. see \cite{Hatcher,rotman2013introduction}).

In this section, we provide a more concrete description of a related concept of the fundamental group of a graph.

\begin{notation}
    For every $a \in V_G$, let $\pi_1(G)[a]$ denote the set of non-backtracking cycles of $G$ which begin and end at $a$. We endow this set with the operation $(p,p') \mapsto p\star p'$ which makes it a group whose identity element is the cycle of length zero $a$ and in which the inverse of a cycle $p$ is $p^{-1}$. 
\end{notation}

\begin{definition}[Fundamental Group]
All the groups  $\left(\pi_1(G)[a],\star\right)$ are isomorphic. Indeed, for two vertices $a,b$ and a walk $p$ from $b$ to $a$, the map $c\mapsto p\star c\star p^{-1}$ is a group isomorphism from $\left(\pi_1(G)[a],\star\right)$ to $\left(\pi_1(G)[b],\star\right)$. 
We call \textbf{fundamental group} of $G$ their equivalence class and denote it by $\pi_1(G)$.
\end{definition}

In practice, in order to simplify the notations, we will drop the vertex $a$ - that we call the \textbf{base vertex} - in the notation and write $\pi_1(G)$. For all the notations introduced below, we will also drop the base point after introducing them. Any statement that we formulate is true for all the base vertices.

A spanning tree of a graph $G$ is a connected graph $T$ such that $V_T = V_G$ and $E_T \subset E_G$ which has no nontrivial simple cycle.

\begin{notation}
    Provided a spanning tree $T$ of a graph $G$, for all $a \in V_G$, we denote by $p^a_T(b)$ the unique non-backtracking walk on $T$ which begins at $a$ and ends at $b$. This notation is illustrated in Figure~\ref{fig.example.p.func}. 
\end{notation}

The following result is well-known and can be derived from standard material in \cite{MR0695906}. 

\begin{notation}
For any graph $G$ and $T$ a spanning tree of $G$, we denote by $R_T(G)$ the following set: 
\[R_T(G) \coloneqq E_T \cup \{e e'~:~e,e'\in E_G\wedge (\exists u,v \in V_G : e=(u,v),e'=(v,u))\}.\]
\end{notation}
\begin{proposition}
\label{proposition: Free product}
For every graph $G$, we have: 
\[\pi_1(G)\cong \langle E_G : R_T(G)\rangle \cong \langle E_G\backslash E_T : R_T(G)\setminus E_T\rangle.\]
\end{proposition}

\begin{proof}[Sketch of the proof]
Consider a group homomorphism $\beta_T$ from 
$F_{E_G}$ to $\pi_1(G)$ defined by: 
\begin{equation}\label{equation: defining the map on edges}   
\beta_T((u,v))=\begin{cases}
\varepsilon\text{ if } (u,v) \in E_T\\
p_T(u) \star (u,v) \star (p_T(v))^{-1}\text{ otherwise.}
\end{cases}\end{equation}
The kernel of $\beta_T$ contains $R_T(G)$, because:
\begin{enumerate}
    \item The image by $\beta_T$ of an edge in $T$ is a cycle in $T$ which is equal to the identity in the group $\pi_1(G)$. 
    \item Using Equation \ref{equation: defining the map on edges},  one can see that for all $(u,v)\in E_G$, $\beta_T((u,v) (v,u))$ is the identity in the group $\pi_1(G)$.
\end{enumerate}
The map $\beta_T$ thus yields a map from $\langle E_G : R_T(G)\rangle$ to $\pi_1(G)$, which is an isomorphism. Indeed, it is inverted by the homomorphism from $\pi_1(G)$ to $\langle E_G : R_T(G)\rangle$ which to a cycle $p$ associates the product of the elements $(p_i,p_{i+1})$.
\end{proof}
The following is an immediate consequence of the definition of $R_T(G)$.
\begin{corollary}
    For every graph $G$, $\pi_1(G)$ is a free product of a free group $\mathbb F_k$ and $n$ copies of $\mathbb Z/2\mathbb Z$, where $n$ is the number of self-loops in $G$ and $k =  |E_G| - |V_G| - n+1$.
\end{corollary}

\begin{remark}\label{remark: Self-loop problem}
    Proposition~\ref{proposition: Free product} may be confusing for a reader familiar with algebraic topology but not with graph theory. A reference such as \cite[Proposition 1.A2]{Hatcher} states that the fundamental group of any graph is a free group, but our definitions differ on self-loops. Consider the graph $G$ with a unique vertex $a$ and a self-loop $(a,a)$. If one thinks of $G$ as a CW-complex representing the topological space $\mathbb R/\Z$ then the corresponding fundamental group is $\Z$. However, the fundamental group of $G$ according to our definition is $\Z/2\Z$.
    This is due to the fact that the clockwise and counterclockwise cycles are distinct on $\mathbb R/\Z$ while the cycle $(a,a)$ is equal to its inverse in $G$. Our results on graph covers (for instance Proposition \ref{ proposition:config.lifting.2}) would fail for the topological definition. 
\end{remark}

\begin{figure}[ht!]
    \centering
        \begin{tikzpicture}[scale=0.5]
        \draw (0,0) rectangle (2,2);
        \draw[fill=gray!40] (0,0) circle (3pt);
        \draw[fill=gray!40] (2,2) circle (3pt);
        \draw[fill=gray!40] (2,0) circle (3pt);
        \draw[fill=gray!40] (0,2) circle (3pt);

        \begin{scope}[yshift=-3cm]
        \draw (1,0) -- (2,1) -- (2,-1) -- (1,0) -- (0,1) -- (0,-1) -- (1,0);
        \draw[fill=gray!40] (1,0) circle (3pt);
        \draw[fill=gray!40] (2,1) circle (3pt);
        \draw[fill=gray!40] (2,-1) circle (3pt);
        \draw[fill=gray!40] (0,-1) circle (3pt);
        \draw[fill=gray!40] (0,1) circle (3pt);
        \end{scope}

        \begin{scope}[yshift=-7cm]
        \draw (-0.25,0) circle (8pt);
        \draw (0,0) -- (2,1) -- (2,-1) -- (0,0);
        \draw[fill=gray!40] (0,0) circle (3pt);
        \draw[fill=gray!40] (2,1) circle (3pt);
        \draw[fill=gray!40] (2,-1) circle (3pt);
        \end{scope}

        \begin{scope}[xshift=5cm]
        \draw (2,0) -- (0,0) -- (0,2) -- (2,2);
        \draw[dashed, ->] (2,0) -- (2,2) node[midway, right] {$g_1$};
        \draw[fill=gray!40] (0,0) circle (3pt);
        \draw[fill=gray!40] (2,2) circle (3pt);
        \draw[fill=gray!40] (2,0) circle (3pt);
        \draw[fill=gray!40] (0,2) circle (3pt);
        \node at (8,1) {$\langle g_1:\rangle = \mathbb{Z}$};
        \end{scope}

        \begin{scope}[xshift=5cm,yshift=-3cm]
        \draw (0,1) -- (2,-1);
        \draw (0,-1) -- (2,1);
        \draw[dashed, ->] (0,1) -- (0,-1) node[midway, left] {$g_1$};
        \draw[dashed, ->] (2,1) -- (2,-1) node[midway, right] {$g_2$};
        \draw[fill=gray!40] (1,0) circle (3pt);
        \draw[fill=gray!40] (2,1) circle (3pt);
        \draw[fill=gray!40] (2,-1) circle (3pt);
        \draw[fill=gray!40] (0,-1) circle (3pt);
        \draw[fill=gray!40] (0,1) circle (3pt);
        \node at (8,0) {$\langle g_1, g_2:\rangle=\mathbb{F}_2$};
        \end{scope}

        \begin{scope}[xshift=5cm,yshift=-7cm]
        \draw[dashed] (-0.25,0) circle (8pt) node[left] {$g_1$};
        \draw (0,0) -- (2,1);
        \draw (2,-1) -- (0,0);
        \draw[dashed, ->] (2,-1) -- (2,1) node[midway, right] {$g_2$};
        \draw[fill=gray!40] (0,0) circle (3pt);
        \draw[fill=gray!40] (2,1) circle (3pt);
        \draw[fill=gray!40] (2,-1) circle (3pt);
        \node at (8,0) {$\langle g_1,g_2 : g_1^2\rangle = \mathbb{Z}/2\mathbb{Z}\raisebox{0.25ex}{$\star$} \mathbb{Z}$};
        \end{scope}
    \end{tikzpicture}
    \caption{Illustration for the definition of fundamental group on three examples. Left column: the graph $G$. Middle column: full lines represent the chosen spanning tree, dotted lines are the generators. Right column: presentation for the fundamental group obtained by Proposition~\ref{proposition: Free product}.
    }
\label{figure.example.fundamental.group}\end{figure}
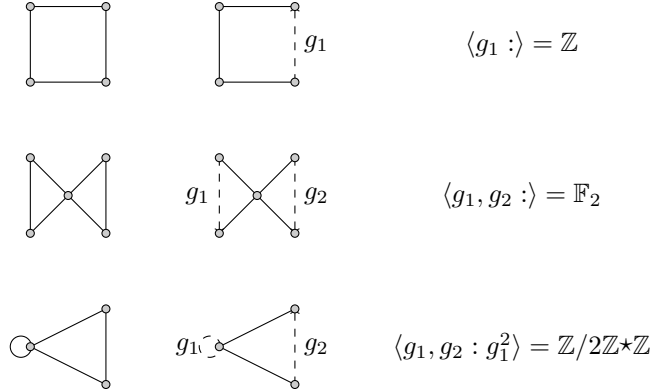
    
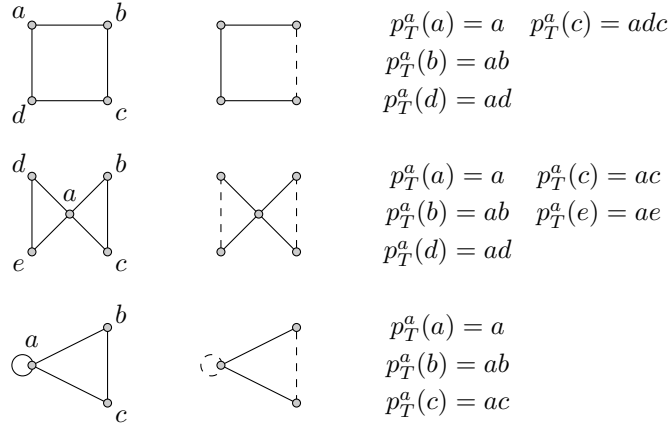
\begin{figure}
    \centering
        \begin{tikzpicture}[scale=0.5]
        \draw (0,0) rectangle (2,2);
        \draw[fill=gray!40] (0,0) circle (3pt);
        \draw[fill=gray!40] (2,2) circle (3pt);
        \draw[fill=gray!40] (2,0) circle (3pt);
        \draw[fill=gray!40] (0,2) circle (3pt);
        \node at (2.35,2.35) {$b$};
        \node at (2.35,-0.35) {$c$};
        \node at (-0.35,-0.35) {$d$};
        \node at (-0.35,2.35) {$a$};
        
        \begin{scope}[yshift=-3cm]
        \draw (1,0) -- (2,1) -- (2,-1) -- (1,0) -- (0,1) -- (0,-1) -- (1,0);
        \draw[fill=gray!40] (1,0) circle (3pt);
        \draw[fill=gray!40] (2,1) circle (3pt);
        \draw[fill=gray!40] (2,-1) circle (3pt);
        \draw[fill=gray!40] (0,-1) circle (3pt);
        \draw[fill=gray!40] (0,1) circle (3pt);
        \node at (2.35,1.35) {$b$};
        \node at (2.35,-1.35) {$c$};
        \node at (-0.35,1.35) {$d$};
        \node at (-0.35,-1.35) {$e$};
        \node at (1,0.5) {$a$};
        \end{scope}

        \begin{scope}[yshift=-7cm]
        \draw (-0.25,0) circle (8pt);
        \draw (0,0) -- (2,1) -- (2,-1) -- (0,0);
        \draw[fill=gray!40] (0,0) circle (3pt);
        \draw[fill=gray!40] (2,1) circle (3pt);
        \draw[fill=gray!40] (2,-1) circle (3pt);
        \node at (2.35,1.35) {$b$};
        \node at (2.35,-1.35) {$c$};
        \node at (0,0.625) {$a$};
        \end{scope}

        \begin{scope}[xshift=5cm]
         \draw (2,0) -- (0,0) -- (0,2) -- (2,2);
        \draw[dashed] (2,0) -- (2,2);
        \draw[fill=gray!40] (0,0) circle (3pt);
        \draw[fill=gray!40] (2,2) circle (3pt);
        \draw[fill=gray!40] (2,0) circle (3pt);
        \draw[fill=gray!40] (0,2) circle (3pt);
        \node at (6,1) {$p_T^a(b) = ab$};
        \node at (6,2) {$p_T^a(a) = a$};
        \node at (6,0) {$p_T^a(d) = ad$};
        \node at (10,2) {$p_T^a(c) = adc$};
        \end{scope}

        \begin{scope}[xshift=5cm,yshift=-3cm]
        \draw (0,1) -- (2,-1);
        \draw (0,-1) -- (2,1);
        \draw[dashed] (0,1) -- (0,-1);
        \draw[dashed] (2,1) -- (2,-1);
        \draw[fill=gray!40] (1,0) circle (3pt);
        \draw[fill=gray!40] (2,1) circle (3pt);
        \draw[fill=gray!40] (2,-1) circle (3pt);
        \draw[fill=gray!40] (0,-1) circle (3pt);
        \draw[fill=gray!40] (0,1) circle (3pt);
        \node at (6,0) {$p_T^a(b) = ab$};
        \node at (6,1) {$p_T^a(a) = a$};
        \node at (6,-1) {$p_T^a(d) = ad$};
        \node at (10,1) {$p_T^a(c) = ac$};
        \node at (10,0) {$p_T^a(e) = ae$};
        \end{scope}

        \begin{scope}[xshift=5cm,yshift=-7cm]
        \draw[dashed] (-0.25,0) circle (8pt);
        \draw (0,0) -- (2,1);
        \draw (2,-1) -- (0,0);
        \draw[dashed] (2,-1) -- (2,1);
        \draw[fill=gray!40] (0,0) circle (3pt);
        \draw[fill=gray!40] (2,1) circle (3pt);
        \draw[fill=gray!40] (2,-1) circle (3pt);
        \node at (6,0) {$p_T^a(b) = ab$};
        \node at (6,1) {$p_T^a(a) = a$};
        \node at (6,-1) {$p_T^a(c) = ac$};
        \end{scope}
    \end{tikzpicture}
    \caption{Illustration of the function $v \mapsto p^a_T(v)$. The spanning tree $T$ has been indicated in the middle column.}
    \label{fig.example.p.func}
\end{figure}

\subsection{Graph covers and deck transformations\label{sec:covers}}

\begin{definition}
A \textbf{covering map} from a graph $\tilde G$ to $G$ is a graph homomorphism $\theta: \tilde G\to G$ which is a `local isomorphism', meaning that for all $a \in G$,
\[\theta^{-1}(N_G(a))=\bigsqcup_{\tilde a\in \theta^{-1}(a)}N_{\tilde G}(\tilde a),\] 
and the map $\theta|_{N_{\tilde G}(\tilde a)}$ is bijection onto $N_G(\theta(\tilde a))$ for all $\tilde a\in V_{\tilde G}$. A \textbf{cover} of a graph $G$ is a graph $\tilde G$ such that there exists a covering map from $\tilde G$ to $G$. One can find an illustration for this definition in Figure~\ref{figure.covers}. 
\label{def.cover}
\end{definition}

\begin{figure}[ht!]
    \centering
        \begin{tikzpicture}[scale=0.5]
        \begin{scope}
            \draw (0,0) rectangle (2,2);
            \foreach \x in {0,2} {
                \foreach \y in {0,2} {
                \draw[fill=gray!40] (\x,\y) circle (3pt);
                }
            }
            \node[scale=0.9] at (-0.35,-0.35) {b};
            \node[scale=0.9] at (2.35,-0.35) {c};
            \node[scale=0.9] at (2.35,2.35) {d};
        \node[scale=0.9] at (-0.35,2.35) {a};
        \end{scope}

        \begin{scope}[xshift=7.5cm]
            \draw (0,-0.5) rectangle (3,2.5);
            \foreach \y in {0,1,2,3} {
            \draw[fill=gray!40] (3,-0.5+\y) circle (3pt);
            \draw[fill=gray!40] (0,-0.5+\y) circle (3pt);
            }
            \foreach \x in {1,2} {
            \draw[fill=gray!40] (\x,2.5) circle (3pt);
            \draw[fill=gray!40] (\x,-0.5) circle (3pt);
            }
            \node[scale=0.9] at (-0.5,2.85) {a};
            \node[scale=0.9] at (1,3.1) {b};
            \node[scale=0.9] at (2,3) {c};
            \node[scale=0.9] at (3.35,2.85) {d};
            \node[scale=0.9] at (3.5,1.5) {a};
            \node[scale=0.9] at (3.5,0.5) {b};
            \node[scale=0.9] at (3.35,-0.85) {c};
            \node[scale=0.9] at (2,-1) {d};
            \node[scale=0.9] at (1,-1) {a};
            \node[scale=0.9] at (-0.35,-0.85) {b};
            \node[scale=0.9] at (-0.35,-0.85) {b};
            \node[scale=0.9] at (-0.5,0.5) {c};
            \node[scale=0.9] at (-0.5,1.5) {d};
        \end{scope}

        \begin{scope}[yshift=-5cm]
        \draw (-0.25,0) circle (8pt);
        \draw (0,0) -- (2,1) -- (2,-1) -- (0,0);
        \draw[fill=gray!40] (0,0) circle (3pt);
        \draw[fill=gray!40] (2,1) circle (3pt);
        \draw[fill=gray!40] (2,-1) circle (3pt);
        \node at (2.35,1.35) {$b$};
        \node at (2.35,-1.35) {$c$};
        \node at (0,0.625) {$a$};
        \end{scope}

        \begin{scope}[yshift=-5cm,xshift=9.5cm]
        \draw (0,0) -- (2,1) -- (2,-1) -- (0,0);
        \draw (0,0) -- (-1,0);
        \draw (-1,0) -- (-3,1) -- (-3,-1) -- (-1,0);
        \draw[fill=gray!40] (0,0) circle (3pt);
        \draw[fill=gray!40] (2,1) circle (3pt);
        \draw[fill=gray!40] (2,-1) circle (3pt);
        \draw[fill=gray!40] (-1,0) circle (3pt);
        \draw[fill=gray!40] (-3,1) circle (3pt);
        \draw[fill=gray!40] (-3,-1) circle (3pt);
        \node at (2.35,1.35) {$b$};
        \node at (2.35,-1.35) {$c$};
        \node at (0,0.625) {$a$};
        \node at (-3.35,1.35) {$b$};
        \node at (-3.35,-1.35) {$c$};
        \node at (-1,0.625) {$a$};
        \end{scope}
    \end{tikzpicture}
    \caption{Illustration for the definition of cover. For each graph $G$ in the left column is represented a cover of $\tilde{G}$ in the right column, where the labels are the images of vertices by the covering map.
    }
    \label{figure.covers}
\end{figure}
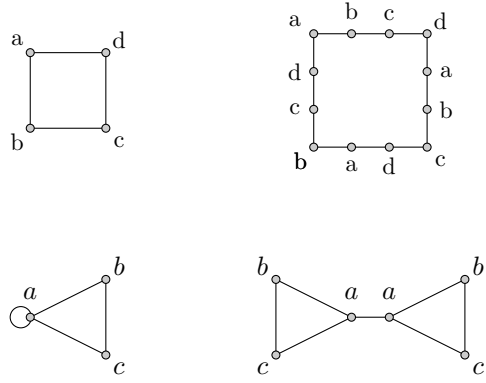

A walk on the graph $G$ can be lifted to a walk on any cover $\tilde G$ of $G$, a property that we use extensively in the remainder of the text.

\begin{proposition}[Walk-lifting property]
\label{prop: path lifting}
Let $\theta: {\tilde G}\to G$ be a covering map. Given a walk $p$ on $G$ and $\tilde p_0\in \theta^{-1}(p_0)$, there is a unique walk $\tilde p$ on $\tilde G$ starting at $\tilde p_0$ such that $l(\tilde p) = l(p)$ and
$\theta(\tilde p_i)=p_i$ for all $0\leq i\leq l(p)$.
\end{proposition}
We call $\tilde p$ a \textbf{lift} of $p$ starting at $\tilde p_0$.

\begin{proof}
By definition of covering map, $\theta|_{N_{{\tilde G}(\tilde p_0)}}$ is a graph isomorphism onto $N_{G}(p_0)$. Thus there is a unique vertex $\tilde p_1 \in N_{\tilde G} (\tilde p_0)$ such that $\theta(\tilde p_1)=p_1$. The statement follows from iterating this reasoning. 
\end{proof} 

\begin{remark}
     A lift of a cycle is not necessarily a cycle: consider the graph $G$ on the first row of Figure \ref{figure.covers}. No cycle in $\tilde G$ is a lift of the cycle $p=abcda$.
\end{remark}

Graph covers are related to the fundamental group via deck transformations, which are actions of the fundamental group on graph covers. 

\begin{definition}
    The \textbf{deck transformations} of a covering map $\theta: \tilde G \to G$ are the graph automorphisms $\eta$ of $\tilde G$ which are equivariant with respect to the projection map $\theta$, that is, $\theta\circ \eta=\theta$. 
\end{definition}

A deck transformation is essentially determined by the image of a vertex: 

\begin{proposition}\label{prop: uniqueness of deck transformation}
    Let $\theta: \tilde G \to G$ be a covering map and $\tilde a, \tilde a'\in \tilde G$ be such that $\theta(\tilde a)=\theta(\tilde a')$. There exists at most one deck transformation $\eta$ such that $\eta(\tilde a)=\tilde a'$.
\end{proposition}

\begin{proof}
Let us assume that there exists two such deck transformations $\eta$ and  $\eta'$. We prove that $\eta = \eta'$.
For all walks $\tilde p$ on $\tilde G$, denote by $\eta(\tilde p)$ (resp. $\eta'(\tilde p)$, $\theta(\tilde p)$) the walk on $\tilde G$ or $G$ obtained by applying $\eta$ (resp. $\eta'$, $\theta$) vertex by vertex. 
Let $\tilde p$ be any walk which starts at $\tilde a$. By the definition of deck transformation, $\eta(\tilde p)$ and $\eta'(\tilde p)$ are both lifts of $\theta(\tilde p)$ which start at ${\tilde a}'$. By Proposition \ref{prop: path lifting}, this lift is unique, so $\eta(\tilde p) = \eta'(\tilde p)$. Since $\tilde G$ is connected and this holds for any walk which starts at $\tilde a$, this implies that $\eta = \eta'$.
\end{proof}

\begin{remark}
    In Proposition \ref{prop: uniqueness of deck transformation}, the deck transformation $\eta$ does not have to exist. For instance, consider Figure \ref{fig:cover-discover} and the covering map $\theta : \tilde G \to G$ defined by dropping indices, that is, $\theta(l_i) = l$ for all $i$ and all $l \in V_G$. There is no deck transformation for $\theta$ such that $\eta(a_2) = a_1$. Indeed, such a deck transformation $\eta$ would satisfy $\eta(b_1) = \eta(b_2) = b_3$, so it is not an automorphism of $\tilde G$.
\end{remark}

\subsection{Universal graph cover\label{sec:universal.covers}}

\begin{notation}
    For every $a \in V_G$, we denote by $\mathcal{U}_G[a]$ the graph whose vertices are the non-backtracking walks on $G$ beginning at $a$ and edges are the pairs of walks $(p,q)$ such that either $p = qv$ or $q = pv$ for some $v \in V_G$, that is, one is extension of the other by a single step. 
\end{notation}

\begin{definition} Like for the fundamental group, all graphs $\mathcal{U}_G[a]$, $a\in G$,  are isomorphic. For example, for $a,b \in G$, $q$ a non-backtracking walk from $a$ to $b$ and $p$ non-backtracking walk starting at $a$, the map $p \mapsto q \star p$ is an isomorphism from $\mathcal{U}_G[a]$ to $\mathcal{U}_G[b]$. The isomorphism class of $\mathcal{U}_G[a]$, $a\in V_G$ is called the \textbf{universal cover} of $G$ and denoted $\mathcal{U}_G$. This definition is illustrated in Figure~\ref{figure.example.universal.cover}. We identify it with any of its elements and drop the base vertex from the notations in the following.
\end{definition} 

\begin{figure}[ht!]
    \centering
    \input{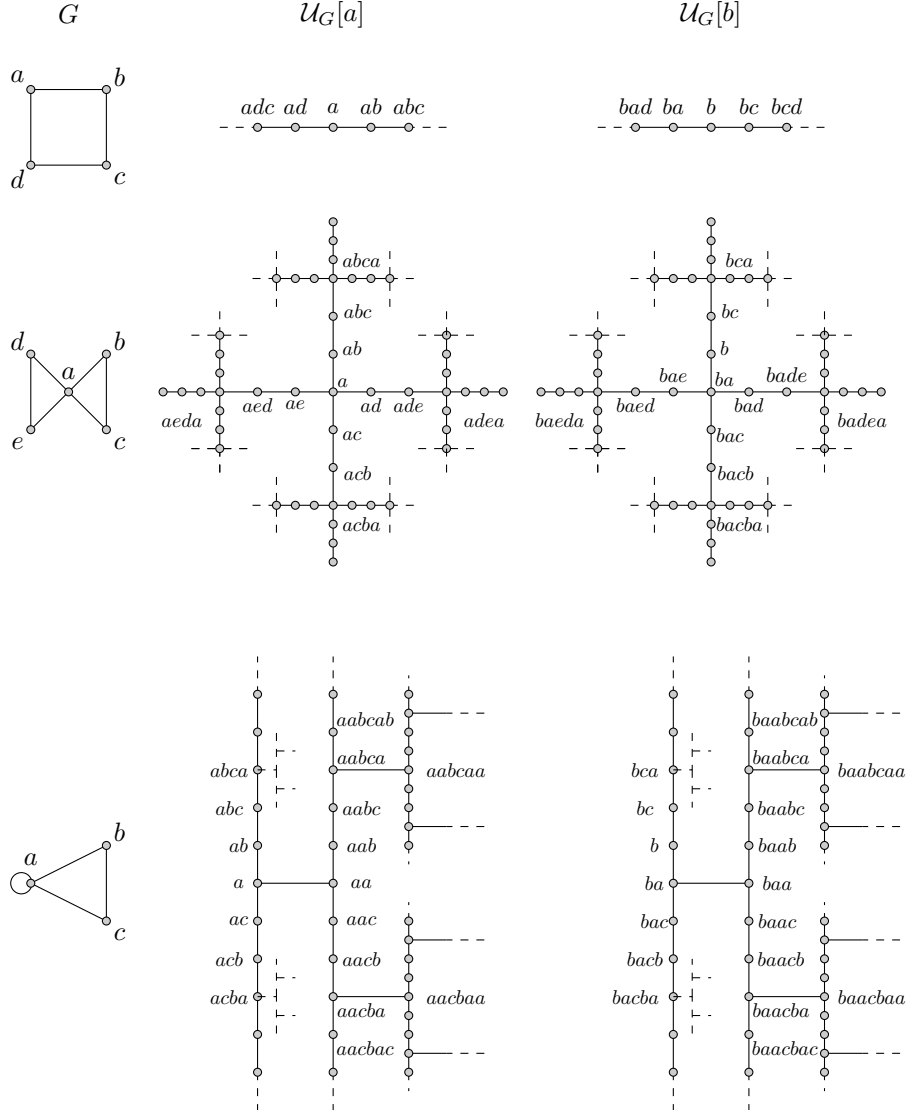}
    \caption{Illustration for the definition of universal cover. Left column: the three graphs from Figure~\ref{figure.example.fundamental.group}. Middle and right columns: two representations of the square cover for different base points, that are indeed isomorphic graphs.}\label{figure.example.universal.cover}
\end{figure}

One can check that the universal cover does not have any cycle.

\begin{notation}
    For all $a \in V_G$, let \notationidx{$\alpha_a$}{morphism that to a walk associate its last vertex.} : $\mathcal U_G[a]\to G$ be the graph homomorphism such that $\alpha_a(p)$ is the terminal vertex of the walk $p$.
\end{notation} 

\begin{proposition}[Universal covers are graph covers]
The map $\alpha$ is a covering map, which makes $\mathcal U_G$ a cover of $G$.
\end{proposition}

\begin{proof}
Fix an arbitrary base vertex $a$. For all $b \in V_G$, $\alpha_a^{-1}(b)$ is the set of non-backtracking walks on $G$ from $a$ to $b$. 

\textbf{1.} We have to prove that for two such walks $p\neq q$, $N_{\mathcal U_G[a]}{(p)}$ and $N_{\mathcal U_G[a]}{(q)}$ are disjoint. Let us denote by $b_1, b_2, \ldots b_r$ the elements of $N_G(b)$. Then we have that
\[N_{\mathcal U_G[a]}{(p)}=\{\varphi(pb_i)~:~1\leq i \leq r\}.\] 
$p$ has no backtracks so $\varphi(pb_i)$ starts with $p_0 \ldots p_{l(p)-1}$. Since $p_0 \ldots p_{l(p)-1} \neq q_0 \ldots q_{l(q)-1}$ ($p$ and $q$ end at $b$), the only possibility for $\varphi(pb_i) = \varphi(qb_j)$ would be that $\varphi(pb_i) = p_0 \ldots p_{l(p)-1}$ and $\varphi(qb_j) = q b_j$, so $p = qb_j b$ (or the symmetric case). This would imply that $p$ is backtracking since $q$ ends at $b$. We have proved that $N_{\mathcal U_G[a]}{(p)} \cap N_{\mathcal U_G[a]}{(q)} = \emptyset$. \bigskip 
   
\textbf{2.} We are left to prove that for every walk $p \in \mathcal U_G[a]$, $\alpha_a$ is a bijection from $N_{\mathcal U_G[a]}{(p)}$ to $N_G(p_{l(p)})$. This comes directly from the fact that the walks $\varphi(p b_i)$, $1\leq i \leq r$, where $b_i$ are the elements of $N_G(p_{l(p)})$, are all distinct.\bigskip 

These two facts imply, by the definition, that $\alpha$ is a covering map.\end{proof}

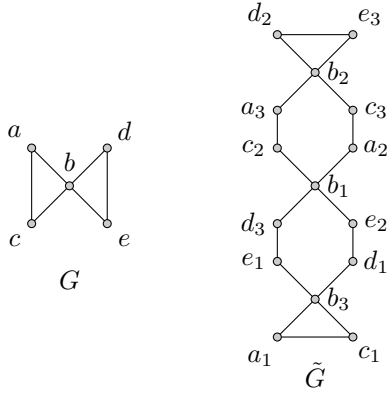
\begin{figure}
    \centering
       \begin{tikzpicture}[scale=0.5]
       \draw (0,0) -- (2,2) -- (2,0) -- (0,2) -- (0,0);
       \draw[fill=gray!40] (0,0) circle (3pt);
       \draw[fill=gray!40] (0,2) circle (3pt);
       \draw[fill=gray!40] (2,0) circle (3pt);
       \draw[fill=gray!40] (2,2) circle (3pt);
       \draw[fill=gray!40] (1,1) circle (3pt);
       \node at (1,-1.5) {$G$};
       \node at (1,1.625) {$b$};
       \node at (2.45,2.45) {$d$};
       \node at (2.45,-0.45) {$e$};
       \node at (-0.45,-0.45) {$c$};
       \node at (-0.45,2.45) {$a$};

       \begin{scope}[xshift=7.5cm,yshift=1cm]
           \draw (0,0) -- (1,1) -- (1,2) -- (0,3) -- (1,4) -- (-1,4) -- (0,3) -- (-1,2) -- (-1,1) -- (0,0) -- (1,-1) -- (1,-2) -- (0,-3) -- (1,-4) -- (-1,-4) -- (0,-3) -- (-1,-2) -- (-1,-1) -- (0,0);
           \node at (0,-5) {$\tilde{G}$};
           \node at (1.45,4.45) {$e_3$};
           \node at (-1.45,4.55) {$d_2$};
           \node at (-1.625,2) {$a_3$};
           \node at (-1.625,1) {$c_2$};
           \node at (1.625,2) {$c_3$};
           \node at (1.625,1) {$a_2$};
           \node at (0.625,3) {$b_2$};
           \node at (0.625,0) {$b_1$};
           \node at (0.625,-3) {$b_3$};

           \node at (1.45,-4.45) {$c_1$};
           \node at (-1.45,-4.55) {$a_1$};
           \node at (-1.625,-2) {$e_1$};
           \node at (-1.625,-1) {$d_3$};
           \node at (1.625,-2) {$d_1$};
           \node at (1.625,-1) {$e_2$};
           
           \draw[fill=gray!40] (0,0) circle (3pt);
           \draw[fill=gray!40] (1,1) circle (3pt);
           \draw[fill=gray!40] (1,2) circle (3pt);
           \draw[fill=gray!40] (0,3) circle (3pt);
           \draw[fill=gray!40] (1,4) circle (3pt);
           \draw[fill=gray!40] (-1,4) circle (3pt);
           \draw[fill=gray!40] (-1,2) circle (3pt);
           \draw[fill=gray!40] (-1,1) circle (3pt);

           \draw[fill=gray!40] (1,-1) circle (3pt);
           \draw[fill=gray!40] (1,-2) circle (3pt);
           \draw[fill=gray!40] (0,-3) circle (3pt);
           \draw[fill=gray!40] (1,-4) circle (3pt);
           \draw[fill=gray!40] (-1,-4) circle (3pt);
           \draw[fill=gray!40] (-1,-2) circle (3pt);
           \draw[fill=gray!40] (-1,-1) circle (3pt);
       \end{scope}
   \end{tikzpicture}
    \caption{An example of cover $\tilde{G}$ of a graph $G$.}
    \label{fig:cover-discover}
\end{figure}

\paragraph{Some deck transformations of the universal cover\label{section:action.fundamental.group}}

To each element of the fundamental group $\pi_1(G)$ we associate a deck transformation of the universal cover $\mathcal U_G$ in the following way.

\begin{notation}
    For all base vertices $a\in G$ and all $g \in \pi_1(G)[a]$, let $\eta^a_g: {\mathcal U_G[a]}\to {\mathcal U_G[a]}$ be the homomorphism such that for any $w$, $\eta^a_g(w)={g}\star w$. 
\end{notation}

\begin{lemma}
    For all $g \in \pi_1(G)$, $\eta_g$ is a deck transformation of $\alpha$. This defines an action of $\pi_1(G)$ on the universal cover $\mathcal U_G$ by deck transformations.\label{lemma: deck transformation}
\end{lemma}

\begin{proof}[Proof sketch]

Since for all $w$, $\eta_g(w)$ and $w$ have the same terminal vertex, $\eta_g$ is a deck transformation. Since the map $\star$ is associative, it follows that this defines a group action. Indeed we have that
\[\eta_{gh}(w)=(gh)\star w= (g\star h)\star w= g\star (h\star w)= \eta_g(\eta_h(w)).\qedhere\]
\end{proof}

Recall from Section~\ref{sec:fundamental_group} that $p^a_T(b)$ is the unique walk from $a$ to $b$ in the spanning tree $T$.

\begin{notation}
For any spanning tree $T$ of $G$ and $a \in G$, set:
\[\psi^a_T: 
\begin{array}{ccc}\pi_1(G) \times G &\to& {\mathcal{U}_G}\\
g,b &\mapsto& \eta_g\circ p^a_T(b).\end{array}\]
\end{notation}

\begin{lemma}\label{lem:inv.psi}
For any spanning tree $T$ of $G$ and $a\in G$, the map $\psi^a_T$ is a bijection. Furthermore, for all $g$, $\psi^a_T(g,.)$ is a graph homomorphism as a map from $T$ to $\mathcal{U}_G$.
\end{lemma}

\begin{proof}
Indeed, its inverse is the map defined by: \[(\psi^a_T)^{-1} : w \mapsto  (w\star (p_T^a(\alpha(w)))^{-1},\alpha(w)).\]
The second part of the statement follows since for all vertices $b$ and $c$ adjacent in $T$, $p^a_T(b)$ and $p^a_T(c)$ are adjacent as well.
\end{proof}

Let us recall that for a group $\Gamma$ and a set $X$, an action of $\Gamma$ on $X$ is a map $(g,x) \mapsto g \cdot x$ from $\Gamma \times X$ to $X$ such that if $g$ is the identity element of $\Gamma$, then $g \cdot x = x$ for all $x \in X$, and for all $g,h \in \Gamma$ and $x \in X$, $g \cdot (h \cdot x) = (gh) \cdot x$.  An action is said to be \textbf{free} when $g \cdot x = x$ implies that $g$ is the identity element of $\Gamma$, and \textbf{transitive} on $Z \subset X$ when for all $z,z' \in Z$ there exists $g \in \Gamma$ such that $z' = g \cdot z$.

\begin{proposition}\label{prop:trans.deck}
     The map $(g,w) \rightarrow \eta_g(w)$ from $\pi_1(G)\times {\mathcal U_G}$ to ${\mathcal U_G}$ is a free group action which is transitive on each $\alpha^{-1}(b)$, $b \in V_G$.
\end{proposition}

\begin{proof}
    By Lemma \ref{lemma: deck transformation} this map is a group action. Let us prove that it is free. Consider $g,w$ such that $\eta_g(w) = w$. This implies that $g\star w=w$ which implies that $g$ is the identity element of $\pi_1(G)$. Let us prove that the action is transitive on each $\alpha^{-1}(b), b \in V_G$. Consider two walks $w,w' \in \mathcal{U}_G$ such that $\alpha(w) = \alpha(w')$. By Lemma \ref{lem:inv.psi}, there exist $g'$ and $g''$ such that $w = g' \star p_T(b)$ and $w' = g'' \star p_T(b)$. We thus have $w' = g\star w = \eta_g(w)$, where $g = g' \star (g'')^{-1}$. 
\end{proof}

\begin{example}
Figure~\ref{fig.example.p.func} illustrates the definition of the map $v \mapsto p_T(v)$. Let us provide one example of value for the map $(\psi^a_T)^{-1}$ for each graph, from top to bottom. For the first graph, if $w=abcdabc$, then $(\psi^a_T)^{-1}(w) = (abcdabc da, c)$. On the second graph, if $w=abcad$ then $(\psi^a_T)^{-1}(w) = (abca, d)$. On the third graph, if $w = aabc$, then $(\psi^a_T)^{-1}(w) = (aabca, c)$.
\end{example}

\paragraph{Quotients of the universal cover}

\begin{notation}
   For an undirected graph $G$, a group $\Gamma$, and $(\gamma,a) \mapsto \gamma \cdot a$ an action of $\Gamma$ on $V_G$, we denote by $\Gamma(a)$ the orbit of $a \in V_G$ under the action of $\Gamma$, meaning that for all $a\in V_G$, $\Gamma(a) \coloneqq\{\gamma \cdot a ~:~\gamma\in \Gamma\}$.
\end{notation}

\begin{definition}
The \textbf{quotient graph} of the action of $\Gamma$ on $G$, denoted by $G/\Gamma$, is the undirected graph such that
$V_{G/\Gamma}=\{\Gamma(a)~:~a\in V_{G}\}$, and $E_{G/\Gamma}=\{(\Gamma(a), \Gamma (b))~:~\exists a',b' \in V_G \ | \ (a',b')\in E_G, \ a'\in \Gamma(a), \ b'\in \Gamma(b)\}.$ 
\end{definition}
\begin{notation}
    Let us consider a graph $G$ and $\Gamma$ a subgroup of $\pi_1(G)$. We denote by $\mathcal{U}_G/\Gamma$ the quotient of $\mathcal{U}_G$ by the action $(\gamma,w) \mapsto \gamma \star w$ of $\Gamma$. The homomorphism $\alpha$ yields a quotient homomorphism $\alpha^\Gamma : \mathcal{U}_G/\Gamma \to G$ by $\alpha^\Gamma(\Gamma(w)) \coloneqq \alpha(w)$ for all $w \in \mathcal{U}_G$. $\alpha^\Gamma$ is well-defined because all the elements in $\Gamma(w)$ end at the same vertex as $w$, and it is a homomorphism as, by definition of the quotient graph, if $(\Gamma(w),\Gamma(w'))\in E_{G/\Gamma}$, then there exist $\gamma,\xi\in \Gamma$ such that $(\gamma \star w, \xi\star w')\in E_{\mathcal U_G}$ which implies that $(\alpha(w), \alpha(w'))\in E_{G}$.
\end{notation}

\begin{proposition}
    For any undirected graph $G$, we have that $G$ is isomorphic to $\mathcal{U}_G / \pi_1(G)$.
\end{proposition}

\begin{proof}
We already know that the action of $\pi_1(G)$ is free and transitive on $\alpha^{-1}(v)$ for all vertices $v\in G$. Thus it follows that map $\alpha^{\pi_1(G)}$ is a graph homomorphism which is bijective on the vertices. Any edge $(v,w)\in E_{G}$ lifts to an edge $(v', w')\in E_{\mathcal U_G}$, and then $(\pi_1(G)(v'), \pi_1(G)(w'))\in E_{\mathcal U_G/\pi_1(G)}$ satisfies $\alpha^{\pi_1(G)}(\pi_1(G)(v'), \pi_1(G)(w')) = (v,w)$. This shows that $\alpha^{\pi_1(G)}$ is a graph isomorphism.
\end{proof}

More generally, covers of $G$ can be constructed as quotients of the universal cover by subgroups of its fundamental group $\pi_1(G)$; this is a particular case of the classification theorem (Theorem 1.38 in \cite{Hatcher}). This provides a way to generate covers different from the universal cover. 

\begin{theorem}\label{thm:subgroup-cover}
    For any subgroup $\Gamma\subset \pi_1(G)$, $\alpha^\Gamma : \mathcal U_{G}/\Gamma \rightarrow G$ is a covering map, which makes $\mathcal U_{G}/\Gamma$ a cover of $G$. 
    Furthermore, $\pi_1(\mathcal U_{G}/\Gamma)$ is isomorphic to $\Gamma$. 
\end{theorem}

\begin{proof}
    As discussed above, the map $\alpha^\Gamma$ is a graph homomorphism. Let us prove that it is a local isomorphism. Fix some $b\in V_G$ and distinct walks $w, w'\in U_G$ such that they terminate at $b$. Let us see why $N_{\mathcal U_G/\Gamma}(\Gamma(w))$ and $N_{\mathcal U_G/\Gamma}(\Gamma(w'))$ are disjoint. If not, suppose that they have a common vertex $\Gamma(w'')$. Then there are $\gamma, \xi, \gamma', \xi'\in\Gamma$ such that $(\gamma \star w , \xi \star w'') \in E_{\mathcal{U}_G}$ and $(\gamma' \star w' , \xi' \star w'') \in E_{\mathcal{U}_G}$. 
     Since $\alpha : \mathcal{U}_G \rightarrow G$ is a covering map and $\alpha(\xi \star w'')=\alpha(\xi' \star w'')$, we have $\gamma \star w = \gamma' \star w'$, so $\Gamma(w) = \Gamma(w')$.

We are left to prove that for all $w \in \mathcal U_{G}$, $\alpha^\Gamma$ is a bijection from $N_{\mathcal U_{G}/\Gamma}(\Gamma(w))$ to $N_G(\alpha(w))$. It is surjective because $\Gamma(w)$ is a neighbor of $\Gamma(\varphi(wb))$ for all $b \in N_G(\alpha(w))$. Let us see that it is injective. Consider a walk $w'$ ending at $b$ such that $(\gamma \star w, \xi \star w') \in E_{\mathcal{U}_G}$ for some $\gamma,\xi\in\Gamma$. Then $\xi \star w' = \varphi(\gamma \star wb)$, which implies that $w' = (\xi^{-1}\gamma)\star (wb)$, so $\Gamma(w') = \Gamma(\varphi(wb))$.
\end{proof}

Furthermore, every cover of $G$ is the quotient of the universal cover by a subgroup of the fundamental group: 

\begin{remark}
    \label{theorem: covering maps and subgroups}
    Let $G$ be a graph, $\tilde G$ a cover of $G$ 
    and $\theta : \tilde G \rightarrow G$ a covering map. Then there is a natural embedding of $\Gamma=\pi_1(\tilde G)$ in $\pi_1(G)$ and $\mathcal U_G/\Gamma$ is isomorphic to $\tilde G$. 
\end{remark}
The proof of this fact is a little more involved but here is a brief sketch. Since any cycle in $\tilde G$ projects to a cycle in $G$, we get a group homomorphism from $\pi_1(\tilde G)$ to $\pi_1(G)$. Since cycles in $G$ have unique lifts in $\tilde G$ up to a choice of base vertex we have that the group homomorphism into $\pi_1(G)$ is injective. Thus we can henceforth identify $\pi_1(\tilde G)$ as a subgroup of $\pi_1(G)$. Finally by careful bookkeeping one can show that $\mathcal U_G/\pi_1(\tilde G)$ is isomorphic to $\tilde G$.

\subsection{Regular covers}\label{section:regular.covers}

In this section, we introduce the notion of regular cover. For these covers, the action by deck transformations has properties which are useful in the remainder of the article.

\begin{definition}\label{definition.regular.cover}
    A \textbf{regular} cover of an undirected graph $G$ is a cover of the form $\mathcal{U}_G/\Gamma$, where $\Gamma\unlhd \pi_1(G)$.
\end{definition}

\begin{remark}
In particular, the universal cover is regular.
\end{remark}

We fix for this section a normal subgroup $\Gamma\unlhd \pi_1(G)$. 

\begin{notation}
    We define an action of $\Gamma$ on $\mathcal{U}_{G / \Gamma}$ by setting, for every $h  \in \pi_1(G)$,
\[\eta_{h\Gamma} : \begin{array}{rcl}
\mathcal{U}_{G/\Gamma} &\to& \mathcal{U}_{G/\Gamma}\\
\Gamma(w)&\mapsto&\Gamma(\eta_h(w)).\end{array}\]
\end{notation}

To see that these functions are well defined, we need to check that for all $w \in \mathcal{U}_G$, $h \in \pi_1(G)$ and $g \in \Gamma$, $\Gamma(\eta_h(\eta_g(w))) = \Gamma(\eta_h(w))$. This is true because, since the subgroup $\Gamma$ is normal, there exists $g' \in \Gamma$ such that $hg = g' h$. Thus $\Gamma(\eta_h(\eta_g(w)))  = \Gamma(\eta_{hg}(w)) = \Gamma(\eta_{g'}(\eta_{h}(w))) = \Gamma(\eta_h(w))$.
\begin{theorem}\label{theorem: transitive action of the regular covers}
For every $h\in \pi_1(G)$, 
    the map $\eta_{h\Gamma}$ is a deck transformation for $\alpha^{\Gamma} : \mathcal{U}_G/\Gamma \rightarrow G$. 
    The map $(h\Gamma,\Gamma(w)) \mapsto \eta_{h\Gamma}(\Gamma(w))$ is a free action of the group on {$\mathcal{U}_G/\Gamma$} which is transitive on $(\alpha^\Gamma)^{-1}(b)$ for all $b\in V_G$. Finally, $(\mathcal{U}_G/\Gamma)/(\pi_1(G)/\Gamma)$ is isomorphic to the graph $G$.
\end{theorem}

\begin{proof}[Sketch of proof]
The group $\pi_1(G)$ acts freely and transitively on each preimage by $\alpha$ of a vertex in $\mathcal U_G$.
Thus for all $b\in V_G$, $\pi_1(G)$ acts transitively on the set $\{\Gamma(w) : w \in \mathcal{U}_G, \alpha(w)=b\}$, and the stabilizer of this action is $\Gamma$.
This induced action preserves adjacency and from this it follows that $\pi_1(G)/\Gamma$ on $\mathcal U_G/\Gamma$ acts by deck transformations. It is both transitive and free on the preimage of $b$ for the covering map. The uniqueness follows from Proposition \ref{prop: uniqueness of deck transformation}. Finally using the transitivity of the map we can conclude that ${ (\mathcal{U}_G/\Gamma)}/(\pi_1(G)/\Gamma)$ is isomorphic to the graph $G$.
\end{proof}

\section{\label{section.square.quotients}Quotienting by squares}

Let us recall that a square in $G$ is a non-backtracking cycle of length $4$. 
This means that it can be written as $a _0 a_1 a_2 a_3 a_4$ with $a_0 = a_4$, $a_0\neq a_2$ and $a_1\neq a_3$. 

\begin{definition}
Given a covering map $\theta : \tilde G \rightarrow G$ and $x\in X^d_G$,
a configuration $\tilde x \in X^d_{\tilde G}$ is called a \textbf{lift} of $x$ for $\theta$ if $\theta(\tilde x_{\boldsymbol{u}})=x_{\boldsymbol{u}}$ for all $\boldsymbol{u}\in \Z^d$.
\end{definition}

\begin{notation}
For all walks $p$ in $\mathbb{Z}^d$, and $x$ a configuration of $X^d_G$ we denote by $x_p$ the walk $x_{p_0} \ldots x_{p_{l(p)}}$.
\end{notation}

In this section, for all graph $G$ and integer $d>1$, we formulate a necessary and sufficient condition for a cover $\tilde G$ to be such that every configuration of $X^d_G$ admits a lift in $X^d_{\tilde{G}}$ (for $\alpha^{\Gamma}$, where $\Gamma$ is such that $\tilde G$ is isomorphic to $\pi_1(G) / \Gamma$): this is possible exactly when every square of $G$ can be lifted to a square of $\tilde G$ (which is independent of the choice of lift for the base vertex). 
The universal cover $\mathcal{U}_G$ does not always satisfy this: for example, for $G=C_4$ (cycle graph with $4$ vertices), $\mathcal{U}_{C_4}$ is the graph of $\mathbb{Z}$ and contains no square. Therefore the squares in $G$ do not admit any lift. On the other hand, $G$, as a cover of itself, trivially satisfies the condition. We define the square cover of $G$ as the largest cover of $G$ for which this property is true.

{The above fact} is not surprising: we know from standard algebraic topology \cite[Chapter 2]{Hatcher} that in order for a configuration $x\in X^d_G$ with $x_{\boldsymbol{0}}=a$ to have a lift $\tilde x\in X_{\mathcal U_G[a]}$, the natural map $x_\star:\pi_1(\Z^d)[0]\to \pi_1(G)[a]$ induced by this configuration must be constant with value $\mathbbm{1}_{\pi_1(G)[a]}$, which is a strong constraint. 

The intuition which underlies the definition of the square group comes from the very origins of algebraic topology. It can be found for instance in \cite[Proposition \textbf{1.26}]{Hatcher} and \cite{rotman1973covering}. The same idea was used in the context of searching for continuous factors from the free part of the full shift in \cite{gao2018continuous} as mentioned in the introduction. We will see that it is related (Section \ref{section:sq.cover}) to the square cover, a notion introduced in \cite{GHO23}. All of this indicates that this is an important object of study and we expect many more connections with the dynamics beyond what has been mentioned and will be explored in a forthcoming paper.

\subsection{The square group} \label{Section: Square group cover}

\subsubsection{Definition}

\begin{notation}
      Let us denote by $\Delta(G)[a]$ the subgroup of $\pi_1(G)[a]$ generated by elements of the form $p \star s \star {p^{-1}}$, where $p$ is a non-backtracking walk and $s$ is a square. As all $\Delta(G)[a]$ are isomorphic, we denote \notationidx{$\Delta(G)$}{the subgroup of $\pi_1(G)$ generated by elements of the form $p \star s \star {p^{-1}}$, where $p$ is a non-backtracking walk and $s$ is a square
      } its isomorphic class and call the elements of $\Delta(G)$ \textbf{square-decomposable} cycles of $G$.
\end{notation}

\begin{definition}\label{def:sqgroup}
    Let us denote by $\pi_1^{\square}(G)[a]$ the quotient of $\pi_1(G)[a]$ by $\Delta(G)[a]$. The groups $\pi^\square_1(G)[a]$, $a \in G$ are all isomorphic to  \notationidx{$\pi_1^{\square}(G)$}{the square group of an undirected graph $G$}$ \coloneqq \pi_1(G)/\Delta(G)$, which we call \textbf{square group} of $G$.
\end{definition}

\begin{remark}\label{remark: cycle in delta group}
    It follows from Definition \ref{def:sqgroup} that a cycle $c\in \pi_1(G)$ is an element of $\Delta(G)$ if and only if
\[c=c_1\star c_2 \star \ldots\star c_n,\quad \text{where}\quad c_i=p_i\odot s_i \odot (p_i)^{-1}\] with each $p_i$ a non-backtracking walk and $s_i$ a square.
\end{remark}

\subsubsection{Finite presentation}

Let us see that the square group is finitely presented by exhibiting a particular set of generators and relations. Fix a spanning tree $T$ of $G$ and a base vertex $a$ of $G$.

\begin{notation}
    For any square $s$ in $G$, we denote by \notationidx{$\Delta_T^a(s)$}{notation for the cycle $p_T^a(s_0) \star s \star (p_T^a(s_0))^{-1}$} the cycle $p_T^a(s_0) \star s \star (p_T^a(s_0))^{-1}$. Here as well, we will omit the base vertex $a$ from the notation when it is clear in context.
\end{notation}

\begin{proposition}\label{prop:snormsubgroup}
    The group $\Delta(G)[a]$ is the smallest normal subgroup of $\pi_1(G)[a]$ containing the cycles $\Delta^a_T(s)$, where $s$ is a square of $G$.
\end{proposition}
\begin{proof}
    The fact that $\Delta(G)[a]$ is a normal subgroup simply comes from the associativity of $\star$. 
    Furthermore, $\Delta(G)[a]$ contains all the cycles of the form $\Delta^a_T(s)$. Every normal subgroup $H$ of $\pi_1(G)[a]$ containing these cycles also contains the cycles 
    \[ p \star \Delta^a_T(s)\star p^{-1} = (p \star p^a_T(s_0))) \star s \star (p \star p^a_T(s_0))^{-1},\]
    where $p \in \pi_1(G)[a]$. Since every non-backtracking walk on $G$ can be written as $p\star p^a_T(b)$, where $p \in \pi_1(G)[a]$ and $b$ is a vertex of $G$, $H$ contains  $\Delta(G)[a]$.
\end{proof}

For simplicity we say that a word $e_0e_1e_2e_3$ on alphabet $E_G$ is a square if $e_i = (s_i,s_{i+1})$ for some square $s$. Denote by $R^{\square}_T(G)$ the union of $R_T(G)$ with all squares of $G$.

\begin{theorem}\label{thm:REforSquare}
    For any finite undirected graph $G$ and $T$ a spanning tree of $G$, we have  \[\pi^\square_1(G)\cong\langle E_G:R^{\square}_T(G)\rangle \cong\langle E_G \backslash E_T: R^{\square}_T(G)\backslash E_T\rangle.\]
    As a consequence $\pi_1^{\square}(G)$ is finitely presented.
\end{theorem}

\begin{proof}
We use the notations introduced in Section~\ref{sec:fundamental_group}.
The group $\pi_1(G)$ is generated by the cycles $\beta_T((u,v))$, where $(u,v)$ is an edge of $G$ which is not in $T$, and the only relations are $\beta_T((u,v))\beta_T((v,u)) = 1$ for all $(u,v) \in E_G$.
From Proposition \ref{prop:snormsubgroup}, a presentation of $\pi_1^{\square}(G)$ is obtained from this presentation of $\pi_1(G)$ by adding the relations $\Delta_T(s) = 1$ for all squares $s$. These relations are equivalent to the relations $\beta_T\left(\Delta_T(s)\right) = 1$, which can be rewritten as 
\[\beta_T((s_0, s_1)) \beta_T((s_1, s_2)) \beta_T((s_2 ,s_3)) \beta_T((s_3, s_4)) = 1.\]
Since there are finitely many squares in $G$, this provides a finite presentation of $\pi_1^{\square}(G)$. The statement is yielded by rewriting $\beta_T((u,v))$ as $(u,v)$.
\end{proof}

\begin{example}
    Figure~\ref{fig.example.square.group} provides examples for the definition of square group. The graph $C_4$ (first from the top) has trivial square group; this is the case whenever all the cycles in the graph are square-decomposable (that is, elements of $\Delta(G)$). The square group of the second graph is equal to its fundamental group, because this graph does not have any squares. The square group of the third graph is $\mathbb Z/2\mathbb Z$, hence nontrivial and different from the fundamental group.
\end{example}

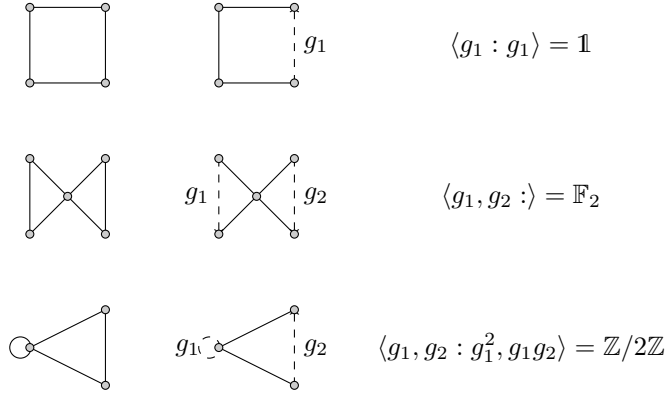
\begin{figure}[ht!]
    \centering
        \begin{tikzpicture}[scale=0.5]
        \draw (0,0) rectangle (2,2);
        \draw[fill=gray!40] (0,0) circle (3pt);
        \draw[fill=gray!40] (2,2) circle (3pt);
        \draw[fill=gray!40] (2,0) circle (3pt);
        \draw[fill=gray!40] (0,2) circle (3pt);

        \begin{scope}[yshift=-3cm]
        \draw (1,0) -- (2,1) -- (2,-1) -- (1,0) -- (0,1) -- (0,-1) -- (1,0);
        \draw[fill=gray!40] (1,0) circle (3pt);
        \draw[fill=gray!40] (2,1) circle (3pt);
        \draw[fill=gray!40] (2,-1) circle (3pt);
        \draw[fill=gray!40] (0,-1) circle (3pt);
        \draw[fill=gray!40] (0,1) circle (3pt);
        \end{scope}

        \begin{scope}[yshift=-7cm]
        \draw (-0.25,0) circle (8pt);
        \draw (0,0) -- (2,1) -- (2,-1) -- (0,0);
        \draw[fill=gray!40] (0,0) circle (3pt);
        \draw[fill=gray!40] (2,1) circle (3pt);
        \draw[fill=gray!40] (2,-1) circle (3pt);
        \end{scope}

        \begin{scope}[xshift=5cm]
        \draw (2,0) -- (0,0) -- (0,2) -- (2,2);
        \draw[dashed, ->] (2,0) -- (2,2) node[midway, right] {$g_1$};
        \draw[fill=gray!40] (0,0) circle (3pt);
        \draw[fill=gray!40] (2,2) circle (3pt);
        \draw[fill=gray!40] (2,0) circle (3pt);
        \draw[fill=gray!40] (0,2) circle (3pt);
        \node at (8,1) {$\langle g_1 : g_1\rangle = \mathds{1}$};
        \end{scope}

        \begin{scope}[xshift=5cm,yshift=-3cm]
        \draw (0,1) -- (2,-1);
        \draw (0,-1) -- (2,1);
        \draw[dashed, ->] (0,1) -- (0,-1) node[midway, left] {$g_1$};
        \draw[dashed, ->] (2,1) -- (2,-1) node[midway, right] {$g_2$};
        \draw[fill=gray!40] (1,0) circle (3pt);
        \draw[fill=gray!40] (2,1) circle (3pt);
        \draw[fill=gray!40] (2,-1) circle (3pt);
        \draw[fill=gray!40] (0,-1) circle (3pt);
        \draw[fill=gray!40] (0,1) circle (3pt);
        \node at (8,0) {$\langle g_1, g_2: \rangle=\mathbb{F}_2$};
        \end{scope}

        \begin{scope}[xshift=5cm,yshift=-7cm]
        \draw[dashed] (-0.25,0) circle (8pt) node[left] {$g_1$};
        \draw (0,0) -- (2,1);
        \draw (2,-1) -- (0,0);
        \draw[dashed, ->] (2,-1) -- (2,1) node[midway, right] {$g_2$};
        \draw[fill=gray!40] (0,0) circle (3pt);
        \draw[fill=gray!40] (2,1) circle (3pt);
        \draw[fill=gray!40] (2,-1) circle (3pt);
        \node at (8,0) {$\langle g_1,g_2: g_1^2, g_1g_2\rangle = \mathbb{Z}/2\mathbb{Z}$};
        \end{scope}
    \end{tikzpicture}
    \caption{Illustration for the definition of square group. Left: the graph $G$. Middle: a choice of spanning tree (full edges). Right: the corresponding presentation for the square group of $G$ (where we denote by $\mathds{1}$ the trivial group).}\label{fig.example.square.group}
\end{figure}

\subsection{The square cover\label{section:sq.cover}}

\subsubsection{Definition}

The graphs $\mathcal{U}_G^{\square}[a] = \mathcal U_G[a]/\Delta(G)[a]$, $a \in V_G$, are all isomorphic. We call \notationidx{$\mathcal{U}_G^{\square}$}{the square cover of an undirected graph $G$} the \textbf{square cover} of the graph $G$, omitting the base vertex from the notation. It is the cover of the graph $G$ corresponding to the subgroup $\Delta(G)\unlhd \pi_1(G)$.
We provide a more elementary description at the end of this section.

Since $\Delta(G)$ is a normal subgroup of $\pi_1(G)$ (Proposition~\ref{prop:snormsubgroup}), $\mathcal{U}_G^{\square}$ is a regular cover of $G$.
We denote by \notationidx{$\alpha^{\square}$}{quotient covering map for the square cover}$:\mathcal U^\square_G\to G$ the covering map obtained by quotienting the covering map $\alpha:\mathcal U_G\to G$. The following proposition is immediate from Proposition \ref{theorem: transitive action of the regular covers}.

\begin{proposition}\label{Proposition: square cover finite square group}
    The square group of $G$ is finite if and only if the square cover of $G$ is finite.
\end{proposition}

\subsubsection{Square lifting}

Let us see (Lemma \ref{lemma: lift of squares}) that the square cover of a graph $G$ is the largest cover in which a square of $G$ always has a lift which is a square. This is an important property for the remainder of the text and the root of the configuration lifting property which was instrumental in \cite{GHO23}. We prove a stronger version of this property in Section \ref{section:config.lifting}.

\begin{lemma}\label{lemma: automorphism taking to adjacent vertex}
    Let $G$ be a graph. Fix a vertex $p$ of $\mathcal U_G$. Consider a square $s$ which begins and end at $\alpha(p)$. There is a unique element $\tilde s\in \pi_1(G)$ such that 
    $p\star s ={\tilde s \star p
    }$. Furthermore, $\tilde s \in \Delta(G)$.
\end{lemma}

\begin{proof}
The cycle $\tilde s=p\star s\star p^{-1}$
 satisfies the requirement. For uniqueness, if $\tilde s \star p = \tilde s' \star p$, then $\tilde s \star p \star p^{-1} = \tilde s' \star p \star p^{-1}$, and thus $\tilde s = \tilde s'$. By definition of $\Delta(G)$, we have $\tilde s \in \Delta(G)$. 
\end{proof}

\begin{remark}
    Lemma \ref{lemma: automorphism taking to adjacent vertex} holds when $s$ is a cycle of length four (not necessarily a square). 
\end{remark}

\begin{lemma}\label{lemma: lift of squares}
    Let $G$ be a graph and $\Gamma$ be a subgroup of $\pi_1(G)$. Set $\tilde G \coloneqq\mathcal U_G/\Gamma$. We have $\Delta(G)\subset \Gamma $ if and only if for every square $s$ in $G$
    and any lift $\tilde s_0$ in $\tilde G$ of $s_0$,  the lift $\tilde s$ in $\tilde G$ of $s$ starting at $\tilde s_0$ is a square. 
\end{lemma}

\begin{proof}Fix a base vertex $a$.

    $(\Rightarrow)$ Let us assume that $\Delta(G) \subset \Gamma$. Fix any square $s$ and vertex $\tilde s_0$ as in the lemma, and let $\tilde s$ be the unique lift of $s$ starting at $\tilde s_0$ (see Theorem~\ref{thm:subgroup-cover}). Let $w_0$ be a walk in $\mathcal{U}_G$ such that $\tilde s_0 = \Gamma(w_0)$ and set $w_4 = w_0 \star s$. By Lemma \ref{lemma: automorphism taking to adjacent vertex}, 
    there exists $s' \in \Delta(G)$ (which by assumption implies that $s' \in \Gamma$) 
    such that $w_4 = s' \star w_0$. Since lifts are unique, we thus have $\Gamma(w_4) = \tilde s_4$. This implies that $ \Gamma(w_0) = \Gamma(w_4)$. In turn, this means that the unique lift $\tilde s$ of $s$ in $\tilde G$ which begins at $\tilde s_0$ is a square. 
    
    $(\Leftarrow)$ Conversely, let us assume that for every square $s$ in $G$ and any lift $\tilde s_0 \in V_{\tilde G}$ of $s_0$ in $\tilde G$, there exists a lift $\tilde s$ of $s$ in $\tilde G$ starting at $\tilde s_0$ which is a square. In order to prove that $\Delta(G) \subset \Gamma$, 
    it is sufficient by Proposition \ref{prop:snormsubgroup} to prove that for all squares $s'$, $\Gamma$ contains the cycle $\Delta_T(s')$. By assumption, the unique lift of the square $s'$ in $\tilde G$ starting at $\Gamma(p_T(s'_0))$ is a square. Thus $\Gamma(p_T(s'_0)) = \Gamma(p_T(s'_0) \star s')$. Therefore there exists $g \in \Gamma$ such that $g \star p_T(s'_0) = p_T(s'_0) \star s'$. This implies that $g = \Delta_T(s')$, which gives $\Delta_T(s')\in \Gamma$.
\end{proof}

\subsubsection{Square equivalence}

In this section, we prove the configuration lifting property which, roughly speaking, states that configurations of a homshift $X_G^d$ can be lifted to configurations of the homshift $X_{\mathcal{U}_G^{\square}}^d$. This is implied by the stronger statement of Proposition \ref{proposition:config.lifting}. Let us first introduce some terminology.

\begin{definition}\label{def:diff_sq}
    We say that two non-backtracking walks $p,q$ on $G$ which start at the same vertex and end at the same vertex \textbf{differ by a square} when there exists a non-backtracking walk $w$ and a square $s$  in $G$ starting at $w_{l(w)}$ such that $p \star q^{-1} = w \star s \star w^{-1}$.
    \label{definition:differ.square}
\end{definition}

\begin{remark}
    Notice that when $p \star q^{-1} = w \star s \star w^{-1}$, we also have $q \star p^{-1} = w \star s^{-1} \star w^{-1}$. This makes differing by a square a symmetric relation.
\end{remark}

The next two lemmas imply that Definition \ref{definition:differ.square} is equivalent to the definition of differing by a square from \cite{GHO23}. We use the notion of the \emph{circular shift} \notationidx{$\omega(p)$}{circular shift of a cycle $p$}$= p_1 \ldots p_{l(p)} p_1$ of a cycle $p$.

\begin{lemma}
    For every square $s$ and non-backtracking walk $w$ such that $s_0 = w_{l(w)}$, there is a prefix $p$ of $w$ and an integer $i$ such that $w \star s \star w^{-1} = p \odot \omega^i(s) \odot p^{-1}$.\label{lemma:square.induction}
\end{lemma}

\begin{proof}
    If $w \odot s \odot w^{-1}$ is non-backtracking, the statement is immediately satisfied. Otherwise, because both $s$ and $w$ are non-backtracking, we have
    $s_1 = w_{l(w)-1}$.
    Then $w\star s \star w^{-1} = w'\star \omega(s) \star w'^{-1}$, where $w' = w_0 \ldots w_{l(w)-1}$ (by deleting the backtrack corresponding to indices $l(w)-1,l(w),l(w)+1$). The statement is obtained by applying this transformation inductively, since $w$ is a finite word.
\end{proof}

For a walk $p$, a cycle $c$ and $k \le l(p)$ such that $p_k=c_0$, denote by $p \oplus_k c$ the walk $p_{\llbracket 0 , k\rrbracket} \odot c \odot p_{\llbracket k , l(p)\rrbracket}$.
Recall that for every walk $p$, $\varphi(p)$ is the walk obtained from $p$ by replacing successively each backtrack $aba$ by $a$.

\begin{lemma}
Two non-backtracking walks $p$ and $q$ on a graph $G$ differ by a square if and only there exist some square $s$ and an integer $k$ such that $p = \varphi(q \oplus_k s)$ or $q = \varphi(p \oplus_k s)$.
    \label{definition:differ by a square}
\end{lemma}

\begin{proof}
$(\Rightarrow)$ Consider $p,q$ non-backtracking walks such that there exists a square $s$ and a non-backtracking walk $w$ such that $p \star q ^{-1} = w \star s \star w^{-1}$. This implies that $p$, $q$ and $w$ start at the same vertex. We thus have 
\begin{equation}
    p = \varphi(p) = w \star s \star w^{-1}\star q.\label{eq:pwsw}
\end{equation} By Lemma \ref{lemma:square.induction}, there exists a prefix $w'$ of $w$ and an integer $i$ such that $w \star s \star w^{-1} = w' \odot \omega^i(s) \odot {w'}^{-1}$. Set $l\coloneqq\min(l(w'),l(q))$, and  $\lambda \coloneqq l\left({w'}^{-1}_{\llbracket 0,l\rrbracket}\star q_{\llbracket 0,l\rrbracket}\right)/2$. In other words, $l-\lambda$ is the length of the largest common prefix of $w'$ and $q$. Therefore Equation~(\ref{eq:pwsw}) becomes 
\[p = q_{\llbracket 0,l - \lambda \rrbracket} \star  \left( {w'}_{\llbracket l - \lambda, l(w')\rrbracket} \odot \omega^i(s) \odot {w'}_{\llbracket l - \lambda, l(w')\rrbracket}^{-1}\right) \star q_{\llbracket l - \lambda, l(q)\rrbracket}.\]

 On the one hand, if $l-\lambda = l(w')$, 
$p = q_{\llbracket 0,l - \lambda \rrbracket} \star \omega^i(s) \star q_{\llbracket l - \lambda, l(q)\rrbracket}$,
which can be rewritten as $p = \varphi(q \oplus_{l-\lambda} \omega^i(s))$.
On the other hand, if $l-\lambda \neq l(w')$, this implies that 
\[p = q_{\llbracket 0,l - \lambda \rrbracket} \odot  \left( {w'}_{\llbracket l - \lambda, l(w')\rrbracket}  \odot \omega^i(s)\odot {w'}_{\llbracket l - \lambda, l(w')\rrbracket}^{-1}\right) \star q_{\llbracket l - \lambda, l(q)\rrbracket}\]
so that $q = \varphi(p\oplus_{l(w')} {\omega^i(s)}^{-1})$.

$(\Leftarrow)$ 
Let us consider $p,q$ non-backtracking walks such that $q = \varphi(p\oplus_k s)$ for some integer $k$ and a square $s$ (the other case is similar). Then there exists a non-backtracking walk $w$ and an integer $l$ such that \[p = q_{\llbracket 0,l\rrbracket} \odot w \odot s^{-1} \odot w^{-1}\odot q_{\llbracket l,l(q)\rrbracket}.\]
This implies that 
\[p = w' \star s^{-1} \star (w')^{-1} \star q,\]
where $w'\coloneqq q_{\llbracket 0,l\rrbracket} \odot w$.
\end{proof}

Figure~\ref{figure.differ} illustrates the types of situations where two walks differ by a square.
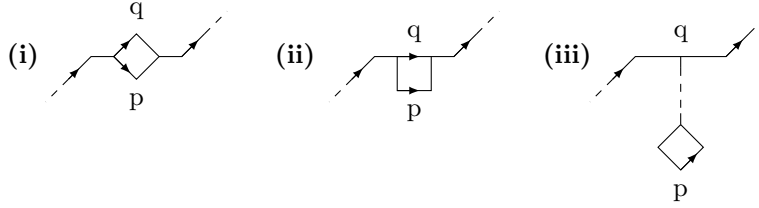
\begin{figure}[ht!]

\begin{center}
\begin{tikzpicture}[scale=0.3]
\draw (0,0) -- (1,1) -- (2,1) -- (3,2) -- (4,1) -- (5,1) -- (6,2);
\draw[-latex] (0,0) -- (0.5,0.5);
\draw[-latex] (5,1) -- (5.75,1.75);
\draw[-latex] (2,1) -- (2.75,1.75);
\draw[-latex] (2,1) -- (2.75,0.25);

\draw (2,1) -- (3,0) -- (4,1);
\draw[dashed] (6,2) -- (7,3);
\draw[dashed] (0,0) -- (-1,-1);
\node at (3,-1) {p};
\node at (3,3) {q};
\node at (-2,1) {\textbf{(i)}};

\begin{scope}[xshift=12cm]
\draw (0.5,0) -- (1.5,1) -- (2.5,1) -- (4,1) -- (5,1) -- (6,2);
\draw[-latex] (0.5,0) -- (1,0.5);
\draw[-latex] (2.5,1) -- (3.5,1);
\draw[-latex] (2.5,-0.5) -- (3.5,-0.5);
\draw[-latex] (5,1) -- (5.75,1.75);
\draw (2.5,1) -- (2.5,-0.5) -- (4,-0.5) -- (4,1);
\draw[dashed] (6,2) -- (7,3);
\draw[dashed] (0.5,0) -- (-0.5,-1);
\node at (3.25,-1.5) {p};
\node at (3.25,2) {q};
\node at (-2,1) {\textbf{(ii)}};
\end{scope}

\begin{scope}[xshift=24cm]
\draw (0,0) -- (1,1) -- (2,1) -- (4,1) -- (5,1) -- (6,2);
\draw[-latex] (0,0) -- (0.5,0.5);
\draw[-latex] (5,1) -- (5.75,1.75);
\begin{scope}[yshift=-3cm]
\draw (3,1) -- (2,0) -- (3,-1) -- (4,0) -- (3,1);
\draw[-latex] (3,-1) -- (3.75,-0.25);
\end{scope}
\draw (3,1) -- (3,0.5);
\draw[dashed] (3,0.5) -- (3,-1.5);
\draw (3,-1.5) -- (3,-2);
\draw[dashed] (6,2) -- (7,3);
\draw[dashed] (0,0) -- (-1,-1);
\node at (3,-5) {p};
\node at (3,2) {q};
\node at (-2,1) {\textbf{(iii)}};
\end{scope}
\end{tikzpicture}
\end{center}
\caption{Partial representation of two walks $p,q$ which differ by a square.\label{figure.differ}}
\end{figure}

Equation \eqref{eq:pwsw} motivates the following definition.

\begin{definition}
    We say that two non-backtracking walks $p,q$ on a graph $G$ both starting at some vertex $a$ and ending at the same vertex are \textbf{square-equivalent} when $p \star q^{-1} \in \Delta(G)[a]$. In other terms, $p$ and $q$ have the same orbit via the action of $\Delta(G)[a]$.  We denote this by \notationidx{$p \sim_{\square} q$}{$p,q$ are square-equivalent}. 
\end{definition}
Clearly the relation $\sim_{\square}$ is an equivalence relation. The following proposition is an immediate consequence of this definition.

\begin{proposition}\label{corollary.def.square-cover}
The square cover $\mathcal U^{\square}_{G}[a]$ is isomorphic to the quotient of $\mathcal U_G[a]$ by the equivalence relation $\sim_\square$.
\end{proposition}

Recall that this quotient is defined as follows: vertices are the equivalence classes of $\mathcal U_G[a]$ for $\sim_{\square}$; there is an edge between $c$ and $c'$ if and only if there are $p \in c$ and $p' \in c'$ such that $p$ and $p'$ are neighbors in $\mathcal U_G[a]$.

The following proposition, together with Lemma \ref{definition:differ by a square}, implies that definition of the square equivalence relation is equivalent to the definition of square equivalence written in \cite{GHO23}, and subsequently, as consequence of Proposition \ref{corollary.def.square-cover}, the same goes for the definition of square cover. 

\begin{proposition}
Two non-backtracking walks $p,q$ starting at some vertex $a$ in $G$ are square-equivalent if and only if there exists a sequence $p_0 , \ldots , p_k$ of non-backtracking walks such that for all $i$, $p_i$ and $p_{i+1}$ differ by a square, $p_0 = p$ and $p_k = q$.
    \label{proposition: square equivalent}
\end{proposition}
\begin{proof} Consider two walks $p,q$. By definition of $\Delta(G)[a]$, we have $p\star q^{-1} \in \Delta(G)[a]$ if and only if $p\star q^{-1} = \overline{s_1}\star\cdots \star\overline{s_n}$, for some $n\in\N$, where for all $i$, $\overline{s_i} = w_i\star s_i \star w_i^{-1}\in \Delta(G)[a]$ for some square $s_i$ and some non-backtracking walk $w_i$ starting at $a$. This is equivalent to the existence of two sequences $(p_i)^n_{i=0}$ and $(\overline{s_i})^n_{i=0}$ such that $p_0 = p$, $p_{i+1} = (\overline{s_i})^{-1}\star p_i  $ for all $i$, and $p_n = q$, which is equivalent to the existence of a sequence $(p_i)^n_{i=0}$ such that $p = p_0$, $q = p_n$, and for all $i$, $p_{i+1}$ and $p_i$ differ by a square.
\end{proof}

\subsubsection{Configuration lifting\label{section:config.lifting}}

\begin{proposition}\label{proposition:config.lifting}
Let $G$ be a graph and $\Gamma\subset \pi_1(G)$ be a subgroup. Set $\tilde G \coloneqq\mathcal U_G/\Gamma$ and let $\theta: \tilde G\to G$ be a covering map.
Then $\Delta(G)\subset \Gamma$ if and only if, for all $x\in X^d_{G}$ and all lift $w$ of $x_{\boldsymbol{0}}$ to $\tilde G$, there exists a lift $\tilde x\in X^d_{\tilde G}$ of $x$ for $\theta$ such that $\tilde x_{\boldsymbol{0}} = w$.
\end{proposition}
\begin{proof}
Denote by $\tilde \theta: \mathcal U_G\to \tilde G$ the quotient map and observe that $\alpha=\theta\circ\tilde \theta$,
where $\alpha$ is the covering map for the universal cover $\mathcal U_G$ of $G$.

$(\Rightarrow)$ 
Assume that $\Delta(G) \subset \Gamma$.  Fix $x\in X^d_{G}$. We want to find $\tilde x\in X^d_{\tilde G}$ such that {$\theta(\tilde x_{\boldsymbol{u}}) =  x_{\boldsymbol{u}}$  for every $\boldsymbol{u} \in \mathbb{Z}^d$.} 
\paragraph{Definition of the lift.} Set $b \coloneqq x_{\boldsymbol 0}$ and choose $\tilde b\in V_{\tilde G}$ such that $\alpha(\tilde b)=b$. Now define $\tilde x \in (V_{\tilde G})^{\Z^d}$ in the following way. Fix some $\boldsymbol m \in \Z^d$. Choose any walk $p$ from $\boldsymbol{0}$ to $\boldsymbol{m}$ in $\Z^d$.
By the walk lifting property, the walk $x_{p}$ admits a unique lift $\tilde b_0 \tilde b_1 \ldots \tilde b_k$ in the graph $\tilde G$ such that $\tilde b_0 = \tilde b$. We set $\tilde x_{\boldsymbol{m}}  \coloneqq \tilde b_k$.
\paragraph{Independence from the choice of $p$.} Let us prove that this definition does not depend on the choice of walk from $\boldsymbol{0}$ to $\boldsymbol{m}$. Consider two such walks $p$ and $p'$. Since they are square-equivalent in $\Z^d$, $x_{p}$ and $x_{p'}$ are also square-equivalent in $G$. Furthermore, by Lemma \ref{definition:differ by a square} and Lemma \ref{lemma: lift of squares}, the lifts of two non-backtracking walks which differ by a square also differ by a square, therefore the lifts of non-backtracking square-equivalent walks are square-equivalent. Since by definition two square-equivalent walks end at the same vertex, the definition of $\tilde x_{\boldsymbol{m}}$ does not depend on the choice of the walk $p$.

\paragraph{$\boldsymbol{\tilde x \in X^d_{\tilde G}}.$} Indeed, for $\boldsymbol{m}$ and $\boldsymbol{m}'$ which are neighbors in $\mathbb{Z}^d$, there exist two non-backtracking walks $p$ and $p'$ in $\mathbb{Z}^d$ starting at $\boldsymbol{0}$ and ending respectively at $\boldsymbol{m}$ and $\boldsymbol{m}'$ such that one is a prefix of the other. This implies that one of the lifts of $x_p$ and $x_{p'}$ is prefix of the other, which implies that $\tilde x_{\boldsymbol{m}}$ and $\tilde x_{\boldsymbol{m}'}$ are neighbors in $\tilde G$. 

\paragraph{$(\Leftarrow)$} Let us prove this by contraposition. Assume that $\Delta(G)\setminus \Gamma\neq\emptyset$. 

 By Lemma \ref{lemma: lift of squares}, there is a square $s$ in $G$ and $\tilde s_0\in V_{\tilde G}$ such that $\theta(\tilde s_0)=s_0$ and the lift of $s$ in $\tilde G$ starting at $\tilde s_0$ is not a square.  
 The configuration $x\in X^2_G$ presented on Figure \ref{figure:config.spec} cannot have a lift whose value at $\boldsymbol{0}$ is $\tilde s_0$, for otherwise $s$ would have a lift starting at $\tilde s_0$ which is a square. To obtain such a configuration in $X_G^d$, $d>2$, consider $x'\in X_G^d$ whose restriction to the first two dimensions is equal to $x$ and is constant along all the other dimensions.
 \begin{figure}
 \begin{center}
            \begin{tikzpicture}[scale=0.35]
        \fill[gray!5] (-1,-1) rectangle (9,9);
        \fill[gray!50] (3,3) rectangle (5,5);
        \fill[gray!30] (0,8) rectangle (1,7);
        \fill[gray!30] (2,6) rectangle (1,7);
        \fill[gray!30] (2,6) rectangle (3,5);
        \fill[gray!30] (5,5) rectangle (6,6);
        \fill[gray!30] (7,7) rectangle (6,6);
        \fill[gray!30] (7,7) rectangle (8,8);
        \fill[gray!30] (5,3) rectangle (6,2);
        \fill[gray!30] (6,2) rectangle (7,1);
        \fill[gray!30] (7,1) rectangle (8,0);
        \fill[gray!30] (0,0) rectangle (1,1);
        \fill[gray!30] (2,2) rectangle (1,1);
        \fill[gray!30] (2,2) rectangle (3,3);
        \draw (0,0) grid (8,8);
        \node at (4,9.5) {$\vdots$};
        \node at (4,-.5) {$\vdots$};
        \node at (-1,4) {$\hdots$};
        \node at (9,4) {$\hdots$};
        \node[scale=0.8] at (0.5,1.5) {$s_0$};
        \node[scale=0.8] at (0.5,3.5) {$s_0$};
        \node[scale=0.8] at (0.5,5.5) {$s_0$};
        \node[scale=0.8] at (1.5,2.5) {$s_0$};
        \node[scale=0.8] at (1.5,4.5) {$s_0$};
        \node[scale=0.8] at (2.5,3.5) {$s_0$};
        \node[scale=0.8] at (0.5,2.5) {$s_3$};
        \node[scale=0.8] at (0.5,4.5) {$s_3$};
        \node[scale=0.8] at (0.5,6.5) {$s_3$};
        \node[scale=0.8] at (1.5,3.5) {$s_3$};
        \node[scale=0.8] at (1.5,5.5) {$s_3$};
        \node[scale=0.8] at (2.5,4.5) {$s_3$};

        \node[scale=0.8] at (1.5,0.5) {$s_2$};
        \node[scale=0.8] at (3.5,0.5) {$s_2$};
        \node[scale=0.8] at (5.5,0.5) {$s_2$};
        \node[scale=0.8] at (2.5,1.5) {$s_2$};
        \node[scale=0.8] at (4.5,1.5) {$s_2$};
        \node[scale=0.8] at (3.5,2.5) {$s_2$};
        \node[scale=0.8] at (2.5,0.5) {$s_3$};
        \node[scale=0.8] at (4.5,0.5) {$s_3$};
        \node[scale=0.8] at (6.5,0.5) {$s_3$};
        \node[scale=0.8] at (3.5,1.5) {$s_3$};
        \node[scale=0.8] at (5.5,1.5) {$s_3$};
        \node[scale=0.8] at (4.5,2.5) {$s_3$};

        \node[scale=0.8] at (1.5,7.5) {$s_1$};
        \node[scale=0.8] at (3.5,7.5) {$s_1$};
        \node[scale=0.8] at (5.5,7.5) {$s_1$};
        \node[scale=0.8] at (2.5,6.5) {$s_1$};
        \node[scale=0.8] at (4.5,6.5) {$s_1$};
        \node[scale=0.8] at (3.5,5.5) {$s_1$};
        \node[scale=0.8] at (2.5,7.5) {$s_0$};
        \node[scale=0.8] at (4.5,7.5) {$s_0$};
        \node[scale=0.8] at (6.5,7.5) {$s_0$};
        \node[scale=0.8] at (3.5,6.5) {$s_0$};
        \node[scale=0.8] at (5.5,6.5) {$s_0$};
        \node[scale=0.8] at (4.5,5.5) {$s_0$};

        \node[scale=0.8] at (7.5,1.5) {$s_1$};
        \node[scale=0.8] at (7.5,3.5) {$s_1$};
        \node[scale=0.8] at (7.5,5.5) {$s_1$};
        \node[scale=0.8] at (6.5,2.5) {$s_1$};
        \node[scale=0.8] at (6.5,4.5) {$s_1$};
        \node[scale=0.8] at (5.5,3.5) {$s_1$};
        \node[scale=0.8] at (7.5,2.5) {$s_2$};
        \node[scale=0.8] at (7.5,4.5) {$s_2$};
        \node[scale=0.8] at (7.5,6.5) {$s_2$};
        \node[scale=0.8] at (6.5,3.5) {$s_2$};
        \node[scale=0.8] at (6.5,5.5) {$s_2$};
        \node[scale=0.8] at (5.5,4.5) {$s_2$};
        
        \node[scale=0.8] at (0.5,7.5) {$s_0$};
        \node[scale=0.8] at (1.5,6.5) {$s_0$};
        \node[scale=0.8] at (2.5,5.5) {$s_0$};
        \node[scale=0.8] at (3.5,4.5) {$s_0$};
        \node[scale=0.8] at (4.5,4.5) {$s_1$};
        \node[scale=0.8] at (4.5,3.5) {$s_2$};
        \node[scale=0.8] at (3.5,3.5) {$s_3$};
        \node[scale=0.8] at (5.5,5.5) {$s_1$};
        \node[scale=0.8] at (6.5,6.5) {$s_1$};
        \node[scale=0.8] at (7.5,7.5) {$s_1$};
        \node[scale=0.8] at (5.5,2.5) {$s_2$};
        \node[scale=0.8] at (6.5,1.5) {$s_2$};
        \node[scale=0.8] at (7.5,0.5) {$s_2$};
        \node[scale=0.8] at (0.5,0.5) {$s_3$};
        \node[scale=0.8] at (1.5,1.5) {$s_3$};
        \node[scale=0.8] at (2.5,2.5) {$s_3$};
        \end{tikzpicture}
    \end{center}
    \caption{A configuration that cannot be lifted to $\tilde G$, provided a square $s$ in $G$ which cannot be lifted to a square in $\tilde G$.}
     \label{figure:config.spec}
     \end{figure}
\end{proof}

This proposition can be made slightly stronger: to have $\Delta(G)\subset \Gamma$, it is equivalent that every configuration has a lift (regardless of the value at the origin).

\begin{proposition}\label{ proposition:config.lifting.2}Under the notations of the previous proposition,
if every $x\in X^d_{G}$ has a lift for $\theta$ to $X^d_{\tilde G}$, then $\Delta(G)\subset \Gamma$. 
\end{proposition}

\begin{proof}
For the sake of contradiction assume that $\Delta(G)\not\subset \Gamma$ and $d=2$. The proof carries forward verbatim to higher dimensions. By Lemma \ref{lemma: lift of squares}, there is a square $s$ in $G$ and $\tilde s_0\in V_{\tilde G}$ such that $\theta(\tilde s_0)=s_0$ and the lift of $s$ in $\tilde G$ starting at $\tilde s_0$ is not a square. Now consider the pattern
\[p=\begin{smallmatrix}
    s_0&s_1&s_0\\
    s_3&s_0&s_3\\
    s_0&s_1&s_0
\end{smallmatrix}.\]
Recall the notation for the covering map from the universal cover $\alpha:\mathcal U_G\to G$. It follows easily that for all $s_0'\in V_{\mathcal U_G}$ for which $\alpha(s_0')=s_0$, there is a unique lift of $p$ to $\mathcal U_G$ given by
\[ p^{s_0'}=\begin{smallmatrix}
    s_0'&s_1'&s_0'\\
    s_3'&s_0'&s_3'\\
    s_0'&s_1'&s_0'
\end{smallmatrix}.\]
By transitivity of $X^2_{\mathcal U_G}$, there exists $x'\in X^2_{\mathcal U_G}$ such that all patterns $p^{s_0'}$ appear in $x'$. Note that $\mathcal U_G$ can be infinite, in which case $X^2_{\mathcal U_G}$ is not a homshift, but the proof of \cite[Proposition 3.1]{MR3743365} (transitivity for homshifts) applies to infinite connected graphs. 

By Theorem \ref{theorem: transitive action of the regular covers} any lift of $\alpha(x')$ to $\mathcal U_G$ is of the form $\eta(x')$ for some $\eta\in \pi_1(G)$. Since $\eta$ is an automorphism of $\mathcal U_G$ it follows that all the patterns of the form $ p^{s_0'}$ must appear in $\eta(x')$, and thus in each lift of $\alpha(x')$ to $\mathcal U_G$. 

Now for all $\tilde t_0\in V_{\tilde G}$ for which $\theta(\tilde t_0)= s_0$, we know that there is a unique lift $\tilde p^{\tilde t_0}$ of $p$ to $\tilde G$ given by 
\[\tilde p^{\tilde t_0}=\begin{smallmatrix}
    \tilde t_0&\tilde t_1&\tilde t_0\\
    \tilde t_3&\tilde t_0&\tilde t_3\\
    \tilde t_0&\tilde t_1&\tilde t_0
\end{smallmatrix}.\]

Let $\alpha': \mathcal U_G\to \tilde G$ be the covering map. Since $\theta\circ \alpha'=\alpha$ we have that all possible patterns of the type $\tilde p^{\tilde t_0}$ appear in each lift $\tilde x\in X^2_{\tilde G}$ of $\alpha(x')$ to $\tilde G$.

Now consider $x\in X^2_G$ where we replace the middle $s_0$ by $s_2$ in every appearance of $p$ in $\alpha(x')$. Since the square $s$ does not lift to a square in $\tilde G$ starting at $\tilde s_0$, it follows that $x$ does not have a lift to $X^2_{\tilde G}$.
\end{proof}

The following corollary says that the square cover is the ``maximum'' cover to which homomorphisms from $\Z^d$ to the graph can be lifted to.
\begin{corollary}\label{corolarry.square.covering.map}
  Let $G$ be a graph and $\tilde G$ be a cover of $G$ and $\theta: \tilde G\to G$ be a covering map. Suppose that for all $x\in X_G^d$ there exists a lift $\tilde x\in X_{\tilde G}^d$, that is, $\theta\cdot \tilde x=x$. Then there exists a covering map $\theta': \mathcal U_G^{\square}\to \tilde G$ such that $\theta\circ \theta'=\alpha^{\square}$.
\end{corollary}

\begin{proof}
This follows from the correspondence of subgroups of the fundamental group and covers of the graph $G$.
\end{proof}

\section{\label{section.undecidability}Undecidability}

In this section we prove the main Theorem \ref{thm:main}. For this we will prove and use the following realization theorem which already appeared in \cite[Lemma 4.4.3]{gao2018continuous}. They provided a brief sketch of a proof and we decided to give a more explicit construction with a self-contained argumentation. 

\begin{restatable}{theorem}{realization}
\label{thm:everygroup}
 Given any finitely presented group $\langle E:R\rangle$ one can algorithmically construct a (bipartite) graph $G$ such that 
    \[\pi_1^{\square}(G) \cong \langle E:R\rangle.\]
\end{restatable}

One key observation of \cite{gao2018continuous} on which the proof of Theorem \ref{thm:everygroup} relies is that the notion of square group of a graph has a natural interpretation in the framework of algebraic topology. In particular, any graph $G$ can be transformed naturally into a two-dimensional CW-complex by interpreting $G$ as a one-dimensional CW-complex and attach to each square in $G$ a copy of $[0,1]^2$. The square group of $G$ is the fundamental group of this CW-complex. This is an old idea and proofs which rely on it go back at least to \cite{rotman1973covering, Hatcher}.

In this section, after providing some intuition of why Theorem \ref{thm:everygroup} is true [Section \ref{section.example}], we propose a fully detailed proof [Section \ref{section.proof.realization}], following 
the presentation in \cite[Section 1.2]{Hatcher} and drawing inspiration from \cite{rotman1973covering}. Our proof relies in particular on an adapted version of Van Kampen theorem [Section \ref{section.van.kampen}] and the definition of a natural class of graphs having trivial square group [Section \ref{section.hopscotch}].
We then discuss consequences of this theorem in terms of undecidability of dynamical properties of homshifts in Section \ref{subsection: undecidable}. 

\subsection{\label{section.example}Some intuition for Theorem \ref{thm:everygroup}}
Let us begin by presenting a simple graph construction whose square group is $\mathbb Z/3\mathbb Z = \langle g : g^3\rangle$, since this is the simplest group for which there is no such trivial construction. We then provide some informal explanations for the general case, hoping that it is convincing enough for readers who do not want to go through the whole proof of the theorem.

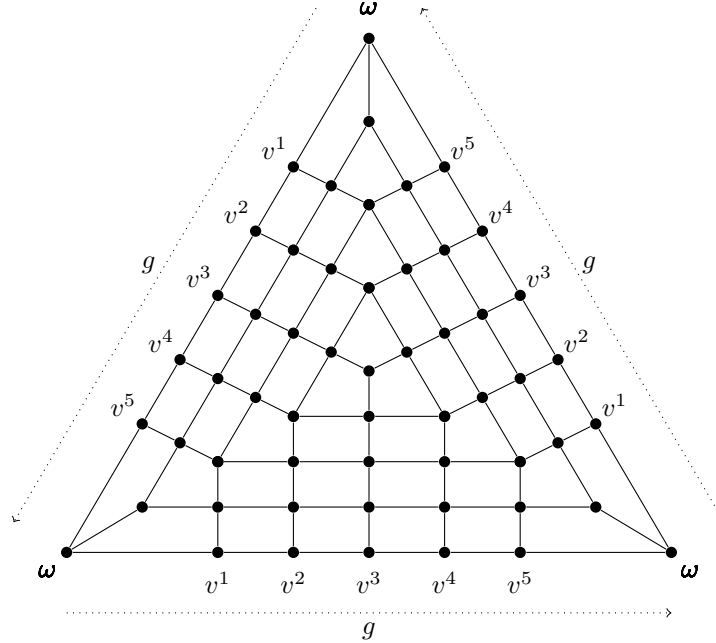
\begin{figure}[ht!]
\begin{center}
\begin{tikzpicture}[scale=1]
\draw [->, dotted] (-1, -1.4) -- (7,-1.4) node [midway, below] {$g$};
\draw [->, dotted] (7.7, -.2) -- (3.7,6.6) node [midway, right] {$g$};
\draw [->, dotted] (2.3, 6.6) -- (-1.7,-.2) node [midway, left] {$g$};

   \foreach \part/\name/\x/\y in {1/O/-1/-.6, 1/1/1/-.6, 1/2/2/-.6, 1/3/3/-.6, 1/4/4/-.6, 1/5/5/-.6, 2/O/7/-.6, 2/1/6/1.1, 2/2/5.5/1.95, 2/3/5/2.8, 2/4/4.5/3.65, 2/5/4/4.5, 3/O/3/6.2, 3/1/2/4.5, 3/2/1.5/3.65, 3/3/1/2.8, 3/4/.5/1.95, 3/5/0/1.1}
   {
        \node[fill, circle, inner sep=1.5pt](o\part\name) at (\x, \y) {};
        \ifthenelse{\equal{\name}{O}}{}{
        \ifthenelse{\equal{\part}{1}}
           {\node at (\x, \y-.4) {$v^\name$};}{}
        \ifthenelse{\equal{\part}{2}}
           {\node at (\x+.25, \y+.25) {$v^\name$};}{}
        \ifthenelse{\equal{\part}{3}}
           {\node at (\x-.25, \y+.25) {$v^\name$};}{}
        }
        \node at (-1.25,-.85) {$\omega$};
        \node at (7.25,-.85) {$\omega$};
        \node at (3,6.6) {$\omega$};
   }
   \draw (o1O) -- (o11) -- (o12) -- (o13) -- (o14) -- (o15) -- (o2O) -- (o21) -- (o22) -- (o23) -- (o24) -- (o25) -- (o3O) -- (o31) -- (o32) -- (o33) -- (o34) -- (o35) -- (o1O);

   \foreach \part/\name/\x/\y in {1/O/0/0, 1/1/1/0, 1/2/2/0, 1/3/3/0, 1/4/4/0, 1/5/5/0, 2/O/6/0, 2/1/5.5/0.85, 2/2/5/1.7, 2/3/4.5/2.55, 2/4/4/3.4, 2/5/3.5/4.25, 3/O/3/5.1, 3/1/2.5/4.25, 3/2/2/3.4, 3/3/1.5/2.55, 3/4/1/1.7, 3/5/.5/0.85}
   {
        \node[fill, circle, inner sep=1.5pt](\part\name) at (\x, \y) {};
        \draw (\part\name) -- (o\part\name);
   }
   \draw (1O) -- (11) -- (12) -- (13) -- (14) -- (15) -- (2O) -- (21) -- (22) -- (23) -- (24) -- (25) -- (3O) -- (31) -- (32) -- (33) -- (34) -- (35) -- (1O);
   
\foreach \part/\name/\x/\y in {11/1/1/0.6, 11/2/2/0.6, 11/3/3/0.6, 11/4/4/0.6, 11/5/5/0.6, 22/1/5/0.6, 22/2/4.5/1.45, 22/3/4/2.3, 22/4/3.5/3.15, 22/5/3/4, 33/1/3/4, 33/2/2.5/3.15, 33/3/2/2.3, 33/4/1.5/1.45, 33/5/1/0.6}
    {\node[fill, circle, inner sep=1.5pt](\part\name) at (\x, \y) {};
    \ifthenelse{\equal{\part}{11}}
           {\draw (1\name) -- (11\name);}{}
        \ifthenelse{\equal{\part}{22}}
           {\draw (2\name) -- (22\name);}{}
        \ifthenelse{\equal{\part}{33}}
           {\draw (3\name) -- (33\name);}{}
   }
   \draw (111) -- (112) -- (113) -- (114) -- (221) -- (222) -- (223) -- (224) -- (331) -- (332) -- (333) -- (334) -- (111);
    \foreach \part/\name/\x/\y in {111/2/2/1.2, 111/3/3/1.2, 111/4/4/1.2, 222/2/4/1.2, 222/3/3.5/2.05, 222/4/3/2.9, 333/2/3/2.9, 333/3/2.5/2.05, 333/4/2/1.2}
    {\node[fill, circle, inner sep=1.5pt](\part\name) at (\x, \y) {};
     \ifthenelse{\equal{\part}{111}}
           {\draw (11\name) -- (111\name);}{}
        \ifthenelse{\equal{\part}{222}}
           {\draw (22\name) -- (222\name);}{}
        \ifthenelse{\equal{\part}{333}}
           {\draw (33\name) -- (333\name);}{}
   }
   
   \draw (1112) -- (1113) -- (2222) -- (2223) -- (3332) -- (3333) -- (1112);
    \node[fill, circle, inner sep=1.5pt](R) at (3,1.8) {};
    \draw (1113) -- (R) (2223) -- (R) (3333) -- (R); 
\end{tikzpicture}
\end{center}
\caption{\label{fig:Z3}A bipartite graph whose square group is $\mathbb Z/3\Z$. All vertices with the same label are identified.}
\end{figure}

In Figure~\ref{fig:Z3}, any cycle beginning and ending at $\omega$ is square-equivalent to a cycle that only goes through vertices $\omega$ and $v^1$ to $v^5$. Up to removing backtracks, such a cycle is the concatenation of copies of either $g = \omega v^1\cdots v^5 \omega$ or $g^{-1} = \omega v^5\cdots v^1 \omega$. In particular, since removing squares and backtracks maintains parity, all the cycles are of even length, which implies that the graph is bipartite. By construction, the cycle $g^3$ is square-equivalent to the trivial cycle, and it should be intuitive that there is no other relation, so that the square group is $\mathbb Z/3\Z$.

Now let us give a very informal description for how one may proceed in a more general situation. Consider a finitely presented group $\langle E : R \rangle$. Start from a graph $G$ with a single vertex $\omega$. For each generator $g \in E$, we add to this graph a simple cycle of length 6 starting and ending at $\omega$, whose vertices which are different from $\omega$ are denoted by $v^1(g)$ to $v^5(g)$. At this point, the square-group is a free group with $|E|$ generators. For each relation $r = g_1^{\varepsilon_1}\dots g_n^{\varepsilon_n} \in R$, where $\varepsilon_k = \pm 1$, do the following:

\begin{itemize}
\item create a graph $G'$ which consists in a cycle of length $6n$ that we think as an $n$-gon whose sides are of length 6;
\item add to $G'$ a quadrangulation of this cycle (we leave to the reader to see that it is always possible, for instance following Figure~\ref{fig:Z3}); 
\item glue $G'$ onto $G$ in such a way that each corner of the $n$-gon is identified with $\omega$ and the $k$th side of this $n$-gon is identified vertex by vertex with $v^1(g_k) , \ldots , v^5(g_k)$ when $\varepsilon_k = 1$ and with $v^5(g_k) , \ldots , v^1(g_k)$ when $\varepsilon_k = -1$.
\end{itemize}

This defines a bipartite graph (for the same reason as above) whose square group is $\langle E : R \rangle$.

\subsection{Van Kampen theorem for square groups\label{section.van.kampen}}

This section is devoted to statement of a version of Van Kampen theorem adapted to square groups [Theorem \ref{theorem.vankampen}]. As mentioned in the introduction of Section~\ref{section.undecidability}, if the graph has no self-loops then the square group is isomorphic to the fundamental group of an associated CW-complex \cite{gao2018continuous}. Using this we could translate the standard Van Kampen theorem from algebraic topology to our setting and use it in our proofs later. However we want to keep the text as self-contained as possible and also cover the case when the graph has a self-loop. For this reason and to save the reader the inconvenience of translation, we decided to present a short argument proving what we need.

\begin{notation}
Given two graphs $G_1$ and $G_2$, let us define $G_1\cup G_2$ the union graph by $V_{G_1\cup G_2}=V_{G_1}\cup V_{G_2}$ and $E_{G_1\cup G_2}= E_{G_1}\cup E_{G_2}$ and $G_1\cap G_2$ the intersection graph by $V_{G_1\cap G_2}=V_{G_1}\cap V_{G_2}$ and $E_{G_1\cap G_2}= E_{G_1}\cap E_{G_2}$.
\end{notation}

The following is straightforward: 

\begin{lemma}\label{lemma.union.dec}
    Consider two groups $\Gamma_1 = \langle E_1:R_1\rangle$ and $\Gamma_2 = \langle E_2:R_2\rangle$ such that $E_1 \cap E_2 = \emptyset$. Then $\Gamma_1*\Gamma_2 = \langle E_1\cup E_2:R_1\cup R_2\rangle$.
\end{lemma}

The following is an immediate consequence of Theorem \ref{thm:REforSquare} and Lemma \ref{lemma.union.dec}.
\begin{theorem}[Van Kampen theorem]\label{theorem.vankampen}
Let $(G_{i})_{i=1,\ldots,n}$ be a finite sequence of connected graphs.
Let $T$ be a spanning tree of $\cup_{i=1}^n G_i$ such that 
for all $i$, $T_i \coloneqq T\cap G_i$ is a spanning tree of $G_i$. 
Set $\tilde E= \cup_{i=1}^n E_{G_i}$ and $\tilde R=\cup_{i=1}^n R^{\square}_{T_i}(G_i)\cup R_{\texttt{add}}$, where $R_{\texttt{add}}$ is the set of squares $e_1e_2e_3e_4$, where $e_j$ is in $\tilde E$ for all $j$, which are not contained in any of the graphs $G_i$, $i \in \llbracket 1,n\rrbracket$.  
Then we have:  \[\pi_1^\square(\cup_{i=1}^n G_i) = \langle \tilde E~:~\tilde R\rangle.\]
\end{theorem}

\begin{remark}
    In the literature, Van Kampen theorem is formulated in terms of amalgamated products. Although we do not state it this way, it is possible to formulate Theorem \ref{theorem.vankampen} in those terms.
\end{remark}

\begin{lemma}\label{lemma:tree.extension}
    For every connected graph $H$, every connected graph $\tilde{H}$ which extends it, and every spanning tree $T$ of $H$, there exists a spanning tree $\tilde{T}$ of $\tilde{H}$ such that $T$ is a sub-graph of $\tilde{T}$.
\end{lemma}

\begin{proof}
    We do this as follows. Set $T_0 \coloneqq T$, and assume that we have constructed trees $T_0 , \ldots , T_k$ in $\tilde H$ such that each tree in this sequence is a sub-graph of the next one. Define $T_{k+1}$ as follows. If $T_k$ is not a spanning tree of $\tilde H$, then there is an edge in $\tilde H$ which has exactly one vertex in $T_k$. Indeed, since $\tilde H$ is connected, if there is an edge in $\tilde H$ which does not have any vertex in $T_k$, there is one which has exactly one vertex in $T_k$. When this is the case, $T_{k+1}$ is obtained from $T_k$ by adding this edge and the vertex in it which is not in $T_k$. Otherwise, $T_{k+1} \coloneqq T_k$. Since $H$ is finite, there exists $k$ such that $T_{k+1} = T_k$. Thus $T_k$ is a spanning tree of $\tilde H$ which extends $T$. 
\end{proof}

\begin{lemma}\label{lemma.tree.extension}
    Let $G$ and $(G_i)_{i=1}^n$ be connected graphs such that for all $i \neq j$, $G_i \cap G_j = G$. For every spanning tree $T$ of $G$, there is a spanning tree $\tilde T$ of $\bigcup_{i=1}^n G_i$ which extends $T$ and $\tilde T \cap G_i$ is a spanning tree of $G_i$ for every $i$.
\end{lemma}

\begin{proof}
 We use Lemma \ref{lemma:tree.extension} for $H$ equal to $G$ and for $\tilde H$ equal to $G_i$, for each $i$. The union $\tilde T$ of the obtained trees satisfies the conditions of the statement. By construction $\tilde T \cap G_i$ is a spanning tree of $G_i$ for all $i$. It is straightforward that every vertex of $\cup_{i=1}^n G_i$ is in $\tilde T$. It is a tree because $G_i \cap G_j = G$ for all $i \neq j$. It is thus a spanning tree of $\cup_{i=1}^n G_i$.
\end{proof}

\begin{remark}
    As a consequence, the tree $\tilde T$ provided by Lemma~\ref{lemma.tree.extension} satisfies the conditions in Theorem \ref{theorem.vankampen}. 
\end{remark}

\begin{corollary}\label{corollary: wedge product}
    Let $(G_i)_{i=1}^n$ a finite sequence of graphs such that there is some vertex $v$ in the union $\cup_{i=1}^n G_i$ such that for all $i \neq j$, $G_i\cap G_j$ is the graph with only vertex $v$ and no edge. Assume that each square of $\cup_{i=1}^n G_i$ is contained in $G_j$ for some $j$.
    Then \[\pi_1^\square(\cup G_i)\cong *_{i=1}^n\pi_1^\square(G_i).\]
\end{corollary}
\begin{remark}
    The assumption that each square of $\cup_{i=1}^n G_i$ is contained in $G_j$ for some $j$ is meant to avoid a square made of a triangle and a self-loop, or two self-loops on adjacent vertices, belonging to two different $G_i$, which may still appear with a single shared vertex.
    \end{remark}
\begin{proof}
Fix a spanning tree $T$ of the graph $\cup_{i=1}^n G_i$ and set $T_i \coloneqq G_i\cap T$ for all $i$. Since all $G_i$ share a single vertex, $T_i$ is a spanning tree of $G_i$. For all $i \neq j$, $E_{G_i}$ is disjoint from $E_{G_j}$.
Since for each {square}
of $\cup_{i=1}^n G_i$ there is some $j$ such that it is contained in $G_j$, the set $R_{\texttt{add}}$ in Theorem \ref{theorem.vankampen} is empty. Thus by an application of this theorem: 
\[\pi_1^\square(\cup_{i=1}^n G_i)\cong \langle\cup_{i=1}^n E_{G_i} : \cup_{i=1}^n R^{\square}_{T_i}(G_i)\rangle\cong *_{i=1}^n\pi_1^\square(G_i).\qedhere\]
\end{proof}

 Let us discuss some key examples.

 \begin{example}\label{example.self-loop}
     Let $G$ be a bipartite graph and the graph $\tilde G$ be obtained from $G$ by adding a self-loop on some vertex in $G$. Since $G$ is bipartite, all its cycles have even length, which implies that every square of $\tilde G$ is contained in $G$. Thus Corollary \ref{corollary: wedge product} implies that: 
         \[\pi_1^\square(\tilde G)\cong \pi_1^\square(G)*\Z/2\Z.\]
\end{example}

\begin{notation}
    For all $n > 0$, denote by $C_n$ the $n$-cycle graph defined by $V_{C_n} = \mathbb{Z}/n\mathbb{Z}$ and $E_{C_n} =\left\{(k,k \pm 1) : k \in \mathbb{Z}/n\mathbb{Z}\right\}$.
\end{notation}

\begin{example}
    Suppose that for all $i$, $G_i$ is isomorphic to $C_{N_i}$ for some $N_i > 4$, and that there is a vertex $v$ such that for all $i \neq j$, $G_i\cap G_j$ is the graph with unique vertex $v$ and no edge. Then $\cup_{i=1}^n G_i$ does not contain any square, which implies that:
         \[\pi_1^\square(\cup G_i)\cong *_{i=1}^n\Z\cong \mathbb{F}_n.\]
 \end{example}

\subsection{Flat quadrangulations\label{section.hopscotch}}

A graph $G$ is said to be \emph{planar} when it has a planar embedding, which is a family of maps $p=(p_V,(p_e)_{e\in E})$ such that $p_{V} : V_G \to \mathbb R^2$ is injective and for each $e\in E_G$, $p_e$ is a continuous map $[0,1] \rightarrow \mathbb R^2$ such that: $p_{(v,w)}$ starts at $p_V(v)$ and ends at $p_V(w)$, $p_{(v,w)}(t) = p_{(w,v)}(1-t)$ for all $t$, and the image of $p_e$ and $p_{e'}$ for $e\neq e'^{\pm 1}$ can only intersect on their extremities. A \textit{plane} graph is a tuple $(G,p)$, where $G$ is a planar graph and $p$ is a planar embedding of $G$.

Every plane graph divides $\mathbb R^2$ into parts that are called \emph{faces}. More precisely, faces are the connected components of the space obtained by removing the images of the planar embedding from $\mathbb R^2$. Exactly one of these faces is unbounded
and is called the \emph{external face}. The other ones are called \emph{internal faces}. When a planar embedding is specified, we associate to each face of the plane 
graph its \emph{border}, which is the subgraph of all vertices and edges of $G$ whose images by the embedding are adjacent to this face. We denote by $\partial G$ the border of the external face.

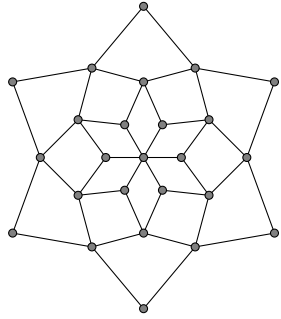
\begin{figure}[ht!]
    \centering
        \begin{tikzpicture}[scale=0.5]
        \begin{scope}[xshift=-10cm]
    \draw (0,0) -- (0.5,0.866) -- (0,2) -- (-0.5,0.866) --(0,0) -- (-1,0) -- (-1.732,-1) -- (-0.5,-0.866) -- (0,0) -- (0.5,-0.866) -- (1.732,-1) -- (1,0) -- (0,0);
\draw (0.5,0.866) -- (1.732,1) -- (1,0);
\draw (-0.5,0.866) -- (-1.732,1) -- (-1,0);
\draw (-0.5,-0.866) -- (0,-2) -- (0.5,-0.866);
\draw (0,2) -- (1.366,2.366) -- (0,4) -- (-1.366,2.366) -- (0,2);
\draw (0,-2) -- (1.366,-2.366) -- (0,-4) -- (-1.366,-2.366) -- (0,-2);
\draw (1.732,1) -- (1.366,2.366) -- (3.464,2) -- (2.732,0) -- (1.732,1); 
\draw (-1.732,1) -- (-2.732,0) -- (-3.464,2) -- (-1.366,2.366) -- (-1.732,1);
\draw (1.732,-1) -- (2.732,0) -- (3.464,-2) -- (1.366,-2.366) -- (1.732,-1);
\draw (-1.732,-1) -- (-2.732,0) -- (-3.464,-2) -- (-1.366,-2.366) -- (-1.732,-1);

\draw[fill=gray] (1.366,2.366) circle (3pt);
\draw[fill=gray] (-1.366,2.366) circle (3pt);
\draw[fill=gray] (1.366,-2.366) circle (3pt);
\draw[fill=gray] (2.732,0) circle (3pt);
\draw[fill=gray] (-2.732,0) circle (3pt);
\draw[fill=gray] (-1.366,-2.366) circle (3pt);

\draw[fill=gray] (0,0) circle (3pt);
\draw[fill=gray] (0,2) circle (3pt);
\draw[fill=gray] (0,-2) circle (3pt);
\draw[fill=gray] (1.732,1) circle (3pt);
\draw[fill=gray] (1.732,-1) circle (3pt);
\draw[fill=gray] (-1.732,1) circle (3pt);
\draw[fill=gray] (-1.732,-1) circle (3pt);

\draw[fill=gray] (0,4) circle (3pt);
\draw[fill=gray] (0,-4) circle (3pt);
\draw[fill=gray] (3.464,2) circle (3pt);
\draw[fill=gray] (3.464,-2) circle (3pt);
\draw[fill=gray] (-3.464,2) circle (3pt);
\draw[fill=gray] (-3.464,-2) circle (3pt);
\draw[fill=gray] (0.5,0.866) circle (3pt);
\draw[fill=gray] (1,0) circle (3pt);
\draw[fill=gray] (0.5,-0.866) circle (3pt);
\draw[fill=gray] (-0.5,0.866) circle (3pt);
\draw[fill=gray] (-1,0) circle (3pt);
\draw[fill=gray] (-0.5,-0.866) circle (3pt);

\end{scope}
    \end{tikzpicture}
    \caption{Example of flat quadrangulation.}
    \label{fig:quadrangulation.10}
\end{figure}

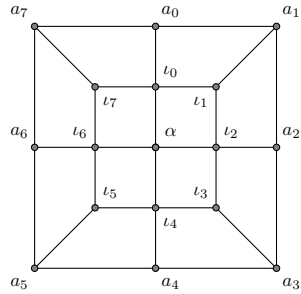
\begin{figure}[ht!]
    \centering
        \begin{tikzpicture}[scale=0.4]

\draw (-2,-2) rectangle (2,2);
\draw (-4,-4) rectangle (4,4);
\draw (-4,-4) -- (-2,-2);
\draw (4,-4) -- (2,-2);
\draw (-4,4) -- (-2,2);
\draw (4,4) -- (2,2);
\draw (-4,0) -- (4,0);
\draw (0,-4) -- (0,4);
\draw[fill=gray] (0,0) circle (3pt);
\draw[fill=gray] (-2,0) circle (3pt);
\draw[fill=gray] (0,-2) circle (3pt);
\draw[fill=gray] (-2,-2) circle (3pt);
\draw[fill=gray] (0,0) circle (3pt);
\draw[fill=gray] (2,0) circle (3pt);
\draw[fill=gray] (0,2) circle (3pt);
\draw[fill=gray] (2,2) circle (3pt);
\draw[fill=gray] (-2,2) circle (3pt);
\draw[fill=gray] (2,-2) circle (3pt);
\draw[fill=gray] (0,-4) circle (3pt);
\draw[fill=gray] (0,4) circle (3pt);
\draw[fill=gray] (-4,0) circle (3pt);
\draw[fill=gray] (4,0) circle (3pt);
\draw[fill=gray] (4,4) circle (3pt);
\draw[fill=gray] (-4,-4) circle (3pt);
\draw[fill=gray] (4,-4) circle (3pt);
\draw[fill=gray] (-4,4) circle (3pt);
\node[scale=0.7] at (0.5,2.5) {$\iota_0$};
\node[scale=0.7] at (0.5,4.5) {$a_0$};
\node[scale=0.7] at (0.5,-2.5) {$\iota_4$};
\node[scale=0.7] at (0.5,-4.5) {$a_4$};
\node[scale=0.7] at (-2.5,0.5) {$\iota_6$};
\node[scale=0.7] at (-4.5,0.5) {$a_6$};
\node[scale=0.7] at (2.5,0.5) {$\iota_2$};
\node[scale=0.7] at (4.5,0.5) {$a_2$};
\node[scale=0.7] at (1.5,1.5) {$\iota_1$};
\node[scale=0.7] at (4.5,4.5) {$a_1$};
\node[scale=0.7] at (1.5,-1.5) {$\iota_3$};
\node[scale=0.7] at (4.5,-4.5) {$a_3$};
\node[scale=0.7] at (-1.5,-1.5) {$\iota_5$};
\node[scale=0.7] at (-4.5,-4.5) {$a_5$};
\node[scale=0.7] at (-1.5,1.5) {$\iota_7$};
\node[scale=0.7] at (-4.5,4.5) {$a_7$};
\node[scale=0.7] at (0.5,0.5) {$\alpha$};
    \end{tikzpicture}
    \caption{Illustration for the proof of Proposition \ref{proposition.quadrangulating.cn}.}
    \label{fig:quadrangulation.11}
\end{figure}

\begin{definition}
A connected planar graph $G$ is called a \textbf{flat quadrangulation} if it has a planar embedding for which borders of all internal faces are squares and every square is the border of some internal face. Figure \ref{fig:quadrangulation.10} provides an example of flat quadrangulation graph.
\end{definition}

\begin{proposition}\label{proposition.quadrangulating.cn}
     For any graph $H$ isomorphic to some $C_n$ with $n\geq 6$ even, there exists a flat quadrangulation $G$ such that $\partial G = H$ and no vertex in $V_G \backslash V_{\partial G}$ has at least two neighbors in $\partial G$.
\end{proposition}

\begin{proof}
    Denote by $a_1, \ldots , a_n$ the vertices of $H$ so that edges are exactly $(a_k , a_{k+1})$ for $k \in \mathbb Z / n\mathbb Z$.    
    Define $G$ from $H$ by adding: vertices $\iota_k$ for $k \in \mathbb Z / n\mathbb Z$; an edge between $a_k$ and $\iota_k$ and an edge between $\iota_k$ and $\iota_{k+1}$ for each $k$; a vertex $\alpha$ and an edge between $\alpha$ and $\iota_k$ for each even $k$.     
    This definition is illustrated on Figure~\ref{fig:quadrangulation.11} for $n =8$. It is straightforward that $G$ satisfies the requirements.
\end{proof}

The following is straightforward:

\begin{lemma}\label{lemma:tree}
Every tree is a flat quadrangulation for all its planar embeddings. Furthermore, a flat quadrangulation is a tree if and only if it has no internal face.
\end{lemma}

\begin{notation}\label{notation.remove.edge}
    Let $G$ be a graph, and $(v,w)$ an edge of $G$. We denote by $G \setminus (v,w)$ the graph obtained from $G$ by removing the edges $(v,w)$ and $(w,v)$. This is illustrated in Figure \ref{fig:remove.edge}.
\end{notation}

\begin{figure}[ht!]
    \centering
    \begin{tikzpicture}[scale=0.3]
\begin{scope}[yshift=-10cm]
    \draw (0,0) -- (0.5,0.866) -- (0,2) -- (-0.5,0.866) --(0,0) -- (-1,0) -- (-1.732,-1) -- (-0.5,-0.866) -- (0,0) -- (0.5,-0.866) -- (1.732,-1) -- (1,0) -- (0,0);
\draw (0.5,0.866) -- (1.732,1) -- (1,0);
\draw (-0.5,0.866) -- (-1.732,1) -- (-1,0);
\draw (-0.5,-0.866) -- (0,-2) -- (0.5,-0.866);
\draw (0,2) -- (1.366,2.366) -- (0,4) -- (-1.366,2.366) -- (0,2);
\draw (0,-2) -- (1.366,-2.366) -- (0,-4) -- (-1.366,-2.366) -- (0,-2);
\draw (1.732,1) -- (1.366,2.366) -- (3.464,2) -- (2.732,0) -- (1.732,1); 
\draw (-1.732,1) -- (-2.732,0) -- (-3.464,2) -- (-1.366,2.366) -- (-1.732,1);
\draw (1.732,-1) -- (2.732,0) -- (3.464,-2) -- (1.366,-2.366) -- (1.732,-1);
\draw (-1.732,-1) -- (-2.732,0) -- (-3.464,-2) -- (-1.366,-2.366) -- (-1.732,-1);

\node[scale=0.8] at (-3.732,2.5) {$\boldsymbol{w}$};
\node[scale=0.8] at (-1.832,2.75) {$\boldsymbol{v}$};

\draw[-latex,dashed] (4.732,0) -- (7.732,0);

\draw[fill=gray] (1.366,2.366) circle (3pt);
\draw[fill=gray] (-1.366,2.366) circle (3pt);
\draw[fill=gray] (1.366,-2.366) circle (3pt);
\draw[fill=gray] (2.732,0) circle (3pt);
\draw[fill=gray] (-2.732,0) circle (3pt);
\draw[fill=gray] (-1.366,-2.366) circle (3pt);

\draw[fill=gray] (0,0) circle (3pt);
\draw[fill=gray] (0,2) circle (3pt);
\draw[fill=gray] (0,-2) circle (3pt);
\draw[fill=gray] (1.732,1) circle (3pt);
\draw[fill=gray] (1.732,-1) circle (3pt);
\draw[fill=gray] (-1.732,1) circle (3pt);
\draw[fill=gray] (-1.732,-1) circle (3pt);

\draw[fill=gray] (0,4) circle (3pt);
\draw[fill=gray] (0,-4) circle (3pt);
\draw[fill=gray] (3.464,2) circle (3pt);
\draw[fill=gray] (3.464,-2) circle (3pt);
\draw[fill=gray] (-3.464,2) circle (3pt);
\draw[fill=gray] (-3.464,-2) circle (3pt);
\draw[fill=gray] (0.5,0.866) circle (3pt);
\draw[fill=gray] (1,0) circle (3pt);
\draw[fill=gray] (0.5,-0.866) circle (3pt);
\draw[fill=gray] (-0.5,0.866) circle (3pt);
\draw[fill=gray] (-1,0) circle (3pt);
\draw[fill=gray] (-0.5,-0.866) circle (3pt);

\begin{scope}[xshift=12.5cm]
    \draw (0,0) -- (0.5,0.866) -- (0,2) -- (-0.5,0.866) --(0,0) -- (-1,0) -- (-1.732,-1) -- (-0.5,-0.866) -- (0,0) -- (0.5,-0.866) -- (1.732,-1) -- (1,0) -- (0,0);
\draw (0.5,0.866) -- (1.732,1) -- (1,0);
\draw (-0.5,0.866) -- (-1.732,1) -- (-1,0);
\draw (-0.5,-0.866) -- (0,-2) -- (0.5,-0.866);
\draw (0,2) -- (1.366,2.366) -- (0,4) -- (-1.366,2.366) -- (0,2);
\draw (0,-2) -- (1.366,-2.366) -- (0,-4) -- (-1.366,-2.366) -- (0,-2);
\draw (1.732,1) -- (1.366,2.366) -- (3.464,2) -- (2.732,0) -- (1.732,1); 
\draw (-1.732,1) -- (-2.732,0) -- (-3.464,2);
\draw (-1.366,2.366) -- (-1.732,1);
\draw (1.732,-1) -- (2.732,0) -- (3.464,-2) -- (1.366,-2.366) -- (1.732,-1);
\draw (-1.732,-1) -- (-2.732,0) -- (-3.464,-2) -- (-1.366,-2.366) -- (-1.732,-1);

\draw[fill=gray] (1.366,2.366) circle (3pt);
\draw[fill=gray] (-1.366,2.366) circle (3pt);
\draw[fill=gray] (1.366,-2.366) circle (3pt);
\draw[fill=gray] (2.732,0) circle (3pt);
\draw[fill=gray] (-2.732,0) circle (3pt);
\draw[fill=gray] (-1.366,-2.366) circle (3pt);

\draw[fill=gray] (0,0) circle (3pt);
\draw[fill=gray] (0,2) circle (3pt);
\draw[fill=gray] (0,-2) circle (3pt);
\draw[fill=gray] (1.732,1) circle (3pt);
\draw[fill=gray] (1.732,-1) circle (3pt);
\draw[fill=gray] (-1.732,1) circle (3pt);
\draw[fill=gray] (-1.732,-1) circle (3pt);

\draw[fill=gray] (0,4) circle (3pt);
\draw[fill=gray] (0,-4) circle (3pt);
\draw[fill=gray] (3.464,2) circle (3pt);
\draw[fill=gray] (3.464,-2) circle (3pt);
\draw[fill=gray] (-3.464,2) circle (3pt);
\draw[fill=gray] (-3.464,-2) circle (3pt);
\draw[fill=gray] (0.5,0.866) circle (3pt);
\draw[fill=gray] (1,0) circle (3pt);
\draw[fill=gray] (0.5,-0.866) circle (3pt);
\draw[fill=gray] (-0.5,0.866) circle (3pt);
\draw[fill=gray] (-1,0) circle (3pt);
\draw[fill=gray] (-0.5,-0.866) circle (3pt);
\end{scope}
\end{scope}

\end{tikzpicture}
    \caption{Illustration of Notation \ref{notation.remove.edge}. The graph on the right is $G \backslash (v,w)$, where $G$ is the graph on its left.}
    \label{fig:remove.edge}
\end{figure}
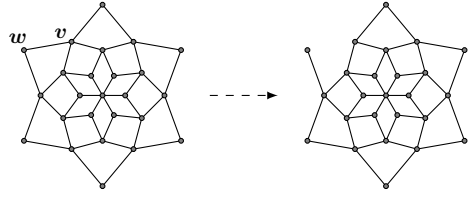

\begin{lemma}\label{lem:reduction:hoph}
    Consider a flat quadrangulation $G$ with fixed planar embedding, and $(v,w)$ an edge which belongs to the border of an external face and the border of an internal face. Then $G \setminus (v,w)$ is a flat quadrangulation, has one less internal face than $G$ and \[\pi_1^\square(G \setminus (v,w) ) \cong \pi_1^\square(G).\]
\end{lemma}

\begin{proof}
It is clear that $G \setminus (v,w)$ is still planar. It is connected because $(v,w)$ belongs to the border of an internal face of $G$. When removing $(v,w)$ from $G$, the internal face is merged into the external face, which means that $G \setminus (v,w)$ has one less internal face than $G$. 
The edge $(v,w)$ belonged to a unique square in $G$, which was the border of the merged internal face. Removing an edge cannot create any square and other internal faces are not affected, thus $G \setminus (v,w)$ is a flat quadrangulation.

Let $T$ be a spanning tree of $G \setminus (v,w)$ (and hence of $G$). The set $R^{\square}_T(G\setminus (v,w))$ is obtained from $R^{\square}_T(G)$ by removing all relations which contain $(v,w)$. The generator $(v,w)$ appears, up to symmetry, in two relations: $(v,w)(w,v)$ and the relation corresponding to the internal face. Applying Lemma \ref{lemma.mod.presentation}, these relations become trivial and other relations are not affected. Therefore: 
\[\langle E_{G}:R^{\square}_T(G)\rangle = \langle E_G\setminus \{(v,w),(w,v)\}:R^{\square}_T(G\setminus (v,w))\rangle.\qedhere\]
\end{proof}

\begin{theorem}\label{lemma: quadrangulation of cycles}
If $G$ is a flat quadrangulation then $\pi^{\square}_1(G)$ is trivial.
\end{theorem}

\begin{proof}
If $G$ has no internal face, then $G$ is a tree, hence $\pi^{\square}_1(G)$ is trivial by definition. If $G$ has $k\geq 1$ internal faces, apply Lemma~\ref{lem:reduction:hoph} to obtain a graph with one less internal face whose square group is isomorphic to the square group of $G$. The result follows by induction.
\end{proof}

\subsection{Proof of Theorem \ref{thm:everygroup}\label{section.proof.realization}}

In this section, we provide a proof of Theorem \ref{thm:everygroup}. To do this, we build a family of graphs from a finite group presentation $(E,R)$ (this notation is introduced in Section \ref{section.background.group.presentations}) such that the square group of any element in this family is isomorphic to $\langle E : R \rangle$ (Theorem \ref{thm:realization}). Furthermore, we provide an algorithm which, provided the presentation $(E,R)$, outputs an element of this family (Proposition \ref{rem:realization}).

The following construction is illustrated in Figure~\ref{fig:illustration.realization.simple}.
We have chosen to leave as much freedom as possible in some choices of parameters, even though this is not necessary to prove Theorem \ref{thm:everygroup}, so that the construction can be more flexible for future work. \bigskip

Let $(E,R)$ be a finite group presentation. If necessary, reduce the presentation: For all $e\in E$ we recursively remove occurrences of $ee^{-1}$ and $e^{-1}e$ from every relation in $R$ and we remove from $E$ any $e\in E\cap R$ (as well as all occurrences of $e$ in relations of $R$), and empty relations. We start the construction of the associated undirected graph, so, whenever we present a construction of a graph, the reader should assume that the reverse edges are present.

\begin{enumerate}
\item Pick an integer $N\geq 6$. For each $a\in E$, define $C^{\omega}_{a,N}$ the $N$-cycle with a base vertex named $\omega$:
    \begin{itemize}
        \item $V_{C^{\omega}_{a,n}} = \{\omega\} \cup (\{a\} \times (\mathbb{Z}/N\mathbb{Z} \backslash \{0\}))$ and 
        \item $E_{C^{\omega}_{a,n}}$ consists of $\{((a,k),(a,k + 1)) : k \in \mathbb{Z}/N\mathbb{Z} \backslash \{0,-1\}\}$, plus $(\omega, (a,1))$ and $((a,-1),\omega)$. 
    \end{itemize}
Denote by $C^{\omega}_{E,N}$ the union of all graphs $C^{\omega}_{a,N}$ for $a \in E$. 

\item For every relation $r$, fix $M_r = N \cdot |r|$ where $|r|$ is the length of the relation seen as a word. Denote by $C_{r,M_r}$ the $M_r$-cycle without base vertex:
    \begin{itemize}
        \item $V_{C_{r,M_r}} = \{r\} \times \mathbb{Z}/M_r\mathbb{Z}$;
        \item $E_{C_{r,M_r}} = \left\{((r,k),(r,k + 1)) : k \in \mathbb{Z}/M_r\mathbb{Z}\right\}$.
    \end{itemize}

\item For every $r  = a_1^{\varepsilon_1} \ldots a_{|r|}^{\varepsilon_{|r|}} \in R$ with $a_i \in E$ and $\varepsilon_i = \pm 1$ for all $i$, denote by $\phi_r : C_{r,M_r} \rightarrow C^{\omega}_{E,N}$ the graph homomorphism such that for all $k$, 
\[\phi_r(r,k) =\begin{cases}  \omega&\text{ if }k\bmod N = 0,\\ 
(a_{\lceil k/N\rceil}, k\bmod N)& \text{ if }\varepsilon_{\lceil k/N\rceil} = +1\text{ and }k\bmod N\neq 0,\\
 (a_{\lceil k/N\rceil}, -k\bmod N) &\text{ if }\varepsilon_{\lceil k/N\rceil} = -1 \text{ and }k\bmod N\neq 0.
 \end{cases}
\]

For every $r \in R$, denote by $\mathcal{H}_r(E,R,N,\boldsymbol{\nu})$ the graph obtained from the union of $C^{\omega}_{E,N}$ with $C_{r,M_r}$ by adding a sequence of edges $\upsilon_r = (\upsilon_{r,k})_{k\leq M_r}$ as follows.

The edge $\upsilon_{r,k}$ is either an edge from $(r,k)$ to $\phi_r(r,k)$, or the diagonal edge from $(r,k+1)$ to $\phi_r(r,k-1)$. Formally, pick $\boldsymbol{\nu} = (\nu_r)_{r \in R}$ a sequence such that $\nu_r \in \{0,1\}^{\Z/M_r\Z}$ for all $r$, and for all $k\in \Z/M_r\Z$, $\upsilon_{r,k}$ is an edge between $(r, k+\nu_{r,k})$ and $\phi_r(r, k - \nu_{r,k})$.

Notice that the sequences $\upsilon_r$ cannot contain multiple copies of the same edge.
This definition is illustrated on Figure \ref{fig:illustration.realization}.

 \item Pick $\textbf{F} = (F_r)_{r \in R}$ a sequence of graphs that are all pairwise disjoint, have trivial square group, contain $C_{r,M_r}$, and do not intersect $C^{\omega}_{E,N}$. 
\end{enumerate}

 Denote by $\mathcal{G}_r(E,R,N,\textbf{F},\boldsymbol{\nu})$ the graph $F_r \cup \mathcal{H}_r(E,R,N,\boldsymbol{\nu})$. The final graph is: 
\[\mathcal{G}(E,R,N,\textbf{F},\boldsymbol{\nu}) \coloneqq \bigcup_{r \in R} \mathcal{G}_r(E,R,N,\textbf{F},\boldsymbol{\nu}).\]

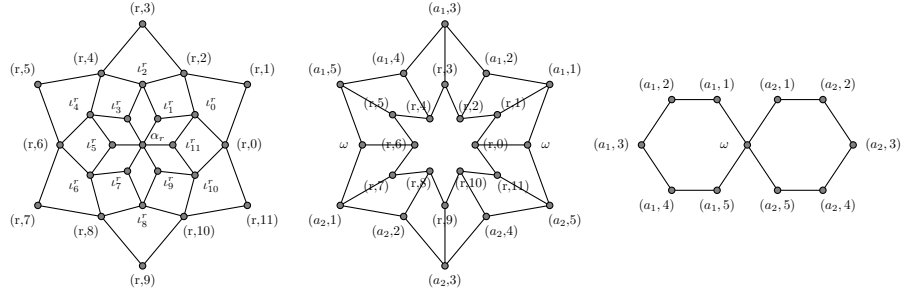
\begin{figure}[ht!]
    \centering
    \begin{tikzpicture}[scale=0.4]
\draw (0.5,0.866) -- (0,2) -- (-0.5,0.866);
\draw (-1,0) -- (-1.732,-1) -- (-0.5,-0.866);
\draw (0.5,-0.866) -- (1.732,-1) -- (1,0);
\draw (0.5,0.866) -- (1.732,1) -- (1,0);
\draw (-0.5,0.866) -- (-1.732,1) -- (-1,0);
\draw (-0.5,-0.866) -- (0,-2) -- (0.5,-0.866);
\draw (0,2) -- (0,4);
\draw (0,-2) -- (0,-4);
\draw (1.366,2.366) -- (0.5,0.866);
\draw (1.366,-2.366) -- (0.5,-0.866);
\draw (-1.366,2.366) -- (-0.5,0.866);
\draw (-1.366,-2.366) -- (-0.5,-0.866);
\draw (1,0) -- (2.732,0);
\draw (-1,0) -- (-2.732,0);
\draw (3.464,2) -- (1.732,1);
\draw (3.464,-2) -- (1.732,-1);
\draw (-3.464,2) -- (-1.732,1);
\draw (-3.464,-2) -- (-1.732,-1);
\draw (1.366,2.366)-- (0,4) -- (-1.366,2.366);
\draw (1.366,-2.366) -- (0,-4) -- (-1.366,-2.366);
\draw (1.366,2.366) -- (3.464,2) -- (2.732,0);
\draw (-2.732,0) -- (-3.464,2) -- (-1.366,2.366);
\draw (2.732,0) -- (3.464,-2) -- (1.366,-2.366);
\draw (-2.732,0) -- (-3.464,-2) -- (-1.366,-2.366);

\draw[fill=gray] (1.366,2.366) circle (3pt);
\draw[fill=gray] (-1.366,2.366) circle (3pt);
\draw[fill=gray] (1.366,-2.366) circle (3pt);
\draw[fill=gray] (2.732,0) circle (3pt);
\draw[fill=gray] (-2.732,0) circle (3pt);
\draw[fill=gray] (-1.366,-2.366) circle (3pt);

\draw[fill=gray] (0,2) circle (3pt);
\draw[fill=gray] (0,-2) circle (3pt);
\draw[fill=gray] (1.732,1) circle (3pt);
\draw[fill=gray] (1.732,-1) circle (3pt);
\draw[fill=gray] (-1.732,1) circle (3pt);
\draw[fill=gray] (-1.732,-1) circle (3pt);

\draw[fill=gray] (0,4) circle (3pt);
\draw[fill=gray] (0,-4) circle (3pt);
\draw[fill=gray] (3.464,2) circle (3pt);
\draw[fill=gray] (3.464,-2) circle (3pt);
\draw[fill=gray] (-3.464,2) circle (3pt);
\draw[fill=gray] (-3.464,-2) circle (3pt);
\draw[fill=gray] (0.5,0.866) circle (3pt);
\draw[fill=gray] (1,0) circle (3pt);
\draw[fill=gray] (0.5,-0.866) circle (3pt);
\draw[fill=gray] (-0.5,0.866) circle (3pt);
\draw[fill=gray] (-1,0) circle (3pt);
\draw[fill=gray] (-0.5,-0.866) circle (3pt);

\begin{scope}[xshift=-10cm]
    \draw (0,0) -- (0.5,0.866) -- (0,2) -- (-0.5,0.866) --(0,0) -- (-1,0) -- (-1.732,-1) -- (-0.5,-0.866) -- (0,0) -- (0.5,-0.866) -- (1.732,-1) -- (1,0) -- (0,0);
\draw (0.5,0.866) -- (1.732,1) -- (1,0);
\draw (-0.5,0.866) -- (-1.732,1) -- (-1,0);
\draw (-0.5,-0.866) -- (0,-2) -- (0.5,-0.866);
\draw (0,2) -- (1.366,2.366) -- (0,4) -- (-1.366,2.366) -- (0,2);
\draw (0,-2) -- (1.366,-2.366) -- (0,-4) -- (-1.366,-2.366) -- (0,-2);
\draw (1.732,1) -- (1.366,2.366) -- (3.464,2) -- (2.732,0) -- (1.732,1); 
\draw (-1.732,1) -- (-2.732,0) -- (-3.464,2) -- (-1.366,2.366) -- (-1.732,1);
\draw (1.732,-1) -- (2.732,0) -- (3.464,-2) -- (1.366,-2.366) -- (1.732,-1);
\draw (-1.732,-1) -- (-2.732,0) -- (-3.464,-2) -- (-1.366,-2.366) -- (-1.732,-1);

\draw[fill=gray] (1.366,2.366) circle (3pt);
\draw[fill=gray] (-1.366,2.366) circle (3pt);
\draw[fill=gray] (1.366,-2.366) circle (3pt);
\draw[fill=gray] (2.732,0) circle (3pt);
\draw[fill=gray] (-2.732,0) circle (3pt);
\draw[fill=gray] (-1.366,-2.366) circle (3pt);

\draw[fill=gray] (0,0) circle (3pt);
\draw[fill=gray] (0,2) circle (3pt);
\draw[fill=gray] (0,-2) circle (3pt);
\draw[fill=gray] (1.732,1) circle (3pt);
\draw[fill=gray] (1.732,-1) circle (3pt);
\draw[fill=gray] (-1.732,1) circle (3pt);
\draw[fill=gray] (-1.732,-1) circle (3pt);

\draw[fill=gray] (0,4) circle (3pt);
\draw[fill=gray] (0,-4) circle (3pt);
\draw[fill=gray] (3.464,2) circle (3pt);
\draw[fill=gray] (3.464,-2) circle (3pt);
\draw[fill=gray] (-3.464,2) circle (3pt);
\draw[fill=gray] (-3.464,-2) circle (3pt);
\draw[fill=gray] (0.5,0.866) circle (3pt);
\draw[fill=gray] (1,0) circle (3pt);
\draw[fill=gray] (0.5,-0.866) circle (3pt);
\draw[fill=gray] (-0.5,0.866) circle (3pt);
\draw[fill=gray] (-1,0) circle (3pt);
\draw[fill=gray] (-0.5,-0.866) circle (3pt);

\node[scale=0.5] at (3.964,-2.5) {(r,11)};
\node[scale=0.5] at (3.532,0) {(r,0)};
\node[scale=0.5] at (3.964,2.5) {(r,1)};
\node[scale=0.5] at (1.866,2.866) {(r,2)};
\node[scale=0.5] at (0,4.5) {(r,3)};
\node[scale=0.5] at (-1.866,2.866) {(r,4)};
\node[scale=0.5] at (-3.964,2.5) {(r,5)};
\node[scale=0.5] at (-3.532,0) {(r,6)};
\node[scale=0.5] at (-3.964,-2.5) {(r,7)};
\node[scale=0.5] at (-1.866,-2.866) {(r,8)};
\node[scale=0.5] at (0,-4.5) {(r,9)};
\node[scale=0.5] at (1.866,-2.866) {(r,10)};

\node[scale=0.5] at (0,2.5) {$\iota^r_2$};
\node[scale=0.5] at (0,-2.5) {$\iota^r_8$};
\node[scale=0.5] at (1.682,0) {$\iota^r_{11}$};
\node[scale=0.5] at (-1.682,0) {$\iota^r_5$};
\node[scale=0.5] at (2.266,-1.366) {$\iota^r_{10}$};
\node[scale=0.5] at (2.266,1.366) {$\iota^r_0$};
\node[scale=0.5] at (-2.266,1.366) {$\iota^r_4$};
\node[scale=0.5] at (-2.266,-1.366) {$\iota^r_6$};
\node[scale=0.5] at (0.866,1.266) {$\iota^r_1$};
\node[scale=0.5] at (-0.866,1.266) {$\iota^r_3$};
\node[scale=0.5] at (-0.866,-1.266) {$\iota^r_7$};
\node[scale=0.5] at (0.866,-1.266) {$\iota^r_9$};

\node[scale=0.5] at (0.5,0.25) {$\alpha_r$};
\end{scope}

\begin{scope}[xshift=10cm]
\draw (0,0) -- (1,1.5) -- (2.5,1.5) -- (3.5,0) -- (2.5,-1.5) -- (1,-1.5) -- (0,0); 
\draw (0,0) -- (-1,1.5) -- (-2.5,1.5) -- (-3.5,0) -- (-2.5,-1.5) -- (-1,-1.5) -- (0,0); 
    \draw[fill=gray] (0,0) circle (3pt);
    \draw[fill=gray] (-3.5,0) circle (3pt);
    \draw[fill=gray] (3.5,0) circle (3pt);
    \draw[fill=gray] (1,1.5) circle (3pt);
    \draw[fill=gray] (2.5,1.5) circle (3pt);
    \draw[fill=gray] (1,-1.5) circle (3pt);
    \draw[fill=gray] (2.5,-1.5) circle (3pt);
        \draw[fill=gray] (-1,1.5) circle (3pt);
    \draw[fill=gray] (-2.5,1.5) circle (3pt);
    \draw[fill=gray] (-1,-1.5) circle (3pt);
    \draw[fill=gray] (-2.5,-1.5) circle (3pt);
    \node[scale=0.5] at (-0.75,0) {$\omega$};
    \node[scale=0.5] at (-1,2) {$(a_1,1)$};
    \node[scale=0.5] at (-3,2) {$(a_1,2)$};
    \node[scale=0.5] at (-1,-2) {$(a_1,5)$};
    \node[scale=0.5] at (-3,-2) {$(a_1,4)$};
    \node[scale=0.5] at (-4.5,0) {$(a_1,3)$};

    \node[scale=0.5] at (1,2) {$(a_2,1)$};
    \node[scale=0.5] at (3,2) {$(a_2,2)$};
    \node[scale=0.5] at (1,-2) {$(a_2,5)$};
    \node[scale=0.5] at (3,-2) {$(a_2,4)$};
    \node[scale=0.5] at (4.5,0) {$(a_2,3)$};
\end{scope}

\node[scale=0.5] at (3.964,-2.5) {($a_2$,5)};
\node[scale=0.5] at (3.332,0) {$\omega$};
\node[scale=0.5] at (3.964,2.5) {($a_1$,1)};
\node[scale=0.5] at (1.866,2.866) {($a_1$,2)};
\node[scale=0.5] at (0,4.5) {($a_1$,3)};
\node[scale=0.5] at (-1.866,2.866) {($a_1$,4)};
\node[scale=0.5] at (-3.964,2.5) {($a_1$,5)};
\node[scale=0.5] at (-3.332,0) {$\omega$};
\node[scale=0.5] at (-3.964,-2.5) {($a_2$,1)};
\node[scale=0.5] at (-1.866,-2.866) {($a_2$,2)};
\node[scale=0.5] at (0,-4.5) {($a_2$,3)};
\node[scale=0.5] at (1.866,-2.866) {($a_2$,4)};

\node[scale=0.5] at (0,2.5) {(r,3)};
\node[scale=0.5] at (0,-2.5) {(r,9)};
\node[scale=0.5] at (1.682,0) {(r,0)};
\node[scale=0.5] at (-1.682,0) {(r,6)};
\node[scale=0.5] at (2.266,-1.366) {(r,11)};
\node[scale=0.5] at (2.266,1.366) {(r,1)};
\node[scale=0.5] at (-2.266,1.366) {(r,5)};
\node[scale=0.5] at (-2.266,-1.366) {(r,7)};
\node[scale=0.5] at (0.866,1.266) {(r,2)};
\node[scale=0.5] at (-0.866,1.266) {(r,4)};
\node[scale=0.5] at (-0.866,-1.266) {(r,8)};
\node[scale=0.5] at (0.866,-1.266) {(r,10)};
\end{tikzpicture}
    \caption{ Illustration for the definition of the graph $H_r$ when $E = \{a_1,a_2\}$ and $r = a_1 a_2$. The graph on the left is a possible choice for $F_r$. The graph on the right is $C_E$. The graph $H_r$ can be seen as the result of identifying vertices with the same name in the middle graph and the right graph.}
    \label{fig:illustration.realization.simple}
\end{figure}

\begin{figure}[ht!]
    \centering
    \begin{tikzpicture}[scale=0.5]
\draw (0.5,0.866) -- (0,2) -- (-0.5,0.866);
\draw (-1,0) -- (-1.732,-1) -- (-0.5,-0.866);
\draw (0.5,-0.866) -- (1.732,-1) -- (1,0);
\draw (0.5,0.866) -- (1.732,1) -- (1,0);
\draw (-0.5,0.866) -- (-1.732,1) -- (-1,0);
\draw (-0.5,-0.866) -- (0,-2) -- (0.5,-0.866);
\draw (0,2) -- (0,4);
\draw (0,-2) -- (0,-4);
\draw (1.366,2.366) -- (0.5,0.866);
\draw (1.366,-2.366) -- (0.5,-0.866);
\draw (-1.366,2.366) -- (-0.5,0.866);
\draw (-1.366,-2.366) -- (-0.5,-0.866);
\draw (1,0) -- (2.732,0);
\draw (-1,0) -- (-2.732,0);
\draw (3.464,2) -- (1.732,1);
\draw (3.464,-2) -- (1.732,-1);
\draw (-3.464,2) -- (-1.732,1);
\draw (-3.464,-2) -- (-1.732,-1);
\draw (1.366,2.366)-- (0,4) -- (-1.366,2.366);
\draw (1.366,-2.366) -- (0,-4) -- (-1.366,-2.366);
\draw (1.366,2.366) -- (3.464,2) -- (2.732,0);
\draw (-2.732,0) -- (-3.464,2) -- (-1.366,2.366);
\draw (2.732,0) -- (3.464,-2) -- (1.366,-2.366);
\draw (-2.732,0) -- (-3.464,-2) -- (-1.366,-2.366);

\draw[fill=gray] (1.366,2.366) circle (3pt);
\draw[fill=gray] (-1.366,2.366) circle (3pt);
\draw[fill=gray] (1.366,-2.366) circle (3pt);
\draw[fill=gray] (2.732,0) circle (3pt);
\draw[fill=gray] (-2.732,0) circle (3pt);
\draw[fill=gray] (-1.366,-2.366) circle (3pt);

\draw[fill=gray] (0,2) circle (3pt);
\draw[fill=gray] (0,-2) circle (3pt);
\draw[fill=gray] (1.732,1) circle (3pt);
\draw[fill=gray] (1.732,-1) circle (3pt);
\draw[fill=gray] (-1.732,1) circle (3pt);
\draw[fill=gray] (-1.732,-1) circle (3pt);

\draw[fill=gray] (0,4) circle (3pt);
\draw[fill=gray] (0,-4) circle (3pt);
\draw[fill=gray] (3.464,2) circle (3pt);
\draw[fill=gray] (3.464,-2) circle (3pt);
\draw[fill=gray] (-3.464,2) circle (3pt);
\draw[fill=gray] (-3.464,-2) circle (3pt);
\draw[fill=gray] (0.5,0.866) circle (3pt);
\draw[fill=gray] (1,0) circle (3pt);
\draw[fill=gray] (0.5,-0.866) circle (3pt);
\draw[fill=gray] (-0.5,0.866) circle (3pt);
\draw[fill=gray] (-1,0) circle (3pt);
\draw[fill=gray] (-0.5,-0.866) circle (3pt);

\node[scale=0.7] at (2.232,-0.325) {$\upsilon_{r, 0}$}; 
\node[scale=0.7] at (2.432,0.825) {$\upsilon_{r, 1}$}; 
\node[scale=0.7] at (1.582,1.625) {$\upsilon_{r, 2}$}; 
\node[scale=0.7] at (0.582,2.525) {$\upsilon_{r, 3}$}; 

\node[scale=0.7] at (1.982,-1.625) {$\upsilon_{r, 11}$}; 
\node[scale=0.7] at (0.782,-2.125) {$\upsilon_{r, 10}$};
\node[scale=0.7] at (-1.482,-1.625) {$\upsilon_{r, 8}$}; 
\node[scale=0.7] at (-0.582,-2.525) {$\upsilon_{r, 9}$}; 

\node[scale=0.7] at (-2.432,-0.925) {$\upsilon_{r, 7}$}; 
\node[scale=0.7] at (-2.132,0.325) {$\upsilon_{r, 6}$}; 
\node[scale=0.7] at (-1.982,1.625) {$\upsilon_{r, 5}$}; 
\node[scale=0.7] at (-0.682,2.125) {$\upsilon_{r, 4}$}; 

\begin{scope}[xshift=-12cm]
    \draw (0.5,0.866) -- (0,2) -- (-0.5,0.866);
    \draw (-1,0) -- (-1.732,-1) -- (-0.5,-0.866);
    \draw (0.5,-0.866) -- (1.732,-1) -- (1,0);
\draw (0.5,0.866) -- (1.732,1) -- (1,0);
\draw (-0.5,0.866) -- (-1.732,1) -- (-1,0);
\draw (-0.5,-0.866) -- (0,-2) -- (0.5,-0.866);
\draw (0,2) -- (1.366,2.366) -- (0,4) -- (-1.366,2.366) -- (0,2);
\draw (0,-2) -- (1.366,-2.366) -- (0,-4) -- (-1.366,-2.366) -- (0,-2);
\draw (1.732,1) -- (1.366,2.366) -- (3.464,2) -- (2.732,0) -- (1.732,1); 
\draw (-1.732,1) -- (-2.732,0) -- (-3.464,2) -- (-1.366,2.366) -- (-1.732,1);
\draw (1.732,-1) -- (2.732,0) -- (3.464,-2) -- (1.366,-2.366) -- (1.732,-1);
\draw (-1.732,-1) -- (-2.732,0) -- (-3.464,-2) -- (-1.366,-2.366) -- (-1.732,-1);

\draw[fill=gray] (1.366,2.366) circle (3pt);
\draw[fill=gray] (-1.366,2.366) circle (3pt);
\draw[fill=gray] (1.366,-2.366) circle (3pt);
\draw[fill=gray] (2.732,0) circle (3pt);
\draw[fill=gray] (-2.732,0) circle (3pt);
\draw[fill=gray] (-1.366,-2.366) circle (3pt);

\draw[fill=gray] (0,2) circle (3pt);
\draw[fill=gray] (0,-2) circle (3pt);
\draw[fill=gray] (1.732,1) circle (3pt);
\draw[fill=gray] (1.732,-1) circle (3pt);
\draw[fill=gray] (-1.732,1) circle (3pt);
\draw[fill=gray] (-1.732,-1) circle (3pt);

\draw[fill=gray] (0,4) circle (3pt);
\draw[fill=gray] (0,-4) circle (3pt);
\draw[fill=gray] (3.464,2) circle (3pt);
\draw[fill=gray] (3.464,-2) circle (3pt);
\draw[fill=gray] (-3.464,2) circle (3pt);
\draw[fill=gray] (-3.464,-2) circle (3pt);
\draw[fill=gray] (0.5,0.866) circle (3pt);
\draw[fill=gray] (1,0) circle (3pt);
\draw[fill=gray] (0.5,-0.866) circle (3pt);
\draw[fill=gray] (-0.5,0.866) circle (3pt);
\draw[fill=gray] (-1,0) circle (3pt);
\draw[fill=gray] (-0.5,-0.866) circle (3pt);

\node[scale=0.7] at (3.964,-2.5) {($a_2$,5)};
\node[scale=0.9] at (3.332,0) {$\omega$};
\node[scale=0.7] at (3.964,2.5) {($a_1$,1)};
\node[scale=0.7] at (1.866,2.866) {($a_1$,2)};
\node[scale=0.7] at (0,4.5) {($a_1$,3)};
\node[scale=0.7] at (-1.866,2.866) {($a_1$,4)};
\node[scale=0.7] at (-3.964,2.5) {($a_1$,5)};
\node[scale=0.9] at (-3.332,0) {$\omega$};
\node[scale=0.7] at (-3.964,-2.5) {($a_2$,1)};
\node[scale=0.7] at (-1.866,-2.866) {($a_2$,2)};
\node[scale=0.7] at (0,-4.5) {($a_2$,3)};
\node[scale=0.7] at (1.866,-2.866) {($a_2$,4)};

\node[scale=0.7] at (2.432,-0.825) {$\upsilon_{r, 0}$}; 
\node[scale=0.7] at (2.432,0.825) {$\upsilon_{r, 1}$}; 
\node[scale=0.7] at (1.982,1.625) {$\upsilon_{r, 2}$}; 
\node[scale=0.7] at (0.582,2.525) {$\upsilon_{r, 3}$}; 

\node[scale=0.7] at (1.982,-1.625) {$\upsilon_{r, 11}$}; 
\node[scale=0.7] at (0.582,-2.525) {$\upsilon_{r, 10}$};
\node[scale=0.7] at (-1.982,-1.625) {$\upsilon_{r, 8}$}; 
\node[scale=0.7] at (-0.582,-2.525) {$\upsilon_{r, 9}$}; 

\node[scale=0.7] at (-2.432,-0.825) {$\upsilon_{r, 7}$}; 
\node[scale=0.7] at (-2.432,0.825) {$\upsilon_{r, 6}$}; 
\node[scale=0.7] at (-1.982,1.625) {$\upsilon_{r, 5}$}; 
\node[scale=0.7] at (-0.582,2.525) {$\upsilon_{r, 4}$}; 
\end{scope}

\node[scale=0.7] at (3.964,-2.5) {($a_2$,5)};
\node[scale=0.9] at (3.332,0) {$\omega$};
\node[scale=0.7] at (3.964,2.5) {($a_1$,1)};
\node[scale=0.7] at (1.866,2.866) {($a_1$,2)};
\node[scale=0.7] at (0,4.5) {($a_1$,3)};
\node[scale=0.7] at (-1.866,2.866) {($a_1$,4)};
\node[scale=0.7] at (-3.964,2.5) {($a_1$,5)};
\node[scale=0.9] at (-3.332,0) {$\omega$};
\node[scale=0.7] at (-3.964,-2.5) {($a_2$,1)};
\node[scale=0.7] at (-1.866,-2.866) {($a_2$,2)};
\node[scale=0.7] at (0,-4.5) {($a_2$,3)};
\node[scale=0.7] at (1.866,-2.866) {($a_2$,4)};

\end{tikzpicture}
    \caption{Illustration of the effect of $\boldsymbol{\nu}$. The graph on the left corresponds to the choice $\nu_{r} = (0,1,0,1 \ldots)$ and the one on the right corresponds to the choice $\nu_{r} = (0,0,0,0 \ldots)$. 
     }
    \label{fig:illustration.realization}
\end{figure}
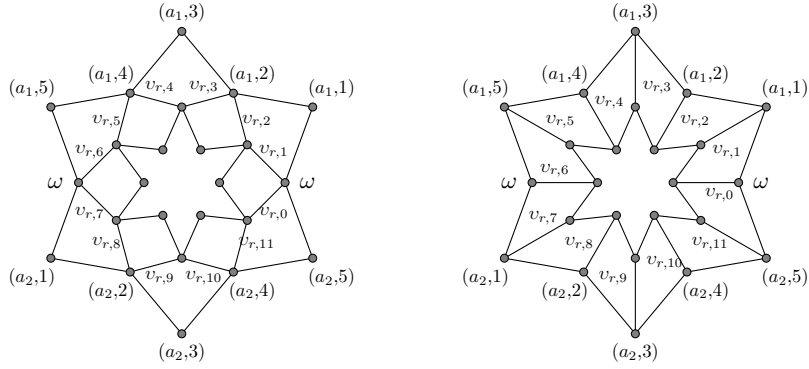

\begin{example}
The graph $G$ given as preliminary example in Figure~\ref{fig:Z3} corresponds to the result of this construction with $E = \{g\}$, $R = \{g^3\}$, $N = 6$, and $\nu = ((0)_{k\leq 18})$.
\end{example}

The next proposition shows that this construction can be done in an effective manner (in particular, we find an appropriate family for Step 4) and proves an additional property that is necessary for the main theorem. 

This explicit construction is sufficient for our purpose, but we left some degree of freedom in the choice of parameters for future work. We hope that other choices of parameters may yield homshifts with stronger mixing properties (constant block-gluing, strong irreducibility). We do not yet completely understand the relationship of these mixing properties with the properties of the graph.

\begin{proposition}
    \label{rem:realization}
There exists an algorithm which takes as input a finite group presentation $(E,R)$ and outputs a choice for $N,\textbf{F}$ and $\boldsymbol{\nu}$ that satisfy the conditions of the construction. Furthermore, every square of $\mathcal{G}(E,R,N,\textbf{F},\boldsymbol{\nu})$ is included in $\mathcal{H}_r(E,R,N,\textbf{F},\boldsymbol{\nu})$ or in $F_r$ for some $r$. 
\end{proposition}

\begin{proof}
    We pick $N=6, \nu_r = (0)_{k\leq 6|r|}$ for all $r$, and $F_r$ as a flat quadrangulation whose border is $C_r$ and such that no vertex outside of the border has two neighbours in the border (for instance the one provided by Proposition \ref{proposition.quadrangulating.cn}). The square group of $F_r$ is trivial by Theorem \ref{lemma: quadrangulation of cycles}. Furthermore, notice that $\mathcal{H}_r(E,R,N,\textbf{F},\boldsymbol{\nu}) \cap F_r = C_r$, and that if a square of $\mathcal{G}(E,R,N,\textbf{F},\boldsymbol{\nu})$ has a vertex in $F_r\backslash C_r$ and a vertex in $\mathcal{H}_r(E,R,N,\textbf{F},\boldsymbol{\nu})\backslash C_r$, these vertices cannot be neighbours so the other two vertices belong to $C_r$, contradicting the choice of $F_r$.   
    Thus $\textbf{F}$ satisfies the required conditions.
\end{proof}

\begin{theorem} \label{thm:realization} 
 Assume that every square of $\mathcal{G}(E,R,N,\textbf{F},\boldsymbol{\nu})$ is the product of cycles of the form $p \star s \star p^{-1}$, where $p$ is a walk and $s$ is a square included in $\mathcal{H}_r(E,R,N,\textbf{F},\boldsymbol{\nu})$ or in $F_r$ for some $r$. 

We have the following: 
    \[\pi_1^\square(\mathcal{G}(E,R,N,\textbf{F},\boldsymbol{\nu})) \cong \langle E:R\rangle.\]
\end{theorem}

\begin{proof}

In order to simplify the notations, let us set: $G \coloneqq \mathcal{G}(E,R,N,\textbf{F},\boldsymbol{\nu})$; $H_r \coloneqq \mathcal{H}_r(E,R,N,\boldsymbol{\nu})$ for all $r \in R$; $C_a \coloneqq C^{\omega}_{a,N}$ for all $a \in E$; $C_E \coloneqq C^{\omega}_{E,N}$; and $C_r \coloneqq C_{r,M_r}$ for all $r \in R$.

For each $a\in E$, define $\Phi(a)$ as the directed edge from $(a, 1)$ to $\omega$ in $C_a$. Extend $\Phi$ to a group isomorphism from the free group $\mathbb{F}_E$ to the free group $\mathbb{F}_{\Phi(E)}$.

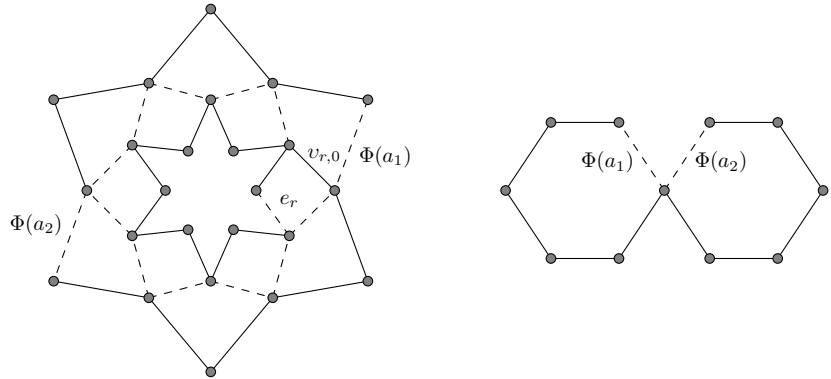
\begin{figure}[ht!]
    \centering
    \begin{tikzpicture}[scale=0.6]
\draw (0.5,0.866) -- (0,2) -- (-0.5,0.866);
\draw (-1,0) -- (-1.732,-1) -- (-0.5,-0.866);
\draw (0.5,-0.866) -- (1.732,-1);
\draw[dashed] (1.732,-1) -- (1,0);
\draw (0.5,0.866) -- (1.732,1) -- (1,0);
\draw (-0.5,0.866) -- (-1.732,1) -- (-1,0);
\draw (-0.5,-0.866) -- (0,-2) -- (0.5,-0.866);
\draw[dashed] (0,2) -- (1.366,2.366);
\draw (1.366,2.366) -- (0,4) -- (-1.366,2.366);
\draw[dashed] (-1.366,2.366) -- (0,2);
\draw[dashed] (0,-2) -- (1.366,-2.366);
\draw (1.366,-2.366) -- (0,-4) -- (-1.366,-2.366);
\draw[dashed] (-1.366,-2.366) -- (0,-2);
\draw[dashed] (1.732,1) -- (1.366,2.366);
\draw (1.366,2.366) -- (3.464,2);
\draw[dashed] (3.464,2) -- (2.732,0);
\node[scale=0.8] at (3.85,0.75) {$\Phi(a_1)$};
\node[scale=0.8] at (-3.85,-0.75) {$\Phi(a_2)$};
\draw (2.732,0) -- (1.732,1); 
\node[scale=0.8] at (2.5,0.75) {$\upsilon_{r,0}$}; 
\draw[dashed] (-1.732,1) -- (-2.732,0);
\draw (-2.732,0) -- (-3.464,2) -- (-1.366,2.366);
\draw[dashed] (-1.366,2.366) -- (-1.732,1);
\draw[dashed] (1.732,-1) -- (2.732,0);
\draw (2.732,0) -- (3.464,-2) -- (1.366,-2.366);
\draw[dashed] (1.366,-2.366) -- (1.732,-1);
\node[scale=0.8] at (1.732,-0.225) {$e_r$}; 
\draw[dashed] (-1.732,-1) -- (-2.732,0);
\draw[dashed] (-2.732,0) -- (-3.464,-2);
\draw (-3.464,-2) -- (-1.366,-2.366);
\draw[dashed]  (-1.366,-2.366) -- (-1.732,-1);

\draw[fill=gray] (1.366,2.366) circle (3pt);
\draw[fill=gray] (-1.366,2.366) circle (3pt);
\draw[fill=gray] (1.366,-2.366) circle (3pt);
\draw[fill=gray] (2.732,0) circle (3pt);
\draw[fill=gray] (-2.732,0) circle (3pt);
\draw[fill=gray] (-1.366,-2.366) circle (3pt);

\draw[fill=gray] (0,2) circle (3pt);
\draw[fill=gray] (0,-2) circle (3pt);
\draw[fill=gray] (1.732,1) circle (3pt);
\draw[fill=gray] (1.732,-1) circle (3pt);
\draw[fill=gray] (-1.732,1) circle (3pt);
\draw[fill=gray] (-1.732,-1) circle (3pt);

\draw[fill=gray] (0,4) circle (3pt);
\draw[fill=gray] (0,-4) circle (3pt);
\draw[fill=gray] (3.464,2) circle (3pt);
\draw[fill=gray] (3.464,-2) circle (3pt);
\draw[fill=gray] (-3.464,2) circle (3pt);
\draw[fill=gray] (-3.464,-2) circle (3pt);
\draw[fill=gray] (0.5,0.866) circle (3pt);
\draw[fill=gray] (1,0) circle (3pt);
\draw[fill=gray] (0.5,-0.866) circle (3pt);
\draw[fill=gray] (-0.5,0.866) circle (3pt);
\draw[fill=gray] (-1,0) circle (3pt);
\draw[fill=gray] (-0.5,-0.866) circle (3pt);

\begin{scope}[xshift=10cm]
\draw[dashed] (0,0) -- (1,1.5);
\draw (1,1.5) -- (2.5,1.5) -- (3.5,0) -- (2.5,-1.5); 
\draw (2.5,-1.5) -- (1,-1.5);
\node[scale=0.8] at (1.25,0.625) {$\Phi(a_2)$};
\node[scale=0.8] at (-1.25,0.625) {$\Phi(a_1)$};
\draw (1,-1.5) -- (0,0); 
\draw[dashed] (0,0) -- (-1,1.5);
\draw (-1,1.5) -- (-2.5,1.5) -- (-3.5,0) -- (-2.5,-1.5) -- (-1,-1.5) -- (0,0); 
    \draw[fill=gray] (0,0) circle (3pt);
    \draw[fill=gray] (-3.5,0) circle (3pt);
    \draw[fill=gray] (3.5,0) circle (3pt);
    \draw[fill=gray] (1,1.5) circle (3pt);
    \draw[fill=gray] (2.5,1.5) circle (3pt);
    \draw[fill=gray] (1,-1.5) circle (3pt);
    \draw[fill=gray] (2.5,-1.5) circle (3pt);
        \draw[fill=gray] (-1,1.5) circle (3pt);
    \draw[fill=gray] (-2.5,1.5) circle (3pt);
    \draw[fill=gray] (-1,-1.5) circle (3pt);
    \draw[fill=gray] (-2.5,-1.5) circle (3pt);
\end{scope}

\end{tikzpicture}
    \caption{Illustration of the definition of the tree $S_r$, in the same context and with the same notations as on Figure \ref{fig:illustration.realization.simple}. The edges outside of $T_r$ are dashed.
    }
    \label{fig:illustration.realization.2}
\end{figure}

\paragraph{Definition of a spanning tree.} For each $r \in R$, let $e_r$ be the edge from $(r, 0)$ to $(r,1)$. 
Consider the spanning tree $S_r$ of $H_r$ whose edges are: the ones of $C_E$
except the edges $\Phi(a)$, $a \in E$; 
the edges of $C_r$ 
except $e_r$; and edge $\upsilon_{r,0}$.
This definition is illustrated on Figure \ref{fig:illustration.realization.2}.
Since $H_r \cap F_r = \partial F_r= C_{r}$, we have by construction $S_r \cap \partial F_r= C_{r}\setminus \{e_r\}$ which is connected. We apply Lemma \ref{lemma.tree.extension} to obtain a spanning tree $T_r$ of $H_r\cup F_r$ such that $T_r\cap H_r = S_r$ and $T_r \cap F_r$ is a spanning tree of $F_r$.
\paragraph{For all $r \in R$, the square group of $H_r\cup F_r$ is $\langle \Phi(E) : \{\Phi(r)\}\rangle $.} 
First note that $T_r$
is the union of the spanning trees $T_r\cap H_r$ and $T_r\cap F_r$.
Then the Van Kampen theorem (Theorem \ref{theorem.vankampen}) implies that
the square group of the graph $H_r\cup F_r$ has presentation $(\tilde E_r, \tilde R_r)$, where $\tilde E_r = E_{H_r} \cup E_{F_r}$, $ \tilde R_r = R^{\square}_{T_r \cap F_r} (F_r) \cup  R^{ \square}_{T_r \cap H_r} (H_r) \cup R_{\texttt{add}}$ where $R_{\texttt{add}}$ is the set of squares
$s$ in $H_r\cup F_r$ that are not completely contained neither in $H_r$ nor $F_r$. 

We apply modifications to the presentation $(\tilde E_r,\tilde R_r)$ in successive steps as follows, where each step does not change the corresponding group.

\begin{enumerate}
    \item 
    As a consequence of the hypothesis, every element of $R_{\texttt{add}}$ is in the smallest normal subgroup containing $R^{\square}_{T_r \cap F_r} (F_r)$ and $R^{\square}_{T_r \cap H_r} (H_r)$. By Lemma \ref{lemma.mod.presentation}(2) we can remove $R_{\texttt{add}}$ from the presentation without changing the group. 

    \item 
    The square group of $F_r$, which can be written as $\langle E_{F_r} : R^{\square}_{T_r \cap F_r} (F_r)\rangle$ is trivial. In particular,  $E_{F_r}$ is contained in the smallest normal subgroup of $\mathbb F_{F_r}$ containing $R^{\square}_{T_r \cap F_r} (F_r)$, and thus in the smallest normal subgroup of $\mathbb F_{\tilde E_r}$ containing $R^{\square}_{T_r \cap F_r} (F_r)$. Therefore
    Lemma \ref{lemma.mod.presentation}(2) 
    implies that $\langle \tilde E_r : \tilde R_r \rangle=\langle \tilde E_r : \tilde R_r \cup \{a: a\in E_{F_r}\} \rangle$.
    As a consequence, we apply Lemma \ref{lemma.mod.presentation}(1), and remove every generator $a \in E_{F_r}$ from $\tilde E_r$ and from every relation in $\tilde R_r$ (this is equivalent to setting $a$ to identity) without changing the generated group.

    \item Because the set $\tilde R_r$ contains the relation $e$ for every edge $e$ remaining in the spanning tree $T_r\cap H_r$, we remove every such generator from $\tilde E_r$ and from every relation remaining in $\tilde R_r$.

    After all modifications made so far, $\tilde E_r$ is reduced to the generators $\Phi(a)$ for $a \in E$ and generators corresponding to all edges $\upsilon_{r,i}$ except $\upsilon_{r,0}$ (which was removed along with the spanning tree), and their inverse edges.
    
   \item We use Lemma \ref{lemma.mod.presentation}(1) to replace the inverse edge of $\Phi(a)$ by $\Phi(a)^{-1}$ and the inverse edge of $\upsilon_{r,i}$ ($i\neq 0$) by $\upsilon_{r,i}^{-1}$ in every relation in $\tilde R_r$.
    
    \item This is the main step of the proof, illustrated in Figure~\ref{fig:illustration.realization.3}.
    
    Write $r =  a_1^{\varepsilon_1} \ldots a_{|r|}^{\varepsilon_{|r|}} \in R$. Observe that the edge $\upsilon_{r,1}$ belongs to a square whose other edges are $\upsilon_{r,0}$ and two edges from $C_r$ and/or $C_{a_1}$. We have two possibilities:
\begin{itemize}
\item $\Phi(a_1)$ is not one of these edges. Then all these generators were removed in previous steps, so the relation in $\tilde R_r$ corresponding to this square became $\upsilon_{r,1}$ in the group $\langle \tilde E_r : \tilde R_r \rangle$. We thus remove $\upsilon_{r,1}$ from $\tilde E_r$ and from every relation in $\tilde R$ without changing the group.
\item $\Phi(a_1)$ is one of these edges. Then the relation in $\tilde R_r$ corresponding to this square became $\upsilon_{r,1}\Phi(a_1)^{-\varepsilon_1}$ (remember that the value of $\varepsilon_1$ affected the order of the sequence of edges $\upsilon_r$ around the cycle $C_r$). Applying Lemma \ref{lemma.mod.presentation}(1), we remove the generator $\upsilon_{r,1}$ from $\tilde E_r$ and replace it by $\Phi(a_1)^{\varepsilon_1}$ in every relation in $\tilde R$ without changing the group.
\end{itemize}

\begin{figure}[ht!]
    \centering
    \input{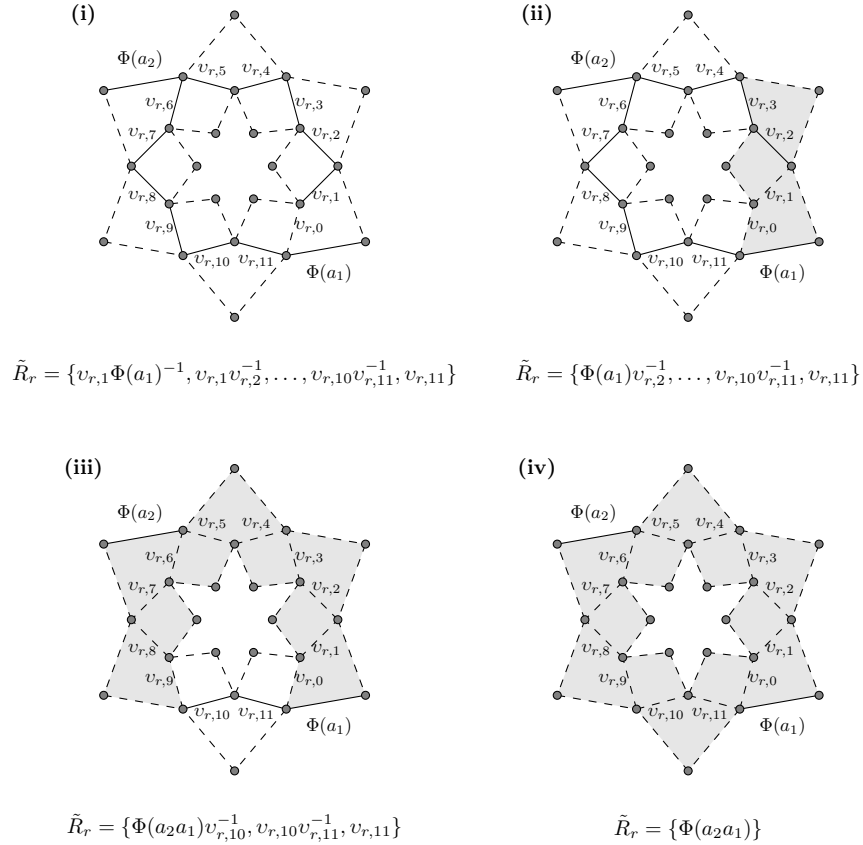}
    \caption{Illustration of the step 5 in the proof of Theorem \ref{thm:realization}, using the same graph as Figure \ref{fig:illustration.realization}. At each step, dashed edges correspond to generators which are already removed and the current contents of the set $\tilde R_r$ are indicated.}
    \label{fig:illustration.realization.3}
\end{figure}

Applying this process inductively, we remove one by one generators $\upsilon_{r,i}$ (for $i\in\Z/M_r\Z$) from $\tilde E_r$ where, at each step, $\upsilon_{r,i}$ is replaced with $\Phi(a_1^{\varepsilon_1} \ldots a_k^{\varepsilon_k})$ in every relation in $\tilde R$, where $k$ is the number of edges in $\Phi(E)$ encountered in the first $i$ steps. 

At the last step, $\upsilon_{r,M_r-1}$ is in a square in common with the edge $\upsilon_{r,0}$ which was previously removed from the set of generators. This square yielded the relation $\upsilon_{r,M_r-1}$, which became, as we showed by induction, $\Phi(a_1^{\varepsilon_1} \ldots a_n^{\varepsilon_n}) = \Phi(r)$.

By construction, there is no other square in $H_r$ than the ones containing two consecutive edges of $\upsilon_r$. As all the relations corresponding to these squares became trivial, the unique nontrivial relation left in $\tilde R_r$ is $\Phi(a_1^{\varepsilon_1} \ldots a_n^{\varepsilon_n}) = \Phi(r)$.
\end{enumerate}

After these manipulations, $\tilde E_r$ is reduced to $\Phi(E)$ and the unique element left in $\tilde R_r$ is $\Phi(r)$, meaning that $(\Phi(E), \{\Phi(r)\})$ is a presentation of the square group of $H_r \cup F_r$.
 
\paragraph{The square group of $G$ is isomorphic to $\langle E : R\rangle $.} We apply Van Kampen theorem (Theorem~\ref{theorem.vankampen}) on the graphs $H_r \cup F_r$ for $r\in R$. As in the previous construction, the theorem hypothesis to ensures that any relation in $R_{add}$ is in the smallest normal subgroup generated by existing relations. Therefore
\[\pi_1^\square(G) = \langle \Phi(E): \Phi(R)\rangle \cong \langle E : R \rangle.\qedhere\]
\end{proof}

\subsection{Undecidability of properties of the square group}\label{subsection: undecidable}

The crux of the undecidability lies in theorems of Adian (often written Adyan) \cite{MR103217} and (independently) Rabin \cite{MR110743} which prove that most properties of finitely presented groups are undecidable. A reference for this is \cite[Chapter IV.4]{Lyndon}.

A non-empty set $P$ of finitely presented groups is called a \emph{Markov property} when it is stable by isomorphism and there exists a finitely presented group which cannot be embedded in an element of $P$.
For instance, the set of finite groups and the trivial group singleton are both Markov.

\begin{theorem}[Adian-Rabin Theorem] Suppose $P$ is a Markov property. It is algorithmically undecidable from $(E,R)$ whether the finitely presented group $\langle E:R\rangle$ is in  $P$.\label{Theorem: Adian rabin}
\end{theorem}

As mentioned in the introduction, a consequence of \cite{GHO23} is that the block gluing distance of a two-dimensional homshift depends on whether its square group is finite. This fact and Theorem \ref{thm:everygroup} yields the main theorem:

\main*

\begin{proof}
It is sufficient to prove this for two-dimensional homshifts on a non-bipartite graph. This is equivalent to prove that for these homshifts, it is not decidable if it is $\Theta(n)$-block gluing or $O(\log(n))$-block gluing. For all bipartite graph $G$, define $s(G)$ the graph obtained from $G$ by adding a self-loop on the first vertex in the description of the graph $G$. The square group of $s(G)$ is equal to $\pi_1^{\square}(G) * \mathbb Z / 2 \mathbb Z$ (Example \ref{example.self-loop}), which is finite if and only if $\pi_1^{\square}(G)$ is trivial. It was shown in \cite[Sections 4.4 and 6]{GHO23} that a homshift $X^2_G$ is $\Theta(n)$-phased block gluing if and only if the square cover $\mathcal U^\square_G$ is infinite. By Proposition \ref{Proposition: square cover finite square group}, this is true if and only if the square group $\pi_1^{\square}(G)$ is infinite. Since $s(G)$ is not bipartite, by Lemma \ref{lemma.phase.to.non.phase} we get that $X^2_{s(G)}$ is $\Theta(n)$-block gluing when the square group of $G$ is not trivial and $O(\log(n)))$-block gluing otherwise. Therefore, if it was possible to decide if a two-dimensional homshift on a non bipartite graph is $O(\log(n))$-block gluing or $\Theta(n)$-block gluing, by Theorem \ref{thm:everygroup} we would be able to decide if a finitely generated group is trivial or not, which is not possible by Adian-Rabin theorem. 
\end{proof}
\begin{remark}
In the previous proof, instead of adding a self-loop we could add a copy of any non-bipartite graph whose square group is $\Z/2\Z$. One example of such graph is the complete graph $K_4$ with four vertices $\{0,1,2,3\}$. 

\begin{figure}[!h]
    \centering
    \begin{tikzpicture}[scale=0.4]

\draw[dashed] (-2,-2) -- (2,-2);   
\draw[dashed] (2,-2) -- (2,2);     
\draw[dashed] (2,2) -- (-2,2);     
\draw[-latex] (-2,2) -- (-2,-2);           
\draw[-latex] (-2,-2) -- (2,2);            
\draw[-latex] (2,-2) -- (-2,2);            

\draw[fill=gray] (-2,-2) circle (3pt);
\draw[fill=gray] (2,-2) circle (3pt);
\draw[fill=gray] (2,2) circle (3pt);
\draw[fill=gray] (-2,2) circle (3pt);

\node[scale=0.8] at (-2.6,-2.6) {$0$};
\node[scale=0.8] at (2.6,-2.6) {$1$};
\node[scale=0.8] at (2.6,2.6) {$2$};
\node[scale=0.8] at (-2.6,2.6) {$3$};

\node[scale=0.9] at (-2.75,0) {$b$};

\node[scale=0.9] at (1.325,0.75) {$c$};
\node[scale=0.9] at (-1.325,0.75) {$a$};

\end{tikzpicture}
    \caption{Illustration for the computation of the square group of the complete graph $K_4$.}
    \label{fig:k4}
\end{figure}
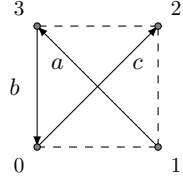

Indeed, consider the spanning tree $T$ whose vertices are $0,1,2$ and edges are $(0,1),(1,0),(1,2),(2,1),(2,3),(3,2)$, corresponding to the dashed edges on Figure \ref{fig:k4}. Let us denote respectively by $a,b,c$ the edges $(1,3),(3,0)$ and $(0,2)$. Up to rotational shift, the squares of this graph are $01230$,$01320$,$02130$,$02310$,$03120$, and $03210$. The corresponding relations, after simplification, are $b$, $ac^{-1}$, $cab$, $ca^{-1}$, $b^{-1} a^{-1} c^{-1}$, and $b^{-1}$. This implies that the 
square group of $K_4$ is 
\[\langle a,c : ac, ac^{-1}\rangle = \langle a : a^2\rangle \cong \mathbb Z / 2 \mathbb Z.\]

It is not difficult  to prove that for all $n \ge 4$, the square group of the complete graph on $n$ vertices is also $\mathbb Z / 2 \mathbb Z$.
\end{remark}

\begin{corollary}
    The classes of $\Theta(n)$-phased block gluing and $O(\log(n))$-phased block gluing two-dimensional homshifts are also computably inseparable.
\end{corollary}

\section{\label{section.open.questions}Open questions}
We hope that our work can be a starting point for exploring undecidability in homshifts, in particular finding undecidable properties and natural subclasses of homshifts for which such properties becomes decidable. Below are some more specific questions.

\subsection{Stronger mixing-type properties}

\mainquestion*

 The main result presented in the present article leads us to believe that this problem is also undecidable, at least when $d = 2$. We present some possible approaches in this case below. On the other hand, we have no idea for the case of higher dimensions $(d > 2)$.

One possibility would be to find a Markov property $P$ of finite presentations and construct a computable map $(E,R) \mapsto \mathcal{G}(E,R)$, where $\mathcal{G}(E,R)$ is a finite undirected graph for all presentation $(E,R)$, such that the square group of $\mathcal{G}(E,R)$ is $\langle E : R \rangle$ and the homshift $X^2_{\mathcal{G}(E,R)}$ is strongly irreducible if and only if $(E,R)$ satisfies the property $P$.

We have not found a property $P$ which satisfies these requirements, and it is not clear that such a property should be related to the square group. On the other hand, we believe that it is possible to choose the parameters of the construction presented in Section \ref{section.proof.realization} so that the obtained graph corresponds to a strongly irreducible two-dimensional homshift. When $E = \{a\}$ and $R = \{a^n\}$ for some integer $n > 0$ (the group $\langle E : R \rangle$ is thus $\mathbb Z / n \mathbb Z$), we believe that the obtained graph can be made strongly irreducible with the following choices: 
\begin{enumerate}
    \item The sequence $\nu$ is alternating $0$'s and $1$'s.
    \item The graph $F_{a^n}$, is \emph{dismantleable} (see \cite{nowakowski1983vertex})
    \item Every walk in $F_{a^n}$ from one vertex in $C_{a^n}$ to another in $C_{a^n}$ can be \emph{folded} to a walk in $C_{a^n}$ without changing the end points (see \cite{nowakowski1983vertex}).
\end{enumerate}
However, we do not know if this can be generalized to every finite presentation, and if it is, how - this is not clear even for finite presentations of finite cyclic groups. This seems to be related with very subtle aspects of the word problem which we haven't been able to identify.

Other mixing-type properties have been studied in subshifts: (quantified) corner gluing \cite{Boyle-Pavlov-Schraudner}, finite extension property, topological strong spatial mixing \cite{briceno2017strong}, etc. We are interested in how homshifts behave with regards to these properties, how they are related to the graph $G$ and whether they are decidable. We mention an open question from \cite{GHO23}.

\begin{question}
    Is there a two-dimensional homshift which is $o(\log(n))$-block gluing but not $O(1)$-block gluing?
\end{question}

\subsection{ Generalizations}

\subsubsection{Homshifts on other Cayley graphs} 

While one can define homshifts on arbitrary Cayley graph of any finitely presented group, there is very little known about mixing properties, or even more elementary questions such as the decidability of the emptiness problem. If the group has more than one end, then this may not be such a difficult question \cite{zbMATH06680456} but we don't really understand what happens in (for example) $\Z^d$ with an arbitrary set of generators. Consider the following question by Jordan Ellenberg \cite{CayleyEllenberg2018}: 

\begin{question}
    If a homshift on an arbitrary Cayley graph of $\mathbb Z^d$ is not empty, does it contain necessarily a periodic configuration ?
\end{question}

In order to approach undecidability of mixing-type properties, it would be appropriate to restrict to some simple classes of Cayley graphs, such as the ones of commutative groups, or of groups generated by two elements and possessing a \emph{fundamental domain} - such as Baumslag-Solitar groups for instance. 

Several of the techniques that we used are very specific to the grid graphs and would not apply in general. On the other hand, it should be possible to generalize some of our definitions - such as the one of the square group - and results - such as the lifting lemma.

\subsubsection{Sublinear block gluing rates of two-dimensional SFT} 

Theorem \ref{thm:main} yields an undecidability result for two-dimensional shifts of finite type: it is undecidable whether such a shift is $o(n)$-block gluing or not, since two-dimensional homshifts are shifts of finite type. However we do not know if, as it is the case for homshifts, there are no two-dimensional shifts of finite type which are $o(n)$-block gluing but not $O(\log(n))$-block gluing \cite{GS}. This question is deceptively hard, as natural ideas translate into constructions which invariably involve some elements making the shift linearly block gluing.

A possible direction of research would be to generalize the results obtained in \cite{GHO23} to larger classes of two-dimensional shifts of finite type, which would probably involve generalizations of the notions of square group and square cover, possibly in relation with the notion of projective fundamental group \cite{geller-propp}.

\subsubsection{Continuous analogues}\label{Section: Continuous analogues}

It is also natural to look for continuous analogues of our results. One difficulty comes from the generalization of homshifts to a continuous context. We mention the following question, which was raised by Tom Meyerovitch during discussions.

Let $M$ be a compact connected Riemannian manifold. For a fixed $n>0$ and two $1$-Lipschitz curves $p,q:[0,n]\to M$, define $d(p,q)$ as the infimum on the positive numbers $t$ such that there exists a $1$-Lipschitz map $\phi:[0,t]\times[0,n]\to M$ such that $\phi(0,\cdot)=p$ and $\phi(t,\cdot)=q$.

Define $D_n(M)$ as the supremum of $d(p,q)$ over all $1$-Lipshitz curves $p$ and $q$ of length $n$ on $M$. It is easy to see that $D_n(M) = O(n)$ for all $M$.
    
\begin{question}
    What are the possible asymptotic behaviors of the sequence $(D_n(M))_n$? In particular, is there a Riemannian manifold $M$ such that $D_n(M)$ is $o(n)$ but not $O(\log(n))$?
\end{question}

\section*{Acknowledgments}
N. Chandgotia is thankful to Lior Silberman for a very fruitful conversation back in 2014 which has helped inspire this line of work and to Brian Marcus for numerous discussions and encouragement. Finally he would like to thank Tom Meyerovitch for drawing our attention to the question of strong irreducibility and smooth analogues of our result (Section~\ref{Section: Continuous analogues}) and Yinon Spinka for giving us more context. We would like to thank Fedor Manin for pointing out connections with discrete homotopy theory and discussions regarding Section~\ref{Section: Continuous analogues}. 

Research of N.~Chandgotia was partially supported by SERB SRG grant and INSA fellowship. Research of P. Oprocha was partially supported by the project  CZ.02.01.01/00/23\_021/0008759 supported by EU funds, through the Operational Programme Johannes Amos Comenius.
P. Oprocha is also grateful to Max Planck Institute for Mathematics in Bonn for its hospitality and financial support.
We thank Kinjal Dey and the referees for the careful reading, pointing out inconsistencies, improvements and typos and especially pointing out that the self-loop can be replaced by $K_4$ in the main construction.
\newpage

\printindex

\bibliographystyle{amsalpha}
\bibliography{biblio.bib}

\newcommand{\etalchar}[1]{$^{#1}$}
\providecommand{\bysame}{\leavevmode\hbox to3em{\hrulefill}\thinspace}
\providecommand{\MR}{\relax\ifhmode\unskip\space\fi MR }
\providecommand{\MRhref}[2]{%
  \href{http://www.ams.org/mathscinet-getitem?mr=#1}{#2}
}
\providecommand{\href}[2]{#2}
\begin{thebibliography}{GHdMO24}

\bibitem[ABC{\etalchar{+}}21]{zbMATH07359187}
Noga Alon, Raimundo Brice{\~n}o, Nishant Chandgotia, Alexander Magazinov, and
  Yinon Spinka, \emph{Mixing properties of colourings of the
  {{\(\mathbb{Z}^d\)}} lattice}, Comb. Probab. Comput. \textbf{30} (2021),
  no.~3, 360--373.

\bibitem[Adi58]{MR103217}
Sergei~I. Adian, \emph{On algorithmic problems in effectively complete classes
  of groups}, Dokl. Akad. Nauk SSSR \textbf{123} (1958), 13--16.

\bibitem[BL05]{zbMATH02207449}
H{\'e}l{\`e}ne Barcelo and Reinhard Laubenbacher, \emph{Perspectives on
  {{\(A\)}}-homotopy theory and its applications}, Discrete Math. \textbf{298}
  (2005), no.~1-3, 39--61.

\bibitem[BnP17]{briceno2017strong}
Raimundo Brice\~{n}o and Ronnie Pavlov, \emph{Strong spatial mixing in
  homomorphism spaces}, SIAM Journal on Discrete Mathematics \textbf{31}
  (2017), no.~3, 2110--2137.

\bibitem[BPS10]{Boyle-Pavlov-Schraudner}
Mike Boyle, Ronnie Pavlov, and Michael Schraudner, \emph{Multidimensional sofic
  shifts without separation and their factors}, Transactions of the American
  Mathematical Society \textbf{362} (2010), 4617--4653.

\bibitem[CM18]{MR3743365}
Nishant Chandgotia and Brian Marcus, \emph{Mixing properties for hom-shifts and
  the distance between walks on associated graphs}, Pacific J. Math.
  \textbf{294} (2018), no.~1, 41--69.

\bibitem[CM21]{zbMATH07433654}
Nishant Chandgotia and Tom Meyerovitch, \emph{Borel subsystems and ergodic
  universality for compact {{\(\mathbb{Z}^d\)}}-systems via specification and
  beyond}, Proc. Lond. Math. Soc. (3) \textbf{123} (2021), no.~3, 231--312.

\bibitem[Coh17]{zbMATH06680456}
David~Bruce Cohen, \emph{The large scale geometry of strongly aperiodic
  subshifts of finite type}, Adv. Math. \textbf{308} (2017), 599--626.

\bibitem[CV25]{Carrasco}
Nicanor Carrasco-Vargas, \emph{On a {R}ice theorem for dynamical properties of
  {SFT}s on groups}, Archiv der Mathematik \textbf{124} (2025), no.~6,
  591--603.

\bibitem[DB04]{Delvenne}
Jean-Charles Delvenne and Vincent~D. Blondel, \emph{Quasi-periodic
  configurations and undecidable dynamics for tilings, infinite words and
  {T}uring machines}, Theoretical Computer Science \textbf{319} (2004),
  no.~1-3, 127--143.

\bibitem[Dob68]{dobrushin1968gibbsian}
Roland~L’vovich Dobrushin, \emph{Gibbsian random fields for lattice systems
  with pairwise interactions}, Functional Analysis and its applications
  \textbf{2} (1968), no.~4, 292--301.

\bibitem[Ell18]{CayleyEllenberg2018}
Jordan~S. Ellenberg, \emph{Chromatic numbers of infinite abelian {C}ayley
  graphs}, MathOverflow, 2018, URL: https://mathoverflow.net/q/304715 (version:
  2019-01-10).

\bibitem[Fri97]{Friedland}
Shmuel Friedland, \emph{On the entropy of {${\mathbb Z}^d$} subshifts of finite
  type}, Linear algebra and its applications \textbf{252} (1997), no.~1-3,
  199--220.

\bibitem[GHdMO24]{GHO23}
Silv{\`e}re Gangloff, Benjamin Hellouin~de Menibus, and Piotr Oprocha,
  \emph{Short-range and long-range order: a transition in block-gluing behavior
  in {H}om shifts}, Journal d'Analyse Math\'ematique (2024), in press.

\bibitem[GJKS25]{gao2018continuous}
Su~Gao, Steve Jackson, Edward Krohne, and Brandon Seward, \emph{Continuous
  combinatorics of abelian group actions}, Mem. Am. Math. Soc., vol. 1573,
  Providence, RI: American Mathematical Society (AMS), 2025 (English).

\bibitem[GP95]{geller-propp}
William Geller and James Propp, \emph{The projective fundamental group of a
  {${\mathbb Z}^2$}-shift}, Ergodic Theory and Dynamical Systems \textbf{15}
  (1995), no.~6, 1091–1118.

\bibitem[GS21]{GS}
Silv{\`e}re Gangloff and Mathieu Sablik, \emph{{Quantified block gluing for
  multidimensional subshifts of finite type: aperiodicity and entropy}},
  {Journal d'Analyse Math{\'e}matique} \textbf{144} (2021), no.~1, 21--118.

\bibitem[Hat02]{Hatcher}
Allen Hatcher, \emph{Algebraic topology}, Cambridge University Press, 2002.

\bibitem[KM25]{zbMATH07969288}
Krzysztof Kapulkin and Udit Mavinkurve, \emph{The fundamental group in discrete
  homotopy theory}, Adv. Appl. Math. \textbf{164} (2025), 56, Id/No 102838.

\bibitem[LS01]{Lyndon}
Roger~C. Lyndon and Paul~E. Schupp, \emph{Combinatorial group theory}, vol.~89,
  Springer, 2001.

\bibitem[NW83]{nowakowski1983vertex}
Richard Nowakowski and Peter Winkler, \emph{Vertex-to-vertex pursuit in a
  graph}, Discrete Mathematics \textbf{43} (1983), no.~2-3, 235--239.

\bibitem[PS15]{PavlovSchraudner}
Ronnie Pavlov and Michael Schraudner, \emph{Entropies realizable by block
  gluing {${\mathbb Z}^d$} shifts of finite type}, Journal d'Analyse
  Mathématique \textbf{126} (2015), 113--174.

\bibitem[PSV23]{Paviet}
L\'{e}o Paviet~Salomon and Pascal Vanier, \emph{Realizing finitely presented
  groups as projective fundamental groups of {SFT}s}, 48th {I}nternational
  {S}ymposium on {M}athematical {F}oundations of {C}omputer {S}cience, LIPIcs.
  Leibniz Int. Proc. Inform., vol. 272, Schloss Dagstuhl. Leibniz-Zent.
  Inform., Wadern, 2023, pp.~Art. No. 75, 15.

\bibitem[QT00]{quas2000shifts}
Anthony~N. Quas and Paul~B. Trow, \emph{Subshifts of multi-dimensional shifts
  of finite type}, Ergodic Theory and Dynamical Systems \textbf{20} (2000),
  no.~3, 859--874.

\bibitem[Rab58]{MR110743}
Michael~O. Rabin, \emph{Recursive unsolvability of group theoretic problems},
  Ann. of Math. (2) \textbf{67} (1958), 172--194.

\bibitem[Rot73]{rotman1973covering}
Joseph Rotman, \emph{Covering complexes with applications to algebra}, The
  Rocky Mountain Journal of Mathematics \textbf{3} (1973), no.~4, 641--674.

\bibitem[Rot13]{rotman2013introduction}
\bysame, \emph{An introduction to algebraic topology}, vol. 119, Springer
  Science \& Business Media, 2013.

\bibitem[Sta83]{MR0695906}
John~R. Stallings, \emph{Topology of finite graphs}, Invent. Math. \textbf{71}
  (1983), no.~3, 551--565.

\bibitem[Wro17]{zbMATH06656178}
Marcin Wrochna, \emph{Square-free graphs are multiplicative}, J. Comb. Theory,
  Ser. B \textbf{122} (2017), 479--507.

\end{thebibliography}

\end{document}